\documentclass[11pt,twoside]{amsart}

\usepackage[english]{babel}
\usepackage[T1]{fontenc}
\usepackage{lmodern}
\usepackage{amsmath,amsthm,amssymb,amsfonts}
\usepackage{hyperref}
\usepackage[capitalise]{cleveref}
\usepackage[left=3cm,right=3cm,top=3cm,bottom=3cm,includeheadfoot]{geometry}
\usepackage{bbm}
\usepackage{xcolor}
\usepackage{enumerate,enumitem}
\usepackage{fancyhdr}
\usepackage{tikz-cd}
\usepackage[abbrev,nobysame]{amsrefs}

\hypersetup{colorlinks=true,
	linkcolor=teal,
	citecolor=orange}

\theoremstyle{plain}
\newtheorem{lem}{Lemma}[section]
\newtheorem{prop}[lem]{Proposition}
\newtheorem{thm}[lem]{Theorem}
\newtheorem{cor}[lem]{Corollary}

\theoremstyle{definition}
\newtheorem{defi}[lem]{Definition}
\newtheorem{exl}[lem]{Example}
\newtheorem{rem}[lem]{Remark}

\theoremstyle{plain}
\newtheorem{thmintro}{Theorem}

\newtheorem*{probintro*}{Problem}
\newtheorem{corintro}{Corollary}
\newtheorem*{corintro*}{Corollary}

\numberwithin{equation}{section}

\crefname{enumi}{}{}
\crefname{equation}{}{}
\crefname{figure}{Figure}{}
\crefname{lem}{Lemma}{Lemmas}
\crefname{thmintro}{Theorem}{Theorems}

\newcommand{\N}{\mathbb{N}}			
\newcommand{\R}{\mathbb{R}}			
\renewcommand{\S}{\mathbb{S}}		
\renewcommand{\H}{\mathbb{H}}		
\newcommand{\Gr}{\mathrm{Gr}}		
\newcommand{\D}{\mathcal{D}}		
\newcommand{\DD}{\mathbb{D}}		
\newcommand{\CC}{\mathcal{C}}		

\newcommand{\BV}{\mathrm{BV}}		

\newcommand{\indf}{\mathbbm{1}}		

\newcommand{\K}{\mathcal{K}}		
\newcommand{\Val}{\mathbf{Val}}		

\DeclareMathOperator{\SO}{SO}		

\newcommand{\abs}[1]{\lvert #1 \rvert}						
\newcommand{\norm}[1]{\lVert #1 \rVert}						
\newcommand{\restr}[2]{\left. #1 \right|_{#2}}				
\newcommand{\pair}[2]{\langle #1,#2 \rangle}				
\DeclareMathOperator{\sign}{sgn}
\DeclareMathOperator{\spt}{spt}
\DeclareMathOperator{\cls}{cls}

\BibSpec{article}{%
	+{}{\PrintAuthors}  		{author}
	+{,}{ \textit}     		{title}
	+{,}{ }             		{journal}
	+{}{ \textbf}       		{volume}
	+{}{ \parenthesize} 		{date}
	+{}{, no. } 		{number}
	+{,}{ }      	      		{conference}
	+{,}{ }      	      		{book}
	+{,}{ }            		{pages}
	+{,}{ }            	 	{note}
	+{,}{ }            	 	{status}
	+{,}{  \texttt } {eprint}
	+{.}{}              {transition}
}

\BibSpec{book}{%
	+{}{\PrintAuthors}  {author}
	+{,}{ \textit}      {title}
	+{,}{ }             {publisher}
	+{,}{ }             {place}
	+{,}{ }             {date}
	+{.}{}              {transition}
}

\BibSpec{incollection}{%
	+{}  {\PrintAuthors}                {author}
	+{,} { \textit}                     {title}
	+{.} { }                            {part}
	+{:} { \textit}                     {subtitle}
	+{,} { \PrintContributions}         {contribution}
	+{,} { \PrintConference}            {conference}
	+{,} { }                            {booktitle}
	+{,} { pp.~}                        {pages}
	+{,}{ }       		{series}
	+{}{ \textbf}       		{volume}
	+{,} { }                            {publisher}
	+{,} { }                 {date}
	+{,} { }                            {status}
	+{,} { \PrintDOI}                   {doi}
	+{,} { available at \eprint}        {eprint}
	+{}  { \parenthesize}               {language}
	+{}  { \PrintTranslation}           {translation}
	+{;} { \PrintReprint}               {reprint}
	+{.} { }                            {note}
	+{.} {}                             {transition}
}

\usepackage{xcolor}
\usepackage[textsize=tiny]{todonotes}

\begin{document}
	\title[Mixed Christoffel--Minkowski problems for bodies of revolution]{Mixed Christoffel--Minkowski problems\\ for bodies of revolution}
	\author{Leo Brauner}
	\address{Institut f\"ur Mathematik \newline%
		\indent Goethe-Universit\"at Frankfurt \newline%
		\indent Robert-Mayer-Str. 10, 60325 Frankfurt am Main, Germany}
	\email{brauner@math.uni-frankfurt.de}
	
	\author{Georg C.\ Hofst\"atter}
	\address{Institut f\"ur diskrete Mathematik und Geometrie \newline%
		\indent Technische Universit\"at Wien \newline%
		\indent Wiedner Hauptstrasse 8-10/1046, 1040 Wien, Austria}
	\email{georg.hofstaetter@tuwien.ac.at}
	
	\author{Oscar Ortega-Moreno}
	\address{Institut f\"ur diskrete Mathematik und Geometrie \newline%
		\indent Technische Universit\"at Wien \newline%
		\indent Wiedner Hauptstrasse 8-10/1047, 1040 Wien, Austria}
	\email{oscar.moreno@tuwien.ac.at}
	
	\begin{abstract}
		The mixed Christoffel--Minkowski problem
		asks for necessary and sufficient conditions for a Borel measure on the Euclidean unit sphere to be the mixed area measure of some convex bodies, one of which, appearing multiple times, is free and the rest are fixed. In the case where all bodies involved are symmetric around a common axis, we provide a complete solution to this problem, without assuming any regularity. In particular, we refine Firey's~\cite{Firey1970} classification of area measures of figures of revolution.
		
		In our argument, we introduce an easy way to transform mixed area measures and mixed volumes involving axially symmetric bodies, and we significantly improve Firey's~\cite{Firey1970a} estimate on the local behavior of area measures. As a secondary result, we obtain a family of Hadwiger type theorems for convex valuations that are invariant under rotations around an axis.
	\end{abstract}
	
	\maketitle
	
	\thispagestyle{empty}

	\section{Introduction}
	\label{sec:intro}
	
	A classic notion in convex and integral geometry is that of \emph{area measures} associated to a convex body, that is, a compact, convex compact set $K\subset \R^n$. These are geometric measures defined on the unit sphere $\S^{n-1}$ encoding local information about the boundary of $K$. Most prominently, the \emph{surface area measure} associated to a body $K$, denoted as $S_{n-1}(K,{}\cdot{})$, assigns to a Borel set $\beta\subseteq \S^{n-1}$ the Hausdorff measure of the boundary points of $K$ facing in the directions of $\beta$. The area measure of order $i\in\{1,\ldots,n-1\}$ can be obtained as a variation from the surface area measure:
	\begin{equation}\label{eq:area_meas_def}
		S_i(K,\beta)
		= \frac{i!}{(n-1)!}\left.\left(\frac{d}{dt}\right)^{\! n-i-1}\right|_{t=0^+} S_{n-1}(K + tB^n,\beta),
	\end{equation}
	where $B^n \subset \R^n$ denotes the Euclidean unit ball. If $K$ is sufficiently smooth, then $S_i(K,{}\cdot{})$ can be expressed in terms of the $i$-th elementary symmetric polynomial of the principal radii of curvature. In particular, $S_1(K,{}\cdot{})$ corresponds to the mean radius of curvature and $S_{n-1}(K,{}\cdot{})$, to the reciprocal Gauss curvature.
	
	A fundamental open problem, dating back to the foundational works of Christoffel (1865) and Minkowski (1897), is to classify the (Borel) measures on the unit sphere arising as $i$-th order area measures of convex bodies. For degree $i=1$, this problem was first investigated by Christoffel and ultimately resolved independently by Berg~\cite{Berg1969} and Firey~\cite{Firey1968}. For degree $i=n-1$, Minkowski~\cites{Minkowski1897,Minkowski1903} and Alexandrov~\cites{Alexandrov1938,Alexandrov1996} showed that a measure on $\S^{n-1}$ arises as the surface area measure of a convex body with non-empty interior if and only if it is not concentrated on any great sphere and its centroid lies at the origin. 

	Let us note that there are many other important geometric measures associated to convex sets, for which similar classification problems have been considered. Emerging from the seminal work by Lutwak~\cites{Lutwak1993,Lutwak1995}, a rich theory has developed around the $L^p$ and dual area measures, most notably cone-volume measures, which have been classified by Böröczky--Lutwak--Yang--Zhang and are closely related to the conjecture logarithmic Brunn--Minkowski inequality (see~\cites{Boeroeczky2012,Boeroeczky2013}). For an extensive exposition, we recommend the upcoming book by Böröczky--Figalli--Ramos~\cite{Boeroeczky2025a}, the monograph by Schneider~\cite{Schneider2014}, as well as the recent survey by Huang--Yang--Zhang~\cite{Huang2025}.

	The so-called \emph{intermediate} Christoffel--Minkowski problem, concerning the cases where $1<i<n-1$, has attracted significant attention, nevertheless it remains widely open. A remarkable breakthrough was achieved by Guan--Ma~\cite{Guan2003} (see also \cite{Bryan2023}), who determined a structurally rich sufficient condition: a centered function $\phi\in C(\S^{n-1})$ is the density of the $i$-th order area measure of some convex body whenever $\phi^{-1/i}$ is the support function of a convex body of class $C^2_+$ (that is, its boundary is a $C^2$ manifold with positive Gauss curvature).
	
	In order to gain some intuition on the general solution to the Christoffel--Minkowski problem, 
	Firey~\cite{Firey1970} considered the case of smooth bodies of revolution. In this case, the corresponding area measure is invariant under rotations about the axis of revolution, also called \emph{zonal}. This assumption reduces the complexity of the problem to one parameter, while still leading to a meaningful geometric insight. Here and in the following, we always take the axis of symmetry to be parallel to the $n$-th coordinate vector $e_n$, and for a zonal function $q$ on $\S^{n-1}$, we denote by $\bar{q}$ the unique function on $[-1,1]$ such that $q(u)=\bar{q}(\pair{e_n}{u})$, $u \in \S^{n-1}$, where $\pair{\cdot}{\cdot}$ denotes the standard inner product. Moreover, we write $\K(\R^n)$ for the set of convex bodies in $\R^n$.
	
	\begin{thm}[{\cite{Firey1970}}]\label{thm:Firey_Christoffel_minkowski}
		Let $1\leq i < n-1$ and $\mu$ be a centered, zonal measure on $\S^{n-1}$ with a strictly positive density $q \in C(\S^{n-1})$.
		Then there exists a body of revolution $K\in\K(\R^n)$ with $\mu=S_i(K,{}\cdot{})$ if and only if for all $t\in (-1,1)$,
		\begin{equation}\label{eq:Firey_Christoffel-Minkowski}
			(1-t^2)^{\frac{n-1}{2}}\bar{q}(t) - (n-i-1)\int_{t}^1 \bar{q}(s)s(1-s^2)^{\frac{n-3}{2}}\, ds > 0.
		\end{equation}
	\end{thm}
	
	Note that here we state the theorem in a form equivalent to that in~\cite{Firey1970}, but using a different parametrization and replacing some of Firey's original conditions with assumptions of positivity and centeredness.
	
	A natural generalization of the Christoffel--Minkowski problem is obtained if the role of the unit ball~$B^n$ in the variational formula~\cref{eq:area_meas_def} is taken by some other convex body~$C$. That is, define the mixed area measure $S_i(K,C;{}\cdot{})$ of two convex bodies $K$ and $C$ as
	\begin{equation}\label{eq:area_meas_anisotropic_def}
		S_i(K,C;\beta)
		= \frac{i!}{(n-1)!}\left.\left(\frac{d}{dt}\right)^{\! n-i-1}\right|_{t=0^+} S_{n-1}(K + tC,\beta),
	\end{equation}
	where $\beta\subseteq \S^{n-1}$ is a Borel set.
	The measure $S_i(K,C;{}\cdot{})$ can be regarded as a somewhat anisotropic version of the $i$-th order area measure associated to $K$. Anisotropic geometric measures are being investigated in various contexts (see, e.g., \cites{Ludwig2014,Ludwig2014a,Boeroeczky2025,Hug2000,Schneider2025,Reilly1976,Scheuer2025,Andrews2021}). This gives rise to the \emph{anisotropic Christoffel--Minkowski problem}:
	
	\begin{probintro*}
		Given $1 \leq i < n-1$ and $C\in \K(\R^n)$, find necessary and sufficient conditions for a Borel measure $\mu$ on $\S^{n-1}$ such that $\mu = S_i(K,C; {}\cdot{})$ for some convex body $K \in \K(\R^n)$.
	\end{probintro*}
	
	This problem has only started to gain attention very recently~\cite{Colesanti2025a}, and little is known about it, even when compared to the isotropic setting. In the same spirit as Firey's work~\cite{Firey1970}, we consider the anisotropic Christoffel--Minkowski problem under the assumption of an axial symmetry.

	\subsection{Main result} In order to state our main result, we introduce some notation to describe a given convex body of revolution $C\in\K(\R^n)$. We define a function $R_C$ on $(-1,1)$ by
	\begin{align*}
		R_C(t) = \sqrt{1-t^2}(\overline{h_C}(t) - t \partial_+ \overline{h_C}(t)), \qquad t \in (-1,1),
	\end{align*}
	where $h_C(u) = \overline{h_C}(\pair{e_n}{u})$ is the support function of $C$ and $\partial_+$ is the right-derivative. Moreover, we denote by $\ell_C$ the maximal length of a vertical segment (that is, parallel to the axis $e_n$) contained in the boundary of $C$. Note that for a.e.\ $u\in\S^{n-1}$, the horizontal cross-section (that is, by a translate of $e_n^\perp$) of $C$ through a boundary point with outer normal $u$ is a disk of radius $R_C(\pair{e_n}{u})$ (see~\cref{lem:geomIntuitRc}). Together with the quantity $\ell_C$, the function $R_C$ captures the profile of $C$, parametrized via the Gauss map. In particular, $R_C$ is a non-negative function that increases on $(-1,0]$ and decreases on $[0,1)$; it is strictly positive on some (maximal) subinterval $(a_-,a_+)\subseteq (-1,1)$, and $a_-<0<a_+$ unless $C$ is a mere vertical segment. The upward and downward facing normal cones of $C$ are given by  $\S^{n-1} \cap N(C, F(C, \pm e_n)) = \mathrm{Cap}(\pm e_n, \pm a_{\pm})$, where throughout
	\begin{equation*}
		\mathrm{Cap}(v,t) 
		= \{ u\in\S^{n-1} : \pair{u}{v} \geq t \}, \quad t \in [0,1],
	\end{equation*}
	denotes the closed spherical cap of radius $\arccos t$ around $v\in\S^{n-1}$. Moreover, to a positive function $R$ on $(a_-,a_+)$ that is continuous at $t=0$, increases on $(a_-,0]$ and decreases on $[0,a_+)$ (or vice versa), we assign a signed Borel measure $\nu_{R}$ on $(a_-,a_+)$. When $R$ is smooth, it is given by $\nu_{R}(dt)=-R'(t)/t\, dt$, and through distributional derivatives, it extends naturally also to non-smooth $R$.
	
	We now state our main result, providing necessary and sufficient conditions for the an\-iso\-tropic Christoffel--Minkowski problem when both bodies involved are bodies of revolution. Here and throughout we write $\bar \mu = (\pair{e_n}{{}\cdot{}})_\ast \mu$ for the pushforward of a measure $\mu$ on $\S^{n-1}$ to the interval $[-1,1]$ under the map $u \mapsto \pair{e_n}{u}$. Moreover, we denote $R(t^\times)=\lim_{\eta\searrow 1}R(\eta t)$.
	
	\begin{thmintro}\label{mthm:Christoffel_Minkowski_zonal}
		Let $1\leq i < n-1$, let $C$ be a convex body of revolution that is not a segment, and let $\mu$ be a centered, zonal Borel measure on $\S^{n-1}$. Then there exists a body of revolution $K\in\K(\R^n)$ that is not a segment such that $\mu=S_i(K,C;{}\cdot{})$ if and only if
		\begin{enumerate}[label=\upshape(\roman*)]
			\item\label{mthm:Christoffel_Minkowski_zonal:supp}
			$\mu$ is supported in $\{u\in\S^{n-1}: a_{-} \leq \pair{e_n}{u} \leq a_{+} \}$,
			\item\label{mthm:Christoffel_Minkowski_zonal:concentr}
			$\mu$ is not concentrated on $\S^{n-1}\cap e_n^\perp$,
			\item\label{mthm:Christoffel_Minkowski_zonal:measPos}
			the signed Radon measure on	$(a_-,a_{+})$ given by
			\begin{equation*}\label{eq:mthmChrMinkMeasPos}
				R_C^{-(n-i-1)}(t^\times) \, \bar{\mu}(dt) + \left(\int_{[t,\sign t]} \!s\, \bar{\mu}(ds)\right) \nu_{R_C^{-(n-i-1)}}(dt),
			\end{equation*}
			is non-negative and finite, 
			\item\label{mthm:Christoffel_Minkowski_zonal:LimEx}
			the limits 
			\begin{align*}\label{eq:mthmChrMinkLimEx}
				\lim_{t \nearrow a_{+}} \frac{\mu(\mathrm{Cap}( e_n,t))}{R_C^{n-i-1}(t)}
				\qquad\text{and}\qquad
				\lim_{t \searrow a_{-}} \frac{\mu(\mathrm{Cap}(- e_n,-t))}{R_C^{n-i-1}(t)}
			\end{align*}
			exist and are finite, and
			\item \label{mthm:Christoffel_Minkowski_zonal:AtEquat}
			$\displaystyle R_C(0)\mu(\S^{n-1}\cap e_n^\perp) - (n-i-1)\ell_C \int_{\S^{n-1}} \max\{0,\pair{e_n}{u}\} \, \mu(du) \geq 0$.
		\end{enumerate}
	\end{thmintro}
	
	Note that condition~\ref*{mthm:Christoffel_Minkowski_zonal:measPos} corresponds to the pointwise inequality~\cref{eq:Firey_Christoffel-Minkowski} in \cref{thm:Firey_Christoffel_minkowski}; if the reference body $C$ is smooth, then $R_C$ and $\nu_{R_C}$ can be expressed in terms of principal radii of curvature of $C$.
	Condition~\ref*{mthm:Christoffel_Minkowski_zonal:concentr} is similar to that in Minkowski's existence theorem, ensuring that $K$ does not degenerate. 
	Moreover, if the face of $K$ in the direction of the north or south pole is a disk (and $a_\pm=\pm 1$), the limits in~\ref*{mthm:Christoffel_Minkowski_zonal:LimEx} describe its size. In general, these limits prescribe the behavior of $\mu$ at the poles, similar to an estimate by Firey~\cite{Firey1970a}, see \cref{mthm:mixedFirey} below.
	The remaining conditions arise from the potential lack of regularity of the boundary of~$C$. If $C$ has an apex, then~\ref*{mthm:Christoffel_Minkowski_zonal:supp} expresses that the support of $S_i(K,C;{}\cdot{})$ is disjoint from the interior of the corresponding normal cone. In that case, the limits in~\ref*{mthm:Christoffel_Minkowski_zonal:LimEx} describe  $S_i(K,C;{}\cdot{})$ at the rim of its support. Finally, if the boundary of $C$ contains some non-trivial vertical segment, condition~\ref*{mthm:Christoffel_Minkowski_zonal:AtEquat} ensures that $\mu$ has sufficient mass at the equator $\S^{n-1}\cap e_n^\perp$ so that indeed $\ell_K\geq 0$.
	
	In the isotropic case, that is, where $C$ is the Euclidean unit ball, we obtain the following corollary that recovers and extends the classification of \cref{thm:Firey_Christoffel_minkowski} by Firey.
	
	\begin{corintro}
		Let $1\leq i <n-1$ and let $\mu$ be a centered, zonal Borel measure on $\S^{n-1}$. Then there exists a body of revolution $K\in\K(\R^n)$ that is not a segment such that $\mu=S_i(K,{}\cdot{})$ if and only if
		\begin{itemize}
			\item
			$\mu$ is not concentrated on $\S^{n-1}\cap e_n^\perp$,
			\item
			the signed Radon measure on	$(-1,1)$ given by
			\begin{equation*}
				(1-t^2)^{-\frac{n-i-1}{2}} \, \bar{\mu}(dt) - (n-i-1) \left(\int_{[t,\sign t]} \!s\, \bar{\mu}(ds)\right) (1-t^2)^{-\frac{n-i+1}{2}}dt,
			\end{equation*}
			is non-negative and finite, and
			\item
			the limits $\lim_{t \to 1} (1-t^2)^{-\frac{n-i-1}{2}}\mu(\mathrm{Cap}( \pm e_n,t)) $ exist and are finite.
		\end{itemize}
	\end{corintro}
	
	We want to emphasize that our approach extends to much more general geometric measures, namely for \emph{mixed area measures}. Indeed, the mixed area measure of a family of convex bodies $K_1,\ldots,K_{n-1}\in\K(\R^n)$, denoted as $S(K_1,\ldots,K_{n-1};{}\cdot{})$, is a Borel measure on the unit sphere, which can be obtained by a variational formula such as \cref{eq:area_meas_anisotropic_def}, involving a multivariate derivative. Perhaps more instructive, they can be obtained from polarization of the surface area measure: For any $K_1,\ldots,K_m\in\K(\R^n)$ and $\lambda_1,\ldots,\lambda_m\geq 0$,
	\begin{equation*}
		S_{n-1}(\lambda_1K_1 + \cdots + \lambda_mK_m,{}\cdot{})
		= \sum_{1\leq i_1,\ldots,i_{n-1}\leq m} \lambda_{i_1}\cdots\lambda_{i_{n-1}}S(K_{i_1},\ldots,K_{i_{n-1}};{}\cdot{}).
	\end{equation*}
	Mixed area measure play a central role in the study of several well-known geometric inequalities, e.g. the Alexandrov--Fenchel inequality, especially with regard to the problem of determining equality cases (see, e.g., \cites{Shenfeld2020,Shenfeld2022,VanHandel2025, Schneider1975c, Schneider1985, Hug2025, Hug2024c}).
	
	Essentially polarizing the conditions, we obtain a version of \cref{mthm:Christoffel_Minkowski_zonal} where the reference body $C$ is replaced by a tuple $\CC = (C_1, \dots, C_{n-i-1})$ of reference bodies of revolution, solving a \emph{mixed Christoffel--Minkowski problem} for the measures $S_i(K, \CC;{}\cdot{}) = S(K^{[i]}, C_1, \dots, C_{n-i-1}; {}\cdot{})$, where $K^{[i]}$ denotes the tuple consisting of $i$ copies of $K$. As the conditions do not change qualitatively, yet become more complex, we omit the statement of the general theorem here and refer the reader to \cref{thm:Christoffel_Minkowski_zonal} in \cref{sec:ChristMinkProb}.
	
	Finally, we would also like to address the question of uniqueness of the body $K\in\K(\R^n)$. It is well known that if the reference bodies are smooth, then the mixed area measure $S_i(K,\CC;{}\cdot{})$ determines the body $K$ up to a translation (regardless of any symmetry assumptions, see \cite{Schneider2014}*{Theorem~8.1.3}). We find that this is still true in the setting of bodies of revolutions as long as none of the reference bodies has any apex; otherwise, an entire patch of the boundary of any solution $K$ may be perturbed, making it highly non-unique (see \cref{rem:ChristoffelMinkowski_unique}).

	\subsection{Ideas of the proof}
	
	The core idea of this article is to relate Christoffel--Minkowski problems with different reference bodies to each other. In this way, we will reduce the problem involving an arbitrary reference body $C$ to the instance where $C$ is the $(n-1)$-dimensional unit disk in $e_n^\perp$, denoted as $\DD=B^n \cap e_n^\perp$. Here, we obtain the following classification, which is somewhat reminiscent of the resolution of the classical Minkowski problem.
	
	\begin{thmintro}\label{mthm:ZonalMinkChristDisk}
		Let $1 \leq i < n-1$ and let $\mu$ be a non-negative, zonal Borel measure on $\S^{n-1}$. Then there exists a body of revolution $K \in \K(\R^n)$ that is not a segment with $\mu = S_i(K, \DD; {{}\cdot{}})$ if and only if $\mu$ is centered and not concentrated on $\S^{n-1}\cap e_n^\perp$.
	\end{thmintro}
	Let us point out that the case $i=1$ of \cref{mthm:ZonalMinkChristDisk} is treated by the authors without any symmetry restrictions on the convex body $K$ in \cite{Brauner2025b}. There, necessary and sufficient conditions are given for $\mu = S_1(K,\DD;{{}\cdot{}})$ whenever $\mu$ has no mass at $\pm e_n$.
	
	Next, we want to transfer the solution for $C = \DD$ to general bodies of revolution. To this end, we construct an integral operator $\widehat{T}_{\CC} : C(-1,1)\to C(-1,1)$ (see \cref{def:TChat}) which we employ to establish the following transformation rule.
	
	\begin{thmintro}\label{mthm:dictionary}
		Let $1\leq i<n-1$ and $\CC=(C_1,\ldots,C_{n-i-1})$ be a family of convex bodies of revolution in $\R^n$, none of which is a vertical segment. Whenever $f,g \in C(\S^{n-1})$ are zonal and $\widehat{T}_{\CC}\bar{f}=\bar{g}$, then
		\begin{equation}\label{eq:mthmdictZonalValContDisk_statement}
			\int_{\S^{n-1}} f \, dS_i(K,\CC;{}\cdot{})
			= \int_{\S^{n-1}} g \, dS_i(K,\DD;{}\cdot{}) \quad \text{ for all $K\in\K(\R^n)$}.
		\end{equation}
	\end{thmintro}
	
	In the proof of \cref{mthm:dictionary}, we interpret each side of \cref{eq:mthmdictZonalValContDisk_statement} as an $\SO(n-1)$-invariant valuation of the body $K \in \K(\R^n)$ (see below for the definition). As such, they were shown in \cite{Brauner2024a} to be already determined by their values on convex bodies in an $(i+1)$-dimensional subspace containing the axis of revolution~$e_n$. Consequently, we need to verify \cref{eq:mthmdictZonalValContDisk_statement} only for such lower-dimensional bodies, in which case we can transform mixed area measures using the \emph{mixed spherical projections} introduced in \cite{Brauner2024} inspired by and generalizing results from \cite{Goodey2011}.
	
	Let us point out that the map $\widehat{T}_{\CC}$, defined on $C(-1,1)$, is injective but in general not surjective. Writing $\D_{\CC}$ for the preimage of $C[-1,1]$ inside $C(-1,1)$, see \cref{def:fctClassSubint}, we analyze $\D_{\CC}$ in \cref{sec:TRinvertability} to extend \cref{mthm:dictionary} to $\bar{f} \in \D_{\CC}$. In this case, the integral over $f$ has to be replaced by a principal value integral, the existence of which is asserted by the following Firey-type estimates. Here, we denote by $\kappa_{j} = V_j(B^j)$ the volume of the unit ball in $\R^j$ and by $K|E$ the orthogonal projection of a body $K \in \K(\R^n)$ onto a linear subspace $E$.
	
	\begin{thmintro}\label{mthm:mixedFirey}
		Let $1\leq i< n-1$, let $C$ be a convex body of revolution which is not a vertical segment, and let $K\in\K(\R^n)$. Then for all $t \in (0,1]$,
		\begin{equation*}
			S_i(K,C; \mathrm{Cap}(\pm e_n,t) )
			\leq \frac{\kappa_{n-i-1}}{\binom{n-1}{i}} V_i(K|e_n^\perp) \frac{R_{C}^{n-i-1}(\pm t) }{t}.
		\end{equation*}
	\end{thmintro}
	Note that \cref{mthm:mixedFirey} extends and improves a result of Firey~\cite{Firey1970a} for intermediate area measures (that is, $C = B^n$). In the proof of \cref{mthm:mixedFirey} we will make use of \cref{mthm:dictionary} (for continuous functions) to estimate the left-hand side by an area measure with $C = \DD$ and give an estimate for this area measure using that $\DD$ is lower-dimensional. As a byproduct, we recover a result of Hug~\cite{Hug1998} about the density of area measures with respect to a suitable Hausdorff measure.
	
	\medskip
	
	The proof of \cref{mthm:Christoffel_Minkowski_zonal} now works as follows: Writing $\mu$ and $\sigma$ for the pushforwards of the measures $S_i(K, \CC; {}\cdot{})$ and $S_i(K, \DD; {}\cdot{})$, respectively, onto $[-1,1]$ by the map $u \mapsto \pair{e_n}{u}$, \cref{mthm:dictionary} implies that
	\begin{align*}
		\int_{[-1,1]} \bar{f} d\mu = \int_{[-1,1]} \widehat{T}_\CC \bar{f} d\sigma, \qquad \text{ for every $\bar{f} \in C[-1,1]$.}
	\end{align*}

	Consequently, $\mu = \widehat{T}_\CC^\ast \sigma$, and in order to solve the mixed Christoffel--Minkowski problem for a given measure $\mu$, we need to answer two questions: First, what are the conditions on $\mu$ to be in the image of $\widehat{T}_\CC^\ast$; and, second, under which conditions does there exist a preimage $\sigma$ satisfying the conditions of \cref{mthm:ZonalMinkChristDisk}? We answer these questions in \cref{sec:ChristMinkProb} using the (partially defined) inverse of $\widehat{T}_\CC$ and the estimates from \cref{mthm:mixedFirey} at the boundaries of the interval.

	\subsection{A family of Hadwiger type theorems}
	
	As a further application of the transformation rule \cref{eq:mthmdictZonalValContDisk_statement}, we obtain a multitude of classification theorems of $\SO(n-1)$ invariant valuations. A \emph{valuation} is a functional $\varphi:\K(\R^n)\to\R$ with the property that
	\begin{equation*}
		\varphi(K)+\varphi(L)
		= \varphi(K\cup L)+\varphi(K\cap L)
		\qquad \text{whenever }K,L,K\cup L\in\K(\R^n).
	\end{equation*}
	Hadwiger's~\cite{Hadwiger1957} famous characterization of the intrinsic volumes as a basis of all continuous and rigid motion invariant valuations laid the foundation of the modern theory of valuations, which has since been central to convex and integral geometry (see, e.g., \cites{Bernig2011,Alesker2001,Faifman2023,Bernig2014,Ludwig2010a,Bernig2024,Wannerer2014,Faifman2025}).
	We denote by $\Val(\R^n)$ the space of continuous and translation invariant valuations, and by $\Val_i(\R^n)$ the subspace of valuations that are \emph{homogeneous} of degree $i\in\{0,\ldots,n\}$ (that is, $\varphi(\lambda K)=\lambda^{i}\varphi(K)$ for all $K\in\K(\R^n)$ and $\lambda\geq 0$). 
	
	In the spirit of Hadwiger's theorem, descriptions of $\SO(n-1)$ invariant valuations in $\Val(\R^n)$, in analogy to zonal measures called \emph{zonal}, have been studied intensively in recent years. The first such result, an integral representation with classical area measures, was obtained by Schuster--Wannerer~\cite{Schuster2018} under additional regularity assumptions, and later extended by Knoerr~\cite{Knoerr2024c}. Recently, the authors of this article~\cite{Brauner2024a} established another integral representation, involving mixed area measures with the disk. Functional versions of Hadwiger's theorem, that is, for valuations on convex functions, were considered, e.g., in \cite{Colesanti2024}.
	
	Via \cref{mthm:dictionary}, we can transfer the description of zonal valuations in \cite{Brauner2024a} to any family $\CC$ of reference bodies of revolution. Below, we denote for a family of bodies of revolution $\CC=(C_1,\ldots,C_{n-i-1})$
	\begin{equation*}
		\S_{\CC} := \bigcap_{j=1}^{k} \big( \S^{n-1} \setminus   N(C_j,F(C_j, \pm e_n)) \big)
		\qquad\text{and}\qquad
		(a_{-},a_{+}):=\{\pair{e_n}{u}:u\in \S_{\CC}\},
	\end{equation*}
	see also \cref{defi:S_C_and_I_C}, as well as,
	\begin{align*}
		U_{\CC, \varepsilon}^\pm
		= \begin{cases}
			\emptyset, & \text{ if $a_{\pm} = \pm 1$ and $\lim_{t \to \pm 1} R_\CC(t) > 0$,}\\
			\mathrm{Cap}(\pm e_n, |a_{\pm}| - \varepsilon), & \text{ else,}
		\end{cases}
	\end{align*}
	and $U_{\CC, \varepsilon} = U_{\CC, \varepsilon}^- \cup U_{\CC, \varepsilon}^+$. Moreover, $\spt \varphi$ denotes the \emph{support} of a valuation $\varphi\in\Val(\R^n)$ (see \cref{sec:bgVal}). 
	We are now ready to state the resulting family of zonal Hadwiger type results.
	\begin{thmintro}\label{mthm:mixedZonalHadwiger}
		Let $1\leq i< n-1$ and $\CC=(C_1,\ldots,C_{n-i-1})$ be a family of convex bodies of revolution none of which is a vertical segment. Then a valuation $\varphi\in\Val_i(\R^n)$ with $\spt \varphi\subseteq \cls \S_{\CC}$ is zonal if and only if there exists a function $f=\bar{f}(\pair{e_n}{{}\cdot{}})\in C(\S_{\CC})$ with $\bar{f}\in \D_{\CC}$ such that
		\begin{equation}\label{eq:mthmMixedHadGen}
			\varphi(K)
			= \lim_{\varepsilon\to 0^+} \int_{\S^{n-1}\setminus U_{\CC,\varepsilon}} f(u) \, dS_i(K,\CC;u),
			\qquad K\in\K(\R^n).
		\end{equation}
		Moreover, $f$ is unique up to the addition of a zonal linear function restricted to $\S_{\CC}$.
	\end{thmintro}
	
	Note that in the cases where $C_j=B^n$ and $C_j=\DD$, we recover \cite{Knoerr2024c}*{Thm.~A} and \cite{Brauner2024a}*{Thm.~A}, respectively. Moreover, the case $i=n-1$ follows from a more general result by McMullen~\cite{McMullen1980}. Let us further note that the authors are optimistic that the methods of this article can also be adapted to obtain (mixed) functional Hadwiger-type theorems as in \cite{Colesanti2024}.
	
	\subsection{Plan of the article}
	
	\cref{sec:preliminaries} contains some background material on convex geometry, valuations, and functions of bounded variation of one variable. In \cref{sec:bodies_of_revol}, we collect some useful facts about bodies of revolution. \cref{sec:mixed_projections} deals with the aforementioned mixed spherical projections. In \cref{sec:transforming}, we discuss the transformation of mixed volumes; in there, we establish \cref{mthm:dictionary}, along with its consequences, \cref{mthm:mixedFirey,mthm:mixedZonalHadwiger}. \cref{sec:ChristMinkProb} is devoted to mixed Christoffel--Minkowski problems; in there, we prove \cref{mthm:ZonalMinkChristDisk}, and deduce from that the main result, \cref{mthm:Christoffel_Minkowski_zonal}, in a more general version.
	
	
	\section{Preliminaries}\label{sec:preliminaries}
	
	In the following, we recall some classical facts from convex geometry, in particular on mixed volumes and mixed area measures, valuation theory and functions of bounded variation that will be needed later on.
	
	\subsection{Convex geometry}
	As a general reference on this section, we refer to the monograph by Schneider~\cite{Schneider2014}. 
	
	First, recall that the coefficients of the homogeneous polynomial
	\begin{align*}
		V(\lambda_1 K_1 + \dots + \lambda_n K_n) = \sum_{j_1,\dots,j_{n}=1}^n \lambda_{j_1}\cdots\lambda_{j_n} V(K_1, \dots, K_n), \qquad \lambda_1, \dots, \lambda_n \geq 0
	\end{align*}
	are called mixed volumes. Using the Riesz representation theorem, the mixed volume functionals $V$ can be localized to give the so-called mixed area measures, determined by
	\begin{equation}\label{eq:MixedVol_Integral}
		V(K_0,K_1,\ldots,K_{n-1})
		= \frac{1}{n} \int_{\S^{n-1}} h_{K_0}(u)\, S(K_1,\ldots,K_{n-1};du),
	\end{equation}
	where $K_0, \dots, K_{n-1} \in \K(\R^n)$. Note that the mixed area measure $S(K_1, \dots, K_{n-1}; {{}\cdot{}})$ is always a centered, non-negative Borel measure. Clearly, it is $1$-homogeneous and translation-invariant in every component and symmetric under permuting the entries. By letting $K_1 = \dots = K_i = K$ and $K_{i+1} = \dots = K_{n-1} = B^n$ for $0\leq i \leq n$ and re-normalizing, we obtain the $i$th intrinsic volume of $K \in \K(\R^n)$,
	\begin{equation}\label{eq:MixedVol_IntrinsicVol}
		V_i(K)
		= \frac{\binom{n}{i}}{\kappa_{n-i}}V(K^{[i]},(B^n)^{[n-i]}).
	\end{equation}
	Here, we denote by $K^{[i]}$ the $i$-tuple $(K, \dots, K)$ with $K$ $i$-times repeated.
	
	For smooth bodies, the mixed area measures are absolutely continuous with respect to the spherical Lebesgue measure. Indeed, if $K$ is of class $C^2_+$, then its support function $h_K$ is a $C^2(\S^{n-1})$ function, and $S_{n-1}(K,du)=\det D^2h_K(u)du$. If all bodies $K_1, \dots, K_{n-1}$ are of class $C^2_+$, then polarizing this formula using the mixed discriminant $D$, the density of $S(K_1, \dots, K_{n-1}; {{}\cdot{}})$ is given by $D(D^2 h_{K_1}, \dots, D^2 h_{K_{n-1}})$.
	
	If one of the bodies is an interval $[-u,u]$ with $u \in \S^{n-1}$, then the mixed volume reduces to a mixed volume in the hyperplane $u^\perp$ (see \cite{Schneider2014}*{Thm.~5.3.1}),
	\begin{equation}\label{eq:MixedVol_Proj}
		V(K_1,\ldots,K_{n-1},[-u,u])
		= \frac{2}{n} V^{u^\perp}\!(K_1|u^\perp,\ldots,K_{n-1}|u^\perp), \quad K_1, \dots, K_{n-1} \in \K(\R^n),
	\end{equation}
	where $K|E$ denotes the orthogonal projection of $K$ onto the subspace $E$. Note that \cref{eq:MixedVol_Proj} and the homogeneity of mixed volumes imply that a mixed volume vanishes if the same interval is repeated more than once.
	
	Moreover, if the bodies $K_1\ldots,K_{n-1}$ are contained in the hyperplane $u^\perp$, then
	\begin{equation}\label{eq:MixedAreaMeas_hyperplane}
		S(K_1,\ldots,K_{n-1};{}\cdot{})
		= V^{u^\perp}\!(K_1,\ldots,K_{n-1})(\delta_u + \delta_{-u}).
	\end{equation}
	
	Next, an important property of mixed area measures is that they are \emph{locally determined}. Indeed, recall that for a body $K \in \K(\R^n)$, its \emph{face} $F(K, u)$ in direction $u \in \S^{n-1}$ is defined as $F(K,u)=\{x\in\R^n: h_K(u)=\pair{x}{u}\}$ and the \emph{reverse spherical image} of $K$ at a Borel set $\beta \subseteq \S^{n-1}$ is defined as
	\begin{equation*}
		\tau(K,\beta)
		= \bigcup_{u\in\beta} F(K,u).
	\end{equation*}
	Then $S(K_1, \dots, K_{n-1}; \beta)$ depends only on $\tau(K_1, \beta), \dots, \tau(K_{n-1}, \beta)$.
	\begin{lem}[{\cite{Schneider2014}*{p.~215}}]\label{lem:MixedAreaMeas_locally_det}
		Let $K_1,K_1',\ldots,K_{n-1},K_{n-1}'\in\K(\R^n)$ and $\beta\subseteq\S^{n-1}$ be a Borel set such that $\tau(K_j,\beta)=\tau(K_j',\beta)$ for all $j\in\{1,\ldots,n-1\}$. Then
		\begin{equation*}
			S(K_1,\ldots,K_{n-1};\beta)
			= S(K_1',\ldots,K_{n-1}';\beta).
		\end{equation*}
	\end{lem}

	\subsection{Valuations}\label{sec:bgVal}
	In this section, we cite important notions and results from valuation theory that we will need in the following. For a more thorough introduction to the topic, we refer to \cite{Schneider2014}*{Sec.~6} as well as to the references given therein and below.
	
	First, recall that the \emph{support} of a continuous, translation-invariant valuation on $\R^n$ was introduced by Knoerr \cite{Knoerr2021}*{Sec.~6} as follows.
	
	\begin{defi}[{\cite{Knoerr2021}}]
		The \emph{support} of a valuation $\varphi\in\Val(\R^n)$, denoted as $\spt\varphi$, is the smallest compact set $S\subseteq \S^{n-1}$ with the following property: Whenever $K,K'\in\K(\R^n)$ and $h_K=h_{K'}$ on an open neighborhood of $S$, then $\varphi(K)=\varphi(K')$.
	\end{defi}
	
	Next, we turn to zonal valuations. A valuation $\varphi: \K(\R^n) \to \R$ is called zonal or $\SO(n-1)$ invariant, if $\varphi(\eta K) = \varphi(K)$ for all $\eta \in \SO(n-1)$ and $K \in \K(\R^n)$. In this article, the main property of a zonal valuation is that it is already completely determined by its values on some subspace containing the axis, as was shown in \cite{Brauner2024a}. In the following, this theorem will be the key tool to prove \cref{mthm:dictionary}. To state it, denote by $\Gr_{k}(\R^n)$ the Grassmanian manifold of $k$-dimensional linear subspaces.
	\begin{thm}[\cite{Brauner2024a}*{Cor.~C}]\label{thm:zonalKlainDetRestr}
		Let $0\leq i\leq n-1$ and $\varphi\in\Val_i(\R^n)$ be zonal. If $\varphi$ vanishes on some subspace $E\in\Gr_{i+1}(\R^n)$ containing $e_n$, then $\varphi= 0$.
	\end{thm}
	
	Using \cref{thm:zonalKlainDetRestr}, the following description of zonal valuations was given in \cite{Brauner2024a}.	
	\begin{thm}[\cite{Brauner2024a}*{Thm.~A}]\label{thm:zonalDiskHadwiger}
		For $1\leq i\leq n-1$, a valuation $\varphi\in\Val_i(\R^n)$ is zonal if and only if there exists a zonal function $g\in C(\S^{n-1})$ such that
		\begin{equation}\label{eq:zonalDiskHadwiger}
			\varphi(K)
			= \int_{\S^{n-1}} g(u)\, dS_i(K,\DD;u),
			\qquad K\in\K(\R^n).
		\end{equation}
		Moreover, $g$ is unique up to the addition of a zonal linear function.
	\end{thm}
	From \cref{thm:zonalDiskHadwiger} together with (an extended version of) \cref{mthm:dictionary} we will later deduce \cref{mthm:mixedZonalHadwiger} in \cref{sec:HadwigerTheorems}.

	\subsection{Functions of locally bounded variation}\label{sec:BV0}
	
	Next, we recall some facts about a special class of functions of bounded variation on $(-1,1)$. For a general reference on functions of bounded variation of a singe variable, we recommend \cite{Kannan1996}*{Chapter~6}. 
	
	\medskip
	
	First, note that if a function $R: (-1,1) \to \R$ is of locally bounded variation, then there exists a signed Radon measure $\mu_R$ on $(-1,1)$ such that
	\begin{equation}\label{eq:BV_defRadon}
		\int_{(-1,1)} \phi'(s)R(s)\, ds
		= - \int_{(-1,1)} \phi(s)\, d\mu_R(s)
		\qquad \text{for all }\phi\in C^1_c(-1,1),
	\end{equation}
	or, equivalently, the distributional derivative of $R$ is a signed Radon measure. On the contrary, if there exists $\mu_R$ such that \eqref{eq:BV_defRadon} holds, then the function $s \mapsto \mu_R((-1,s])$ is a right-continuous function of locally bounded variation which coincides with $R$ up to some additive constant (Lebesgue-)almost everywhere on $(-1,1)$. We will therefore tacitly assume in the following that every function of locally bounded variation possesses some kind of semi-continuity so that the value of the function at every point coincides with one of its one-sided limits.
	
	We will concentrate on those functions $R$ for which $\frac{1}{s}\mu_R(ds)$ is still a signed Radon measure.
	
	\begin{defi}
		We denote by $\BV_0(-1,1)$ the space of all functions $R:(-1,1)\to \R$ of locally bounded variation for which there exists a signed Radon measure $\nu_R$ on $(-1,1)$ such that
		\begin{equation}\label{eq:BV_0_iff}
			\int_{(-1,1)} \phi'(s)R(s)\, ds
			= \int_{(-1,1)} \phi(s)s\, d\nu_R(s)
			\qquad \text{for all }\phi\in C^1_c(-1,1).
		\end{equation}
		In this case, $\mu_R(ds) = -s\, \nu_R(ds)$, and we further assume that $\abs{\nu_R}(\{0\})=0$ so $\nu_R$ is uniquely determined by $R$.
	\end{defi}
	
	\begin{exl}
		Let $R\in C^1(-1,1)$ and suppose that the function $\frac{R'(s)}{s}$ is integrable on a neighborhood around $s=0$. Then $R\in \BV_0(-1,1)$ and integration by parts shows that $d\nu_R(s)=-\frac{R'(s)}{s} ds$.
	\end{exl}
	
	Similar to functions of locally bounded variation, every $R \in \BV_0(-1,1)$ satisfies the following analogues of the fundamental theorem of calculus, $-1 < t_1 < t_2 < 1$,
	\begin{align}\label{eq:fundthmCalcBV1}
		R(t_2^+) - R(t_1^+)
		= -\int_{(t_1,t_2]} x~d\nu_R(x) \quad \text{ and } \quad R(t_2^-) - R(t_1^-)
		= -\int_{[t_1,t_2)} x~d\nu_R(x),
	\end{align}
	where $R(t^\pm) = \lim_{s \to t^\pm} R(s)$ denotes the right- resp. left-sided limit. Let us point out that (semi-continuous) functions of locally bounded variation are continuous up to countably many jumps (of finite height) and one-sided limits exist everywhere.
	
	Applying \eqref{eq:fundthmCalcBV1}, one directly obtains an integration by parts formula for $R \in \BV_0(-1,1)$ and $\phi \in C^1_c(-1,1)$, $-1 < t_1 < t_2 < 1$,
	\begin{equation}\label{eq:BV_0_int_by_parts}
		\int_{(t_1,t_2]} \phi'(t)R(t)\, dt
		= R(t_2^+)\phi(t_2) - R(t_1^+)\phi(t_1) + \int_{(t_1,t_2]} \phi(x)x\, d\nu_R(x),
	\end{equation}
	and similarly for integrals over $[t_1,t_2)$, where $R(t_1^+)$ and $R(t_2^+)$ are replaced by the left-sided limits.
	
	We note some further properties of functions in $\BV_0(-1,1)$ for later reference. The proof is elementary and given in \cref{app:TechnLemmas}.
	\begin{lem}\label{lem:BV0diffCont}
		If $R\in\BV_0(-1,1)$, then $R$ is differentiable at $t=0$ and $R'(0)=0$. In particular, $R$ is continuous at $t=0$, that is, $R(0^+) = R(0^-)$.
	\end{lem}
		
	\begin{lem}\label{BV_0_algebraic}
		Let $R,Q\in\BV_0(-1,1)$ and $\alpha,\beta\in\R$. Then
		\begin{enumerate}[label=\upshape(\roman*)]
			\item \label{BV_0_algebraic:linear_comb}
			$\alpha R + \beta Q \in \BV_0(-1,1)$ and $\nu_{\alpha R+\beta Q}=\alpha\nu_R + \beta \nu_Q$;
			\item \label{BV_0_algebraic:unit}
			$1\in\BV_0(-1,1)$ and $\nu_1=0$;
			\item \label{BV_0_algebraic:product}
			$RQ\in\BV_0(-1,1)$ and
			\begin{align*}
				d\nu_{RQ}(t)
				= Q(t^-) d\nu_R(t) + R(t^+)d\nu_Q(t)
				= Q(t^+) d\nu_R(t) + R(t^-)d\nu_Q(t);
			\end{align*}
			
			\item \label{BV_0_algebraic:inverse}
			if $\inf_I \abs{R} > 0$ for every compact interval $I\subset (-1,1)$, then $\frac{1}{R}\in \BV_0(-1,1)$ and
			\begin{equation*}
				d\nu_{\frac{1}{R} }(t) = - \frac{1}{R(t^-)R(t^+)}\, d\nu_R(t).
			\end{equation*}
		\end{enumerate}
		
	\end{lem}

	\section{Preliminaries on bodies of revolution}\label{sec:bodies_of_revol}
	
	In this section we recall some basic properties of bodies of revolution and properly introduce the notions needed in the formulation of \cref{mthm:Christoffel_Minkowski_zonal} and in the proofs of the other main theorems. To this end, throughout, let $C \in \K(\R^n)$ be a convex body of revolution, which is not a vertical segment, that is, a segment parallel to the axis $e_n$. We write $h_C$ for the support function of $C$ on $\S^{n-1}$ and note that there exists a function $\overline{h_C} \in C[-1,1]$, such that $h_C = \overline{h_C}(\pair{e_n}{\cdot})$.
	
	First, observe that $\overline{h_C}$ is left- and right-differentiable at every point $t \in (-1,1)$. To this end, set $u_t = t e_n + \sqrt{1-t^2}\bar{u}$, where $\bar{u} \in \S^{n-2}(e_n^\perp)$ is chosen arbitrarily. Clearly, $\overline{h_C}(t) = h_C(u_t)$, $t \in [-1,1]$.
	\begin{lem}\label{lem:oneSidDerHc}
		Let $C \in \K(\R^n)$ be a convex body of revolution. Then the following limits exist for all $t_0 \in (-1,1)$ and satisfy
		\begin{align*}
			\partial_\pm \overline{h_C}(t) := \lim\limits_{t' \to t^\pm} \frac{\overline{h_C}(t') - \overline{h_C}(t)}{t'-t} = \pm h_{F(C, u_t)}\left(\pm \frac{d}{dt}u_t\right),
		\end{align*}
		where $\frac{d}{dt}u_t = e_n - \frac{t}{\sqrt{1-t^2}}\bar{u}$.
	\end{lem}
	\begin{proof}
		Since $\overline{h_C}(t) = h_C(u_t)$, and by \cite{Schneider2014}*{Thm.~1.7.2}, the directional derivatives of $h_C$ exist at $u_t \in \S^{n-1}$ and are equal to $h_{F(C,u_t)}$, the claim follows from the chain rule for directional derivatives.
	\end{proof}
	
	\begin{defi}\label{def:defRC}
		Let $C \in \K(\R^n)$ be a convex body of revolution. We define
		\begin{align}
			R_C(t) = \sqrt{1-t^2}\left(\overline{h_C}(t) - t \partial_+ \overline{h_C}(t)\right), \qquad t \in (-1,1).
		\end{align}
	\end{defi}
	
	Note that, by \cref{lem:oneSidDerHc}, $R_C$ is well-defined. The definition of $R_C$ is motivated by the following geometric intuition.
	
	\begin{lem}\label{lem:geomIntuitRc}
		Suppose that $C \in \K(\R^n)$ is a body of revolution. Then for $u \in \S^{n-1}\setminus\{\pm e_n\}$ and $t = \pair{e_n}{u}$, we have 
		\begin{align*}
			C \cap (x_0 + e_n^\perp) = R_C(t) \DD + \pair{e_n}{x}e_n,
		\end{align*}
		for some $x \in F(C, u)$. Moreover, the faces $F(C, \pm e_n)$ are (possibly degenerate) disks of radius
		$\lim_{t \to \pm 1} R_C(t)$, respectively.
	\end{lem}
	\begin{proof}
		First, let $t \in (-1,1)$. By \cref{lem:oneSidDerHc}, and choosing $x\in F(C,u_t)$ appropriately, we obtain
		\begin{align*}
			\partial_+\overline{h_C}(t) = h_{F(C, u_t)}\left(\frac{d}{dt}u_t\right) = \left\langle x, e_n - \frac{t}{\sqrt{1-t^2}}\bar{u}\right\rangle.
		\end{align*}
		Consequently, since $h_C(u_t) = \pair{x}{u_t}$ as $x \in F(C, u_t)$,
		\begin{align*}
			R_C(t) &= \sqrt{1-t^2}(\overline{h_C}(t) - t\partial_- \overline{h_C}(t))\\
			&=\sqrt{1-t^2}h_C(u_t)-t\sqrt{1-t^2}\pair{x}{e_n} + t^2 \pair{x}{\bar u}\\
			&=\sqrt{1-t^2}\pair{x}{u_t}-t\sqrt{1-t^2}\pair{x}{e_n} + t^2 \pair{x}{\bar u}\\
			&=(1-t^2)\pair{x}{\bar u} + t^2\pair{x}{\bar u} = \pair{x}{\bar u},
		\end{align*}
		which yields the first claim. The second claim follows either from the first claim or directly by a short computation from \cref{lem:oneSidDerHc}.
	\end{proof}
	
	\begin{exl}
		We have $R_{B^n}(t) = \sqrt{1-t^2}$ and $R_{\DD}(t) = 1$.
	\end{exl}
	In the following, we will extensively use that $R_C$ is an element of the subspace $\BV_0(-1,1)$ of the space of functions of locally bounded variation on $(-1,1)$, see \cref{sec:BV0} for more details.%
	
	To show that $R_C \in \BV_0(-1,1)$, we first consider the case where $C$ is of class $C^\infty_+$, and, hence, $h_C \in C^\infty(\S^{n-1})$. In this case, we obtain the following expression for the surface area measure of $C$ in terms of $\nu_{R_C^{n-1}}$. Here and throughout, we denote by $\omega_i$ the area of the sphere $\S^{i-1}$ in $\R^i$, $\omega_{i} = i \kappa_i$.
	\begin{lem}\label{lem:SurfAreaMeasRevSmooth}
		Let $C \in \K(\R^n)$ be a body of revolution of class $C^\infty_+$. Then $R_C \in \BV_0(-1,1)$ and for all $f \in C[-1,1]$,
		\begin{align}\label{eq:SurfAreaMeasRevSmooth}
			\int_{\S^{n-1}} \!\! f(\pair{u}{e_n}) dS_{n-1}(C, u) = \frac{\omega_{n-1}}{n-1}\int_{(-1,1)} \!\! f(t) d \nu_{R_C^{n-1}}(t).
		\end{align}
	\end{lem}
	\begin{proof}
		It was shown in \cite{OrtegaMoreno2021}*{Proof of Lem.~5.3} that the Hessian of $h_C$ has eigenvalues $\mathcal{A}_1^C(t) = \overline{h_C}(t) - t\overline{h_C}'(t)$ (with multiplicity $n-2$) and $\mathcal{A}_2^C(t) = (1-t^2)\overline{h_C}''(t) + \overline{h_C}(t) - t\overline{h_C}'(t)$ (with multiplicity $1$) at every $u \in \S^{n-1}\setminus\{\pm e_n\}$ such that $t = \pair{u}{e_n}$. As the density of $S_{n-1}(C, \cdot)$ is the product of the eigenvalues, we obtain after applying cylindrical coordinates
		\begin{align}\label{eq:prfSurfAreaMeasSmooth}
			\int_{\S^{n-1}} f(\pair{e_n}{u}) dS_{n-1}(C, u) = \omega_{n-1}\int_{-1}^{1} f(t) (1-t^2)^{\frac{n-3}{2}} (\mathcal{A}_1^C(t))^{n-2}\mathcal{A}_2^C(t) dt 
		\end{align}
		for every $f \in C[-1,1]$. Noting that
		\begin{align*}
			R_C(t) = \sqrt{1-t^2}\mathcal{A}_1^C(t),
		\end{align*}
		and since
		\begin{align}\label{eq:RCprimeSm}
			-\frac{R_C'(t)}{t} = (1-t^2)^{-\frac{1}{2}}\mathcal{A}_1^C(t) +\sqrt{1-t^2}\overline{h_C}''(t) = (1-t^2)^{-\frac{1}{2}}\mathcal{A}_2^C(t),
		\end{align}
		we conclude, via integration by parts, that $R_C$ satisfies \eqref{eq:BV_0_iff} for $d \nu_{R_C}(t) = (1-t^2)^{-\frac{1}{2}}\mathcal{A}_2^C(t) dt$, that is, $R_C \in \BV_0(-1,1)$. Moreover, by \eqref{eq:prfSurfAreaMeasSmooth} and since $R_C^{n-1}$ is again in $\BV_0(-1,1)$ by \cref{BV_0_algebraic}\cref{BV_0_algebraic:product}, we deduce the claim
		\begin{align*}
			\int_{\S^{n-1}} f(\pair{e_n}{u}) dS_{n-1}(C, u) = \frac{\omega_{n-1}}{n-1}\int_{(-1,1)} f(t) d\nu_{R_C^{n-1}}(t).
		\end{align*}
	\end{proof}
	For a general body of revolution $C$, we need the following obvious observation, defining $\ell_C$.
	
	\begin{lem}\label{lem:vertical_segment_iff}
		For every body of revolution $C\in\K(\R^n)$, there exists a unique body of revolution $\tilde{C}\in\K(\R^n)$ that does not contain any vertical segment in its boundary and a unique $\ell_C\geq 0$ such that
		\begin{equation*}
			C = \tilde{C} + \ell_C [0,e_n].
		\end{equation*}
	\end{lem}

	The version of \cref{lem:SurfAreaMeasRevSmooth} for general bodies of revolution can now be formulated as follows. 
	\begin{lem}\label{lem:SurfAreaMeasRev}
		Let $C \in \K(\R^n)$ be a body of revolution. Then $R_C \in \BV_0(-1,1)$ and for all $f \in C[-1,1]$,
		\begin{align}\label{eq:SurfAreaMeasRev}
			\int_{\S^{n-1}\setminus\{\pm e_n\}} \!\! f(\pair{u}{e_n}) dS_{n-1}(C, u) = \frac{\omega_{n-1}}{n-1}\int_{(-1,1)} \!\! f(t) d \nu_{R_C^{n-1}}(t) + \omega_{n-1}R_C^{n-2}(0)\ell_C f(0).
		\end{align}
	\end{lem}
	The proof will employ an approximation argument using \cref{lem:SurfAreaMeasRevSmooth}. Note here, that -- assuming $R_C^{n-1} \in \BV_0(-1,1)$, $\nu_{R_C^{n-1}}$ lives on $(-1,1)$ and therefore cannot capture the mass of $S_{n-1}(C,{}\cdot{})$ at the poles. Indeed, in the case of the disk $\DD$, its surface area measure is concentrated on $\{-e_n,+e_n\}$ and $\nu_{R_\DD^{n-1}}=0$. However, let us point out that, by \cref{lem:geomIntuitRc}, for $a \in \{\pm 1\}$,
	\begin{align*}
		S_{n-1}(C, \{a e_n\}) = \kappa_{n-1} \lim_{t \to a}R_C(t)^{n-1}.
	\end{align*}
	\begin{proof}[Proof of \cref{lem:SurfAreaMeasRev}]
		First, let $C$ be a body of revolution that does not contain a vertical segment in its boundary. Suppose that $(C_j)_{j \in \N}$ is a sequence of $C^\infty_+$ bodies of revolution converging to $C$ in the Hausdorff metric and fix a $2$-dimensional linear space $E \subset \R^n$, containing $e_n$. Then $C_j|E \to C|E$, and, thus, $S_1^{E}(C_j|E, \cdot) \rightharpoonup S_1^E(C|E, \cdot)$ by the continuity of the surface area measure. By \cref{lem:SurfAreaMeasRevSmooth} applied in $E$, the measures $\nu_{R_{C_j|E}}$ converge weakly to the pushforward of $\frac{1}{\omega_{1}}S_{1}^E(C|E, \cdot)$ onto $(-1,1)$, denoted $\nu$, as $j \to \infty$. However, as the definition of $R_C$ does not depend on the surrounding dimension and $\overline{h_C} = \overline{h_{C|E}}$, we conclude that $\nu_{R_{C_j}} \rightharpoonup \nu$.
		
		Next, note that, by the Hausdorff convergence of $C_j \to C$, $\overline{h_{C_j}} \to \overline{h_C}$ uniformly on $[-1,1]$. Consequently, $\mathcal{A}_1^{C_j} \to \mathcal{A}_1^C$ and $\mathcal{A}_2^{C_j} \to \mathcal{A}_2^C$ converge in the sense of distributions, where $\mathcal{A}_{1}^C$ and $\mathcal{A}_{2}^C$ are defined as in the proof of \cref{lem:SurfAreaMeasRevSmooth}, but in the distributional sense.
		
		As, by \eqref{eq:RCprimeSm}, $d \nu_{R_{C_j}} (t) = (1-t^2)^{-\frac{1}{2}}\mathcal{A}_2^{C_j}(t) dt$, we conclude that $\nu = (1-t^2)^{-\frac{1}{2}}\mathcal{A}_2^C$, as distributions. In particular, $\mathcal{A}_2^C$ is a non-negative measure on $(-1,1)$.
		
		Using \eqref{eq:fundthmCalcBV1}, the weak convergence thus implies
		\begin{align}\label{eq:prfSurfMeasRevContRj}
			R_{C_j}(t) - R_{C_j}(0) = -\int_{(0,t]} sd\nu_{R_{C_j}}(s) \to  -\int_{(0,t]} sd\nu(s), \qquad j \to \infty,
		\end{align}
		for all $t \in (-1,1)$, where $\nu(\{t\}) = 0$, that is, as $\nu$ is a Radon measure, for all except at most countably many $t \in (-1,1)$. Moreover, \cref{lem:geomIntuitRc} implies that $R_{C_j}(0) = \overline{h_{C_j}}(0)$, and, therefore, $R_{C_j}(0) \to R_C(0)$.
		
		On the other hand, as $\mathcal{A}_1^{C_j} \to \mathcal{A}_1^C$ in the sense of distributions, it follows that $R_{C_j} \to R_C$. Together with \eqref{eq:prfSurfMeasRevContRj}, the uniqueness of the limit, and the convergence of $R_{C_j}(0)$, this implies that
		\begin{align*}
			R_C(t) = R_C(0) - \int_{(0,t]} sd\nu(s),
		\end{align*}
		and $R_{C_j} \to R_C$ pointwise almost everywhere on $(-1,1)$.
		
		Suppose now that $f \in C^1_c(-1,1)$. Then by the weak convergence of $\nu_{R_{C_j}}$, and \eqref{eq:BV_0_iff} for $R_{C_j}$,
		\begin{align}\label{eq:prfSurfMeasRevContLimN2}
			\int_{(-1,1)} f(t) t d\nu(t) &= \lim_{j \to \infty} \int_{(-1,1)} f(t) t d\nu_{R_{C_j}}(t) \nonumber \\
			&= \lim_{j \to \infty} \int_{(-1,1)} f'(t) R_{C_j}(t) dt =  \int_{(-1,1)} f'(t) R_{C}(t) dt,
		\end{align}
		where the last step follows by dominated convergence, using that the limit is bounded and thus integrable by \cref{lem:geomIntuitRc}. We conclude that $R_C \in \BV_0(-1,1)$ since it satisfies property~\eqref{eq:BV_0_iff} for $\nu_{R_{C}} = \nu$. Here, we used that the condition that $C$ does not contain a vertical segment in its boundary implies that $|\nu|(\{0\})=0$.
		
		Returning to $\R^n$, note that, by the continuity of the surface area measure, $S_{n-1}(C_j, \cdot) \rightharpoonup S_{n-1}(C, \cdot)$, and by \cref{lem:SurfAreaMeasRevSmooth}, the measures $\nu_{R_{C_j}^{n-1}}$ converge weakly to the pushforward of $\frac{n-1}{\omega_{n-1}}S_{n-1}(C, \cdot)$ onto $(-1,1)$, denoted $\widehat{\nu}$, as $j \to \infty$. Then, as also $R_{C_j}^{n-1} \to R_C^{n-1}$ pointwise almost everywhere on $(-1,1)$, it follows by repeating the argument from \eqref{eq:prfSurfMeasRevContLimN2} and the uniqueness of the measure $\nu_{R_C^{n-1}}$ that $\widehat{\nu} = \nu_{R_{C}^{n-1}}$. This implies \eqref{eq:SurfAreaMeasRev} in the case where $C$ does not contain a vertical segment in its boundary.
		
		Finally, if $C$ contains a vertical segment in its boundary, then, by \cref{lem:vertical_segment_iff}, $C = \widetilde{C} + \ell_C[0,e_n]$ and the boundary of $\widetilde{C}$ does not contain a vertical segment. Note that $R_C = R_{\widetilde{C}}$ and thus $R_C \in \BV_0(-1,1)$ by the previous case. Moreover, by \eqref{eq:MixedVol_Proj},
		\begin{align*}
			\int_{\S^{n-1}\setminus\{\pm e_n\}} f(\pair{e_n}{u}) dS_{n-2}(\widetilde{C}, [0,e_n]; u) &= \int_{\S^{n-1}} f(\pair{e_n}{u}) dS_{n-2}(\widetilde{C}, [0,e_n]; u)\\
			=\frac{1}{n-1}\int_{\S^{n-2}(e_n^\perp)} f(\pair{e_n}{u}) dS_{n-2}^{e_n^\perp}(\widetilde{C}|e_n^\perp, u)
			&=\frac{\omega_{n-1}}{n-1}R_{C}(0)^{n-2}f(0),
		\end{align*}
		for all $f \in C[-1,1]$, using that $S_{n-2}(\widetilde{C}, [0,e_n]; \cdot)$ is concentrated on $\S^{n-1} \setminus \{\pm e_n\}$ and, by \cref{lem:geomIntuitRc}, $\widetilde{C}|e_n^\perp = \widetilde{C}|e_n^\perp = R_{C}(0) \DD$. We can therefore use the multilinearity of the mixed area measure,
		\begin{align*}
			\int_{\S^{n-1}\setminus\{\pm e_n\}} f(\pair{e_n}{u}) dS_{n-1}(C; u) =& \int_{\S^{n-1}\setminus\{\pm e_n\}} f(\pair{e_n}{u}) dS_{n-1}(\widetilde{C}; u) \\
			&+ (n-1)\int_{\S^{n-1}\setminus\{\pm e_n\}} f(\pair{e_n}{u}) dS_{n-2}(\widetilde{C}, [0,e_n]; u),
		\end{align*}
		to deduce \eqref{eq:SurfAreaMeasRev} in the general setting from the previous case.
	\end{proof}

	Using \cref{lem:SurfAreaMeasRev} we can directly deduce that $\nu_{R_C^{n-1}}$, and, by \cref{BV_0_algebraic}\cref{BV_0_algebraic:product}, $\nu_{R_C}$ is non-negative. By \eqref{eq:fundthmCalcBV1}, this implies that $R_C$ is monotone on the subintervals $(-1,0]$ and $[0,1)$. On the other hand, this can also be obtained from \cref{lem:geomIntuitRc} and Brunn's concavity principle.
	
	\begin{cor}\label{cor:RCinBVPlus}
		Suppose that $C \in \K(\R^n)$ is a body of revolution, which is not a vertical segment. Then $R_C$ is increasing on $(-1, 0]$ and decreasing on $[0,1)$. In particular, $\nu_{R_{C}}$ is a non-negative Borel measure on $(-1,1)$ and there exists an open subinterval $I \subset (-1,1)$, containing $0$, where $R_C$ is strictly positive.
	\end{cor}

	\bigskip

	To extend \cref{lem:SurfAreaMeasRev} to mixed area measures, note that by polarizing $\omega_{j}R_C(0)^{j-1}\ell_C$, we obtain the following definition for $\CC= (C_1, \dots, C_{j})$
	\begin{align}\label{eq:defWC}
		W_{\CC} = W_{C_1, \dots, C_{j}} = \frac{\omega_{j}}{j}\sum_{k=1}^{j} \ell_{C_k} \prod_{i \neq k}R_{C_i}(0).
	\end{align}
	
	Write $R_\CC = \prod_{k=1}^{j} R_{C_j}$ for $\CC = (C_1, \dots, C_j)$. A direct polarization argument then implies
	
	\begin{cor}\label{cor:MixedAreaMeasRev}
		Let $\CC=(C_1, \dots, C_{n-1})$ be a family of bodies of revolution in $\R^n$. Then $R_\CC \in \BV_0(-1,1)$ and for all $f \in C[-1,1]$,
		\begin{align}\label{eq:MixedAreaMeasRev}
			\int_{\S^{n-1}\setminus\{\pm e_n\}} f(\pair{u}{e_n}) \, dS(C_1, \dots, C_{n-1}; u) = \frac{\omega_{n-1}}{n-1}\int_{(-1,1)} f(t)\, d \nu_{R_\CC}(t) + W_{\CC} f(0).
		\end{align} 
	\end{cor}
	Note that, in particular, $S(C_1,\ldots,C_{n-1};\S^{n-2}(e_n^\perp)) = W_{C_1,\ldots,C_{n-1}}$.
	
	\section{Mixed spherical projections and liftings}\label{sec:mixed_projections}

We recall that we want to prove \cref{mthm:dictionary} by showing equality of the right and the left hand side for convex bodies inside a single subspace containing the axis of symmetry, and then invoke the zonal Klain \cref{thm:zonalKlainDetRestr} to deduce that equality must indeed hold for all convex bodies. This requires us to relate mixed area measures of lower-dimensional bodies (with reference bodies being bodies of revolution) to their respective surface area measure relative to the given subspace. This problem was already considered and solved in \cite{Goodey2011} (for balls as reference bodies) and in \cite{Brauner2024} (for reference bodies of class $C^2_+$, not necessarily of revolution) introducing the notion of (mixed) spherical liftings and projections. In this section, we refine the results from \cite{Brauner2024} and make them explicit for our case of reference bodies of revolution.

\medskip

Here and in the following, we let $\Gr_{k}(\R^n)$ denote the grassmanian of $k$-dimensional linear subspaces of $\R^n$, $1 \leq k \leq n-1$. For $E \in \Gr_{k}(\R^n)$, we use the abbreviation $\CC|E=(C_1|E,\ldots,C_{n-k}|E)$, where $C|E$ denotes the orthogonal projection of a body $C$ onto $E$. We denote by $E\vee u$ the linear hull of $E \cup \{u\}$, and by $\H^{n-k}(E,u)=\{v\in\S^{n-k}(E^\perp\vee u):\pair{u}{v}>0 \}$ the relatively open $(n-k)$-dimensional half-sphere, generated by $E^\perp$ and $u$. Note that $\H^{n-k}(E,u)$ is equal to the space of all $v \in \S^{n-1} \setminus E^\perp$ such that the spherical projection $\frac{v |E}{\|v | E\|}$ equals $u$.

\begin{defi}[{\cite{Brauner2024}*{Def.~2.3}}]
	Let $1\leq k < n$ and $E\in\Gr_{k}(\R^n)$. Also, let $\CC=(C_1,\ldots,C_{n-k})$ be a family on convex bodies in $\R^n$. The \emph{$\CC$-mixed spherical projection} is the bounded linear operator $\pi_{E,\CC}: C(\S^{n-1}) \to C(\S^{k-1}(E))$,
	\begin{equation*}
		(\pi_{E,\CC}f)(u)
		= \int_{\H^{n-k}(E,u)} f(v)\, dS^{E^\perp \vee u}(\CC| (E^\perp \vee u),v),
		\qquad u\in\S^{k-1}(E).
	\end{equation*}
	We call its adjoint operator $\pi_{E,\CC}^\ast: \mathcal{M}(\S^{k-1}(E))\to\mathcal{M}(\S^{n-1})$ the \emph{$\CC$-mixed spherical lifting}. That is, for $\mu\in\mathcal{M}(\S^{k-1}(E))$, $\pi_{E, \CC}^\ast \mu$ is the unique measure on $\S^{n-1}$ satisfying 
	\begin{equation*}
		\int_{\S^{n-1}} f \, d(\pi_{E,\CC}^\ast \mu)
		= \int_{\S^{k-1}(E)} \pi_{E,\CC}f \, d\mu,
	\end{equation*} 
	for all $f\in C(\S^{n-1})$.
\end{defi}

Let us point out that for $\CC = (B^n, \dots, B^n)$, the operators defined above coincide with the spherical liftings and projections introduced by Goodey, Kiderlen, and Weil in \cite{Goodey2011}*{Thm.~6.2}.

In \cite{Brauner2024}, the authors of this article express mixed area measures of lower dimensional bodies as the mixed spherical lifting of their surface area measure relative to a subspace.

\begin{thm}[{\cite{Brauner2024}}] \label{thm:mixed_sph_lift_gnrl}
	Let $1\leq i<n-1$ and $E\in \Gr_{i+1}(\R^n)$. Also, let $\CC=(C_1,\ldots,C_{n-i-1})$ be a family of convex bodies in $\R^n$ such that
	\begin{equation}\label{eq:condMixedSphLiftGeneral}
		S^{E^\perp \vee u}(\CC| (E^\perp \vee u); \S^{n-i-1}(E^\perp\vee u)\cap E^\perp ) =0
		\qquad\text{for all } u\in \S^{i}(E).
	\end{equation}
	Then for all $K\in\K(E)$,
	\begin{equation}\label{eq:mixed_sph_lift_gnrl}
		S(K^{[i]},\CC;{}\cdot{})
		= \frac{1}{\binom{n-1}{i}} \pi_{E,\CC}^\ast S_i^E(K,{}\cdot{}).
	\end{equation}
\end{thm}

This statement is effectively contained in the proof of \cite{Brauner2024}*{Thm.~B}, which is obtained by approximation from \cite{Brauner2024}*{Thm.~2.5}.
We now turn to bodies of revolution, for which we derive the following corollary.

\begin{cor}\label{cor:mixed_sph_lift_gnrl_Zonal}
	Let $1\leq i < n-1$ and $E\in\Gr_{i+1}(\R^n)$ be such that $e_n\in E$. Also, let $\CC=(C_1,\ldots,C_{n-i-1})$ be a family of convex bodies of revolution in $\R^n$ that do not contain any vertical segment in their respective boundaries. Then condition \eqref{eq:condMixedSphLiftGeneral} is satisfied, and thus, identity \eqref{eq:mixed_sph_lift_gnrl} holds for all $K\in\K(E)$.
\end{cor}
\begin{proof}
	First, recall that for a convex body $C\in\K(\R^n)$ and a Borel subset $\beta\subseteq\S^{n-1}$, $S_{n-1}(C,\beta)=\mathcal{H}^{n-1}(\tau(C,\beta))$ for the reverse spherical image $\tau(C,\beta)$. Now, take a subspace $E'\in\Gr_k(\R^n)$ and a Borel set $\beta\subseteq \S^{k-1}(E')$. It is easy to see that $F(C|E',v)=F(C,v)|E'$ for every $v\in \S^{k-1}(E')$, and thus, $\tau(C|E',\beta)=\tau(C,\beta)|E'$. Consequently,
	\begin{equation}\label{eq:mixed_sph_lift_gnrl:proof1}
		S^{E'}_{k-1}(C|E',\beta)
		= \mathcal{H}^{k-1}(\tau(C|E',\beta))
		= \mathcal{H}^{k-1}(\tau(C,\beta)|E')
		\leq \mathcal{H}^{k-1}(\tau(C,\beta)),
	\end{equation}
	where the final inequality is due to the fact that the Hausdorff measure decreases under a $1$-Lipschitz map.
	
	Next, we take $C\in\K(\R^n)$ to be a body of revolution that does not contain any vertical segment in its boundary. 
	Note that $C|e_n^{\perp}=r\DD$ for some $r\geq 0$, so for every $u\in \S^{n-2}(e_n^{\perp})$,
	\begin{equation*}
		F(C,u)|e_n^{\perp}
		=F(C|e_n^\perp,u)
		=F(r\DD,u)
		= \{ru\}.
	\end{equation*}
	which shows that the face $F(C,u)$ is contained in the vertical line $ru+\R e_n$. 
	By our assumption on $C$, we must have $F(C,u)=\{ru+t e_n\}$ for some $t\in\R$. Hence, the reverse spherical image	
	$\tau(C,\S^{n-i-2}(E^\perp))$ has the same Hausdorff dimension as $\S^{n-i-2}(E^\perp)$, and thus, its $(n-i-1)$-dimensional Hausdorff measure is zero. As an instance of \cref{eq:mixed_sph_lift_gnrl:proof1}, for all $u\in\S^{i}(E)$,
	\begin{equation*}
		S_{n-i-1}^{E^\perp \vee u}(C| (E^\perp \vee u), \S^{n-i-2}(E^\perp) )
		\leq \mathcal{H}^{n-i-1}(\tau(C,\S^{n-i-2}(E^\perp)))
		=  0,
	\end{equation*}
	so condition \cref{eq:condMixedSphLiftGeneral} is satisfied.
	
	To establish the general case, let $\CC=(C_1,\ldots,C_{n-i-1})$ be a family of convex bodies of revolution in $\R^n$ that do not contain any vertical segment in their respective boundaries. Then the body $C=C_1+\cdots+C_{n-i-1}$, due to \cref{lem:vertical_segment_iff}, also does not contain any vertical segment in its boundary. Note that for every Borel set $\beta\subseteq \S^{n-i-1}(E^\perp\vee u)$,
	\begin{equation*}
		S_{n-i-1}^{E^\perp\vee u}(C|(E^\perp\vee u),\beta)
		= \sum_{j_1,\ldots,j_{n-i-1}=1}^{n-i-1} S^{E^\perp\vee u}((C_{j-1},\ldots,C_{j_{n-i-1}})|(E^\perp\vee u);\beta).
	\end{equation*}
	For $\beta=\S^{n-i-2}(E^\perp)$, by our previous argument, the left hand side vanishes. Since each term on the right hand side is non-negative, condition \cref{eq:condMixedSphLiftGeneral} is satisfied. Invoking \cref{thm:mixed_sph_lift_gnrl} then concludes the proof.
\end{proof}

The mixed spherical projection with reference bodies of revolution, applied to a zonal function, can be expressed as a simple integral transform.

\begin{lem} \label{lem:mixedSphProjZonalGeneral}
	Let $1\leq k< n$ and $E\in\Gr_{n,k}$ be such that $e_n\in E$. Also, let $\CC = (C_1, \dots, C_{n-k}) $ be a family of convex bodies of revolution in $\R^n$ with no vertical segment contained in their respective boundaries. Then for all $f\in C[-1,1]$ and $s\in (-1,1)$,
	\begin{align}\label{eq:mixedSphProjZonalGeneral}
		(\bar\pi_{E,\CC}f)(s) = \frac{\omega_{n-k}}{n-k} \int_{(0,1)} f(st) d\nu_{R_{\CC,s}}(t),
	\end{align}
	where $R_{\CC,s}(t) = (1-t^2)^{\frac{n-k}{2}}(1-s^2t^2)^{-\frac{n-k}{2}}R_{\CC}(st)$.
\end{lem}
\begin{proof}
	First, take some body of revolution $C\in\K(\R^n)$ with no vertical segment in its boundary. by definition, for all $u\in\S^{k-1}(E)$,
	\begin{equation*}
		(\pi_{E,C^{[n-k]}}f(\pair{e_n}{{}\cdot{}}))(u)
		= \int_{\H^{n-k}(E,u)} f(\pair{e_n}{v})\, dS^{E^\perp \vee u}_{n-k-1}(C| (E^\perp \vee u),v).
	\end{equation*}
	Next, note that whenever $u\in\S^{k-1}(E)$ and $v\in\S^{n-k}(E^\perp\vee u)$, then $\pair{e_n}{v}=\pair{e_n}{u}\pair{u}{v}$, and thus,
	\begin{equation*}
		h_{C|(E^\perp\vee u)}(v)
		= h_C(v)
		= \bar{h}_C(\pair{e_n}{v})
		= \bar{h}_C(\pair{e_n}{u}\pair{u}{v}).
	\end{equation*}
	Hence, $C|(E^\perp\vee u)$ is a body of revolution in $E^\perp\vee u$ with axis $u$, and denoting $s=\pair{e_n}{u}$, we have that $\bar{h}_{C|(E^\perp\vee u)}(t)=\bar{h}_C(st)$. Hence, for $t\in (-1,1)$,
	\begin{equation*}
		R_{C|(E^\perp\vee u)}(t)
		= (1-t^2)^{\frac 12}(\bar{h}_C(st)-st\partial_+\bar{h}_C(st))
		= (1-t^2)^{\frac 12}(1-s^2t^2)^{-\frac 12}R_C(st).
	\end{equation*}
	Suppose now that $\CC=(C_1,\ldots,C_{n-k})$ is a family of convex bodies with no vertical segment in their respective boundaries. By polarization, $R_{\CC|(E^\perp\vee u)}=R_{\CC,s}$, and thus, noting that $\H^{n-k}(E^\perp\vee u)\subseteq \S^{n-k}(E^\perp\vee u)\setminus\{\pm u\}$, we obtain \cref{eq:mixedSphProjZonalGeneral} from \cref{cor:MixedAreaMeasRev}.
\end{proof}

	\section{Transforming mixed volumes} \label{sec:transforming}
	
	In this section, we show the transformation rule for mixed volumes in \cref{mthm:dictionary} and apply it to deduce \cref{mthm:mixedZonalHadwiger} from \cref{thm:zonalDiskHadwiger}. The proof will be divided into three steps. First, we show \cref{mthm:dictionary} in the case where all functions are continuous. Then, we will use this to prove \cref{mthm:mixedFirey} by reducing it to the case where all reference bodies are disks. In the last step, finally, we will obtain \cref{mthm:dictionary} by approximation from its continuous version, using \cref{mthm:mixedFirey} to show convergence of the (principal value) integrals.
	
	A main ingredient for these steps will be an integral transform on $C(-1,1)$, denoted $T_R$, which we introduce in the following. Most of the properties of the $T_R$ transform will be proved in a general setting, that is, for general functions $R$ of bounded variation. Later, we will only use them in the instance where $R = R_\CC = R_{C_1} \cdots R_{C_{n-i-1}}$. Some of these proofs are very technical. They will mostly be postponed until \cref{app:TechnLemmas}.

	\subsection[Definition of the T R transform and the semigroup property]{Definition of the \texorpdfstring{$T_R$}{T R} transform and the semigroup property}\label{sec:defTRtrans}
	
	Using functions in $\BV_0(-1,1)$, we define the following family of operators on $C(-1,1)$.
	
	\begin{defi}
		For $R\in\BV_0(-1,1)$ and $f\in C(-1,1)$, we define
		\begin{equation}\label{eq:equivFormTR}
			T_Rf(t)
			= R(0)f(t) - \int_{(0,t]} ( sf(t)-t f(s)) \, d\nu_R(s),
			\qquad t\in (-1,1).
		\end{equation}
	\end{defi}
	Note that for $t<0$, we interpret $\int_{(0,t]}$ in \eqref{eq:equivFormTR} as $-\int_{[t,0)}$. Let us further point out that \eqref{eq:equivFormTR} is not changed by altering $R$ on sets of Lebesgue measure zero (not containing $t=0$). Moreover, the domain of integration can be replaced by $(0,t)$, $[0,t)$ or $[0,t]$ as the integrand vanishes on $s=t$ and $\nu_R(\{0\}) = 0$ for $R \in \BV_0(-1,1)$.
	
	By \eqref{eq:fundthmCalcBV1}, definition \eqref{eq:equivFormTR} is equivalent to
	\begin{equation}\label{eq:defTR}
		T_Rf(t)
		= \begin{cases}
			R(t^+)f(t) + t \int_{(0,t]} f(s) \, d\nu_R(s), & t > 0,\\
			R(t^-)f(t) - t \int_{[t,0)} f(s) \, d\nu_R(s), & t\leq 0,
		\end{cases}
	\end{equation}
	where the latter expressions a priori \emph{do} depend on the values on $R$ on zero sets outside zero. Clearly, $T_1 = \mathrm{id}$, where $1$ denotes the constant one function.

	We collect some properties of the $T_R$ transforms, which are of essential use for us.
	\begin{prop}\label{prop:TRContContGroupProp}
		Let $R \in \BV_0(-1,1)$. Then the following holds
		\begin{enumerate}[label=\upshape(\roman*)]
			\item\label{it:T_R_cont_to_cont} The transform $T_R$ maps the space $C(-1,1)$ into itself.
			\item\label{it:T_R_group_property} For $Q \in \BV_0(-1,1)$,
			 \begin{equation}\label{eq:T_R_group_property}
				T_{RQ}=T_RT_Q =T_QT_R. 
			\end{equation}
			In particular, if $\inf_I \abs{R} > 0$ for every compact interval $I\subset (-1,1)$, then $T_R$ is a bijective operator on $C(-1,1)$ and $T_R^{-1}=T_{\frac{1}{R}}$.
			\item\label{it:continuous_in_D_R} Suppose additionally that the limits $\lim_{t\to\pm 1}R(t)$ exist. Then for every $f \in C[-1,1]$, the limits
			\begin{align*}
				\lim_{t\to\pm 1} R(t)f(t)
				\qquad\text{and}\qquad
				\lim_{t\to \pm 1} \int_{(0,t]} f(t)\, d\nu_R(t)
				\qquad\text{exist},
			\end{align*}
			that is, $T_R f \in C[-1,1]$.
		\end{enumerate}
	\end{prop}
	The proof of \cref{prop:TRContContGroupProp} is given in \cref{app:TechnLemmas}. Next, let us give an example of the $T_R$ function that is needed especially in approximation arguments later on.

	\begin{exl}\label{ex:Tind}
		For $r\in (0,1)$, we consider the indicator function $\indf_{(-1,r)}$. Integration by parts shows that this is a $\BV_0(-1,1)$ function and $\nu_{\indf_{(-1,r)}}=\frac{1}{r}\delta_r$. Similarly, $\indf_{(-r,1)}$ is a $\BV_0(-1,1)$ function and $\nu_{\indf_{(-r,1)}}=\frac{1}{r}\delta_{-r}$. The corresponding integral transforms are given by
		\begin{equation*}
			T_{\indf_{(-1,r)}}f(t)
			= \begin{cases}
				f(t),				&t<r,	\\
				\frac{f(r)}{r}t,	&t\geq r,
			\end{cases}
			\qquad\text{and}\qquad
			T_{\indf_{(-r,1)}}f(t)
			= \begin{cases}
				\frac{f(-r)}{-r}t,	&t\leq -r,\\
				f(t),				&t>-r,
			\end{cases}
		\end{equation*}
		where $f\in C(-1,1)$. That is, these transforms truncate the function $f$ and extend it by a linear function to the respective end point of the interval $(-1,1)$. In particular, the transform $T_{\indf_{(-1,r)}}$ maps $C(-1,1)$ into $C(-1,1]$ and $T_{\indf_{(-r,1)}}$ maps $C(-1,1)$ into $C[-1,1)$.
		
		Noting that $\indf_{(-r,r)} = \indf_{(-1,r)}\indf_{(-r,1)}$, by \cref{prop:TRContContGroupProp}\cref{it:T_R_group_property}, observe that $T_{\indf_{(-r,r)}}$ is given by
		\begin{equation*}
			T_{\indf_{(-r,r)}}f(t)
			= \begin{cases}
				\frac{f(-r)}{-r}t,	&t\leq -r,\\
				f(t),				&-r<t<r, \\
				\frac{f(r)}{r}t,	&t\geq r,
			\end{cases}
		\end{equation*}
		and, thus, maps into $C[-1,1]$.
	\end{exl}

	\subsection{The continuous transformation rule}
	First, we prove a version of \cref{mthm:dictionary} for continuous functions. To this end, we define the transform $\widehat{T}_{\CC}$ as follows.
	\begin{defi}\label{def:TChat}
		Let $1\leq k< n-1$ and $\CC=(C_1,\ldots,C_{k})$ be a family of convex bodies of revolution in $\R^n$, none of which is a vertical segment. We define $R_{\CC} = \prod_{j=1}^{k}R_{C_j} \in \BV_0(-1,1)$ and
		\begin{equation*}
			\widehat{T}_{\CC} : C(-1,1)\to C(-1,1),\qquad
			\widehat{T}_{\CC}\bar{f}
			= T_{R_{\CC}}\bar{f} + \frac{n-i-1}{2\omega_{n-i-1}} W_{\CC} \bar{f}(0)\abs{{}\cdot{}},			
		\end{equation*}
		where $T_{R_{\CC}}$ is defined in \cref{eq:equivFormTR} and $W_\CC$ is defined in \cref{eq:defWC}.
	\end{defi}
	\noindent The main result of this subsection is \cref{mthm:dictionary}, which is a corollary of the following slightly more general statement.
	\begin{thm}\label{thm:dictZonalVal_cont_disk}
		Let $1\leq i< n-1$ and $\CC=(C_1,\ldots,C_{n-i-1})$ be a family of convex bodies of revolution in $\R^n$, none of which is a vertical segment. If $f\in C(\S^{n-1})$ is zonal, then $g = \bar{g}(\pair{e_n}{\cdot})$ with $\widehat{T}_{\CC}\bar{f}=\bar{g}$ is continuous, and for all $K\in\K(\R^n)$
		\begin{equation}\label{eq:dictZonalValContDisk_statement}
			\int_{\S^{n-1}} f \, dS_i(K,\CC;{}\cdot{})
			= \int_{\S^{n-1}} g \, dS_i(K,\DD;{}\cdot{}).
		\end{equation}
	\end{thm}
	
	We will prove \cref{thm:dictZonalVal_cont_disk} by restricting \eqref{eq:dictZonalValContDisk_statement} to bodies $K$ which are contained in an $(i+1)$-dimensional linear subspace $E$ containing the axis $e_n$ and applying \cref{thm:zonalKlainDetRestr}. The key point in the proof is the following property of $\widehat{T}_{\CC}$.
	\begin{prop}\label{prop:commDiag}
		Let $1\leq i< n-1$ and $\CC=(C_1,\ldots,C_{n-i-1})$ be a family of convex bodies of revolution in $\R^n$, none of which is a vertical segment. If no $C_j$ contains a vertical segment in its boundary, then
		\begin{align}
			\pi_{E,\CC} = \pi_{E,\DD} \circ \widehat{T}_{\CC} \quad \text{ on } C(-1,1).
		\end{align}
	\end{prop}
	\begin{proof}
		We denote for $s \in (0,1)$ and $0 < t < 1/s$ and any tuple $\CC = (C_1, \dots, C_{n-i-1})$
		\begin{align*}
			Q_{\CC}(s,t) = \left(\frac{(1-t^2)_+}{1-(st)^2}\right)^{\frac{n-i-1}{2}} \prod_{j=1}^{n-i-1}R_{C_j}(st),
		\end{align*}
		where $(1-t^2)_+ = \max\{1-t^2, 0\}$ denotes the positive part.
		
		Assume $s > 0$, the claim for $s < 0$ follows analogously. For every $g \in C(-1,1)$, \eqref{eq:BV_0_iff}, \eqref{eq:BV_0_int_by_parts}, and letting $u = st$ implies that for $0<a<b<1$
		\begin{align*}
			-\int_{(a,b]} g(st) d\nu_{Q_{\CC}(s,\cdot)}(t) &= \left.\frac{g(st)}{t}Q_{\CC}(s,t)\right|_{t=a^+}^{b^+} - \int_{(a,b]}\frac{stg'(st)-g(st)}{t^2}Q_{\CC}(s,t) dt\\
			&= s \left(\left.\frac{g(u)}{u}Q_{\CC}\left(s,\frac{u}{s}\right)\right|_{u=(sa)^+}^{(sb)^+} - \int_{(sa,sb]}\frac{ug'(u)-g(u)}{u^2}Q_{\CC}\left(s,\frac{u}{s}\right) du\right)\\
			&=-s\int_{(sa,sb]}g(u) d\nu_{Q_{\CC}\left(s,\frac{(\cdot)}{s}\right)}(u).
		\end{align*}
		By approximation, the same holds true for $a=0$ and $b=1$. Note also that $Q_{\CC}(s, \frac{(\cdot)}{s}) \in \BV_0(-1,1)$. Consequently, by \eqref{eq:defTR},
		\begin{align*}
			\frac{n-i-1}{\omega_{n-i-1}}(\pi_{E,\DD} g)(s) = T_{Q_{\DD^{[n-i-1]}}(s,\frac{(\cdot)}{s})} g(s) - Q_{\DD^{[n-i-1]}}(s,1) g(s)  = T_{Q_{\DD^{[n-i-1]}}(s,\frac{(\cdot)}{s})}g(s),
		\end{align*}
		as $Q_{\DD^{[n-i-1]}}(s,1) = 0$. Observe now that, by \cref{prop:TRContContGroupProp}\cref{it:T_R_group_property} and since by assumption $\widehat{T}_{\CC} = T_{R_\CC}$, $T_{Q_{\DD}(s,\frac{(\cdot)}{s})} \circ T_{R_\CC} = T_{Q_{\DD}(s,\frac{(\cdot)}{s}) R_{\CC}}$. Hence, since
		\begin{align*}
			Q_{\DD}\left(s,\frac{t}{s}\right) R_{\CC} (t) = \left(\frac{(1-(t/s)^2)_+}{(1-t^2)}\right)^{\frac{n-i-1}{2}}R_{\CC}(t) = Q_{\CC}\left(s,\frac{t}{s}\right), \quad t \in (-1,1),
		\end{align*}
		the claim follows by reversing the steps taken above.
	\end{proof}
	
	\medskip
	
	\begin{proof}[Proof of \cref{thm:dictZonalVal_cont_disk}]
		First, note that by \cref{prop:TRContContGroupProp}\cref{it:continuous_in_D_R}, $\bar{g} = \widehat{T}_\CC \bar{f}$ and, thus, $g$ is continuous.
		
		Next, write $C_j = \widetilde{C}_j + \ell_{C_j}[0,e_n]$, $j = 1, \dots, n-i-1$, with $\ell_{C_j} \geq 0$ chosen by \cref{lem:vertical_segment_iff} such that $\widetilde{C}_j$ does not contain a vertical segment in its boundary. Then a similar argument as in the proof of \cref{lem:SurfAreaMeasRev} using the multilinearity of mixed area measures and \cref{eq:MixedVol_Proj} implies that
		\begin{align*}
			\int_{\S^{n-1}} f \, dS_i(K,\CC;{}\cdot{}) = \int_{\S^{n-1}} f \, dS_i(K,\widetilde{\CC};{}\cdot{}) + \frac{n-i-1}{(n-1)\omega_{n-i-1}}W_{\CC} \int_{\S^{n-2}(e_n^\perp)}f(u) dS_i^{e_n^\perp}(K|e_n^\perp, u).
		\end{align*}
		As $f$ is zonal,
		\begin{align*}
			\frac{1}{n-1}\int_{\S^{n-2}(e_n^\perp)}f(u) dS_i^{e_n^\perp}(K|e_n^\perp, u) = \bar{f}(0) V^{e_n^\perp}((K|e_n^\perp)^{[i]}, \DD^{[n-i-1]}),
		\end{align*}
		which, again by \cref{eq:MixedVol_Proj}, can be rewritten as
		\begin{align*}
			&\bar{f}(0) V^{e_n^\perp}((K|e_n^\perp)^{[i]}, \DD^{[n-i-1]}) = \bar{f}(0)\frac{n}{2} V(K^{[i]}, \DD^{[n-i-1]},[-e_n,e_n])\\ &\qquad = \frac{\bar{f}(0)}{2} \int_{\S^{n-1}} |\pair{e_n}{u}| dS_i(K, \DD; u).
		\end{align*}
		Consequently, as $\widehat{T}_{\CC}f = T_{\widetilde{\CC}}f + \frac{n-i-1}{2\omega_{n-i-1}} W_{\CC} f(0) |\cdot|$, it suffices to show the claim only for bodies of revolution which do not contain vertical segments in their respective boundary.
		
		Next, note that both sides of \eqref{eq:dictZonalValContDisk_statement} define continuous, translation-invariant and zonal valuations when seen as functionals in $K \in \K(\R^n)$. By \cref{thm:zonalKlainDetRestr}, it is therefore sufficient to show \eqref{eq:dictZonalValContDisk_statement} for bodies $K \in \K(\R^n)$ that are contained in some fixed $(i+1)$-dimensional subspace $E$ containing the axis $e_n$.
		
		Let $K \in \K(E)$ be arbitrary, but fixed. Then, by \cref{cor:mixed_sph_lift_gnrl_Zonal}
		\begin{align}\label{eq:prfDictZonalCont}
			\int_{\S^{n-1}} f \, dS_i(K,\CC;{}\cdot{}) = \int_{\S^{i}(E)} \pi_{E,\CC} f \, dS_i^E(K,\CC;{}\cdot{}) = \int_{\S^{i}(E)}  (\bar\pi_{E,\CC} \bar{f})(\pair{e_n}{\cdot}) \, dS_i^E(K,\CC;{}\cdot{}).
		\end{align}
		Since, by \cref{prop:commDiag}, $\bar{\pi}_{E, \CC} \bar f = \bar{\pi}_{E, \DD^{[n-i-1]}} \bar{g}$, repeating the steps in \eqref{eq:prfDictZonalCont} for $g$ and $\DD^{[n-i-1]}$ instead of $f$ and $\CC$ yields the claim.		
	\end{proof}

	\subsection{Local behavior of mixed area measures}\label{sec:fireyEst}
	
	As was pointed out in the introduction, the global behavior of area measures of convex bodies is generally not well understood. Regarding their local behavior however, several estimates and descriptions are known that are applied frequently throughout convex geometry (see, e.g, \cite{Schneider2014}*{Section~4.5}). One important result of this kind is due to Firey~\cite{Firey1970a} and concerns the area measure of spherical caps. In the following, we denote by $\mathrm{Cap}(v,t)
	= \{ u\in\S^{n-1} : \pair{u}{v} \geq t \}$
	the closed spherical cap of unit vectors that enclose an angle of at most $\theta\in (0,\frac \pi2)$, $\cos(\theta) = t$, with a given vector $v\in\S^{n-1}$.
	
	\begin{thm}[{\cite{Firey1970a}}]\label{thm:Fireys_cap_estimate}
		For $1\leq i\leq n-1$, there exists a constant $C_{n,i}>0$ depending only on $n$ and $i$ such that for all $K\in\K(\R^n)$, $v\in\S^{n-1}$, and $t \in (0,1)$,
		\begin{equation}\label{eq:Fireys_cap_estimate}
			S_i(K,\mathrm{Cap}(v,t))
			\leq C_{n,i}\frac{(1-t^2)^{\frac{n-i-1}{2}}}{t}(\mathrm{diam}\, K)^{i}.
		\end{equation}
	\end{thm}
	
	For our purposes, we require a modification of this classical theorem for mixed area measures involving reference bodies of revolution and polar caps. Our reasoning will also yield a significant improvement of \cref{eq:Fireys_cap_estimate} that is in some sense optimal. The key observation here is that every spherical cap has an axis of symmetry; that is, we would like to apply the transformation rule of \cref{thm:dictZonalVal_cont_disk} to the characteristic function of a polar cap. This function however is discontinuous, but we can overcome this obstacle by means of the machinery that will be developed later in \cref{sec:adj}. This is the content of the following proposition.

	\begin{prop}\label{prop:dict_spherical_caps}
		Let $1\leq i< n-1$, let $\CC=(C_1,\ldots,C_{n-i-1})$ be a family of convex bodies of revolution, and let $K\in\K(\R^n)$. Then for all $t \in (0,1)$,
		\begin{equation}\label{eq:dict_spherical_caps}
			\int_{\mathrm{Cap}(\pm e_n,t)}\abs{\pair{e_n}{u}}\, S_i(K,\CC;du)
			= R_{\CC}(\pm t) \int_{\mathrm{Cap}(\pm e_n,t)}\abs{\pair{e_n}{u}}\, S_i(K,\DD;du)
		\end{equation}
	\end{prop}
	\begin{proof}
		Since the mixed area measure $S_i(K,\CC;{}\cdot{})$ is locally determined by the convex bodies involved (see \cref{lem:MixedAreaMeas_locally_det}), we may assume that the reference bodies $C_j$ do not contain any vertical segments in their respective boundaries (otherwise, replace them with $\tilde{C}_j$). Denoting by $\mu$ and $\sigma$ the pushforward of $S_i(K,\CC;{}\cdot{})$ and $S_i(K,\DD;{}\cdot{})$ under the map $\pair{e_n}{{}\cdot{}}$, respectively, \cref{thm:dictZonalVal_cont_disk} yields that $\mu=T_{R_{\CC}}^\ast\sigma$. The statement then follows from \cref{lem:T_R_adjoint:I:caps}, which is postponed until \cref{sec:adj}.
	\end{proof}
	
	As a consequence, we obtain the following estimate. In the case where all reference bodies are identical, that is $\CC=C^{[n-i-1]}$, this gives \cref{mthm:mixedFirey}.
	
	\begin{thm}\label{thm:mixedFireyEst}
		Let $1\leq i< n-1$, let $\CC=(C_1,\ldots,C_{n-i-1})$ be a family of convex bodies of revolution, and let $K\in\K(\R^n)$. Then for all $t \in (0,1)$,
		\begin{equation}\label{eq:mixedFrireyEst}
			S(K^{[i]},\CC; \mathrm{Cap}(\pm e_n,t) )
			\leq \frac{\kappa_{n-i-1}}{\binom{n-1}{i}} \frac{R_{\CC}(\pm t) }{t} V_i(K|e_n^\perp).
		\end{equation}
	\end{thm}
	\begin{proof}
		We give suitable estimates on the integral expressions in \cref{eq:dict_spherical_caps}. For the integral expression on the right-hand side, denoting $t_+=\max\{0,t\}$, we have that for every $t \in (0,1)$
		\begin{align}\label{eq:mixedFrireyEst:proof}
			\begin{split}
				&\int_{\mathrm{Cap}(e_n,t)}\abs{\pair{e_n}{u}}\, S_i(K,\DD;du)
				\leq \int_{\S^{n-1}} {\pair{ e_n}{u}}_+ \, S_i(K,\DD;du)
				= nV(K^{[i]},\DD^{n-i-1},[0, e_n]) \\
				&\qquad =  V^{e_n^\perp}((K|e_n^\perp)^{[i]},\DD^{n-i-1})
				= \frac{\kappa_{n-i-1}}{\binom{n-1}{i}} V_i(K|e_n^\perp),
			\end{split}
		\end{align}
		where we used the classical formulas \cref{eq:MixedVol_Integral}, \cref{eq:MixedVol_Proj}, and \cref{eq:MixedVol_IntrinsicVol}. For the integral expression on the left-hand side of \cref{eq:dict_spherical_caps} and $t \in (0,1)$, 
		\begin{equation*}
			\int_{\mathrm{Cap}(e_n,t)} \abs{\pair{e_n}{u}} \, S_i(K,\CC;du)
			\geq t S_i(K,\CC;\mathrm{Cap}(e_n,t)).
		\end{equation*}
		Finally, invoking \cref{prop:dict_spherical_caps} establishes \cref{eq:mixedFrireyEst} for the northern polar cap. The argument for the southern polar cap is analogous.
	\end{proof}
	
	We return to the case of the classical area measures of degree $1\leq i\leq n-1$. That is, the reference bodies are Euclidean balls, which are symmetric around any axis. Since $R_{(B^n)^{[n-i-1]}}(t) = (1-t^2)^{\frac{n-i-1}{2}}$, we obtain the following corollary.
	
	\begin{cor}\label{cor:Fireys_cap_estimate_improved}
		Let $1\leq i< n-1$ and $K\in\K(\R^n)$. Then for all $v\in\S^{n-1}$ and $t \in (0,1)$,
		\begin{equation}\label{eq:Fireys_cap_estimate_improved}
			S_i(K , \mathrm{Cap}(v,t) )
			\leq \frac{\kappa_{n-i-1}}{\binom{n-1}{i}}  \frac{(1-t^2)^{\frac{n-i-1}{2}} }{t} V_i(K|v^\perp).
		\end{equation}
	\end{cor}	
	
	We want to point out that our result improves Firey's estimate \cref{eq:Fireys_cap_estimate} in several ways. First, we make the constant $C_{n,i}$ explicit. Second, we replace the diameter by an intrinsic volume. Third, the estimate \cref{eq:Fireys_cap_estimate_improved} is sharp up to the term $t$. More precisely, if $K=\DD(v^\perp) = B^n \cap v^\perp$, then
	\begin{equation*}
		S_i(\DD(v^\perp) , \mathrm{Cap}(v,t) )
		= \frac{\kappa_{n-i-1}}{\binom{n-1}{i}}  (1-t^2)^{\frac{n-i-1}{2}}  V_i(\DD(v^\perp)).
	\end{equation*}
	One easy way to see this is to observe that in \cref{eq:mixedFrireyEst:proof}, in the proof of \cref{thm:mixedFireyEst}, both inequalities become equality for $K=\DD$. Consequently, for every $K\in\K(\R^n)$ and $v\in\S^{n-1}$,
	\begin{equation}
		\limsup_{t\nearrow 1} \frac{S_i(K , \mathrm{Cap}(v,t) )}{(1-t^2)^{\frac{n-i-1}{2}}}
		\leq \frac{\kappa_{n-i-1}}{\binom{n-1}{i}}  V_i(K|v^\perp),
	\end{equation}
	and equality holds for $K=\DD(v^\perp)$. In this sense, for small spherical caps, estimate \cref{eq:Fireys_cap_estimate_improved} is optimal. In particular, the multiplicative constant depending on $n$ and $i$ can not be further improved.
	
	Now we turn to another local description of area measures. It is well known that the $i$-th order area measure $S_i(K,{}\cdot{})$ of a convex body $K$ is absolutely continuous with respect to the $(n-i-1)$-dimensional Hausdorff measure (see, e.g., \cite{Schneider2014}*{Theorem~4.5.5}). The $(n-i-1)$-dimensional density of $S_i(K,{}\cdot{})$ was determined by Hug in~\cite{Hug1998}*{Theorem~4.3}. 
	
	\begin{thm}[{\cite{Hug1998}}]\label{thm:MixedAreaMeas_LowerDimDensity}
		Let $1\leq i\leq n-1$. Then for every $K\in\K(\R^n)$ and $v\in\S^{n-1}$,
		\begin{equation*}
			\lim_{t \to 1} \frac{S_i(K, \mathrm{Cap}(v,t) )}{\kappa_{n-i-1}(1-t^2)^{\frac{n-i-1}{2}}}
			= \frac{1}{\binom{n-1}{i}}V_i(F(K,v)).
		\end{equation*}
	\end{thm}
	
	This local description of the $i$-th order area measure was later improved and completed by Colesanti and Hug in~\cite{Colesanti2000}*{Theorem~0.1}. As a byproduct of the transformation rule \cref{prop:dict_spherical_caps}, we obtain an alternative proof of \cref{thm:MixedAreaMeas_LowerDimDensity}.

	\begin{proof}[Proof of \cref{thm:MixedAreaMeas_LowerDimDensity}]
		Fix a convex body $K\in\K(\R^n)$ and some $v\in\S^{n-1}$, and $t \in (0,1)$. By applying \cref{prop:dict_spherical_caps} to the case where $\CC=(B^n)^{[n-i-1]}$, we have that 
		\begin{equation*}
			\int_{\mathrm{Cap}(v,t)}\abs{\pair{e_n}{u}}\, S_i(K;du)
			= (1-t^2)^{\frac{n-i-1}{2}} \int_{\mathrm{Cap}(v,t)}\abs{\pair{e_n}{u}}\, S_i(K,\DD(v^\perp);du),
		\end{equation*}
		using the fact that $R_{(B^n)^{[n-i-1]}}(t) = (1-t^2)^{\frac{n-i-1}{2}}$ and that Euclidean balls are symmetric around any axis. Note that for every Radon measure $\mu$ on $\S^{n-1}$,
		\begin{equation*}
			t \, \mu(\mathrm{Cap}(v,t))
			\leq \int_{\mathrm{Cap}(v,t)}\abs{\pair{e_n}{u}}\, \mu(du)
			\leq  \mu(\mathrm{Cap}(v,t))
		\end{equation*}
		Applying this to $S_i(K,{}\cdot{})$ and $S_i(K,\DD,{}\cdot{})$ yields the following estimates.
		\begin{align*}
			S_i(K; \mathrm{Cap}(v,t) )
			&\leq \frac{1}{t}(1-t^2)^{\frac{n-i-1}{2}} S_i(K,\DD(v^\perp); \mathrm{Cap}(v,t)).\\
			S_i(K;\mathrm{Cap}(v,t))
			&\geq t \,(1-t^2)^{\frac{n-i-1}{2}}  S_i(K,\DD(v^\perp);\mathrm{Cap}(v,t)).
		\end{align*}
		Consequently, by dividing both sides by $(1-t^2)^{\frac{n-i-1}{2}}$ and passing to the limit $t \to 1$,
		\begin{equation*}
			\lim_{t \to 1} \frac{S_i(K; \mathrm{Cap}(v,t) )}{(1-t^2)^{\frac{n-i-1}{2}}}
			= \lim_{t \to 1} S_i(K,\DD(v^\perp),\mathrm{Cap}(v,t))
			= S_i(K,\DD(v^\perp);\{v\})
		\end{equation*}
		Mixed area measures are locally determined by the convex bodies involved (see \cref{lem:MixedAreaMeas_locally_det}), so we may replace $K$ in the final expression by any other convex body with the same face in the direction of $v$. Hence, 
		\begin{equation*}
			S_i(K,\DD(v^\perp);\{v\})
			= S_i(F(K,v),\DD(v^\perp);\{v\}).
		\end{equation*}
		Recall that $S_{n-1}(C,{}\cdot{})=V_{n-1}(C)(\delta_v+\delta_{-v})$ for every convex body $C\in\K(v^\perp)$. Hence, by polarizing this identity and by \cref{eq:MixedVol_IntrinsicVol},
		\begin{equation*}
			S_i(C,\DD(v^\perp);{}\cdot{})
			= V^{v^\perp}\!(C^{[i]},\DD(v^\perp)^{[n-i-1]})(\delta_v + \delta_{-v})
			= \frac{\kappa_{n-i-1}}{\binom{n-1}{i}} V_i(C)(\delta_v + \delta_{-v}).
			\end{equation*}
		Combining the above steps, we obtain that
		\begin{align*}
			\lim_{t \to 1} \frac{S_i(K; \mathrm{Cap}(v,t) )}{(1-t^2)^{\frac{n-i-1}{2}}}
			= S_i(F(K,v),\DD(v^\perp),\{v\})
			= \frac{\kappa_{n-i-1}}{\binom{n-1}{i}}V_i(F(K,v)),
		\end{align*}
		which concludes the argument.
	\end{proof}

	\subsection[Invertability of the T R transform and function classes]{Invertability of the \texorpdfstring{$T_R$}{T R} transform and function classes}\label{sec:TRinvertability}
	
	\cref{prop:TRContContGroupProp}\cref{it:T_R_group_property} shows that for suitable $R$ the transform $T_R:C(-1,1)\to C(-1,1)$ is invertible with inverse $T_{\frac{1}{R}}$. In the following, we will discuss the situation when the codomain is constrained to $C[-1,1] \subset C(-1,1)$ and determine the exact domain making the restriction of $T_R$ bijective. 
	
	As it will be needed in our application, we restrict to an open subinterval $I=(a_-,a_+)\subseteq (-1,1)$ with $0\in I$ (that is, $-1\leq a_-<0<a_+\leq 1$). Here, \cref{ex:Tind} shows that $T_R$ cannot be onto $C[-1,1]$ anymore.
	
	For technical reasons and since full generality is not needed in our application, we will restrict ourselves to the following subclass of $\BV_0(-1,1)$.
	
	\begin{defi}\label{def:BVpFull}
		We define $\BV_+(I)$ as the class of functions $R\in \BV_0(-1,1)$ such that $R$ is strictly positive on $I$, $\nu_R$ is finite, and either $\nu_R \geq 0$ or $-\nu_R \geq 0$.
	\end{defi}
	
	Let us point out that $\nu_R \geq 0$ implies that $R$ is monotonically increasing on $(a_-,0]$ and monotonically decreasing on $[0,a_+)$. Consequently, the limits $\lim_{t \to a_\pm}R(t)$ always exist and are non-negative.
	
	\begin{defi}\label{def:fctClassSubint}
		For $R\in\BV_+(I)$, we define $\D_R$ as the space of functions $f\in C(I)$ for which the limits
		\begin{equation*}
			\lim_{t\to a_\pm} R(t)f(t)
			\qquad\text{and}\qquad
			\lim_{t\to a_\pm} \int_{(0,t]} f(t)\, d\nu_R(t)
			\qquad\text{exist}.
		\end{equation*}
	\end{defi}
	
	We show bijectivity first for the full interval $(-1,1)$.
	
	\begin{prop}\label{lem:T_R_bijective}
		Let $R\in\BV_+(-1,1)$. Then the map $T_R: \D_R \to C[-1,1]$ is a bijection.
	\end{prop}
	\begin{proof}
		By \cref{prop:TRContContGroupProp}\cref{it:T_R_group_property}, the map $T_R: C(-1,1)\to C(-1,1)$ is a bijection and $T_R^{-1}=T_{\frac{1}{R}}$. Clearly, by \eqref{eq:defTR} and the definition of $\D_R$, $T_R$ maps the subspace $\D_R$ into $C[-1,1]$, so it only remains to show that $T_{\frac{1}{R}}$ maps $C[-1,1]$ into $\D_R$. 
		
		To this end, let $g \in C[-1,1]$ and write $f = T_{\frac{1}{R}} g$. By \cref{lem:TRreflect} in the appendix, to show $f \in \D_R$, we only need to prove existence of the limits $t \to 1$. Moreover, since $g = T_R f$, by \eqref{eq:defTR},
		\begin{align*}
			-\lim_{t\to 1} \int_{(0,t]} f(t)\, d\nu_R(t) =  \lim_{t\to 1} \frac{1}{t}\left(R(t)f(t) - g(t)\right) = \lim_{t\to 1}\{R(t)f(t)\} - g(1),
		\end{align*}
		and it suffices to consider the limit of $R(t)f(t)$.
		Then for all $0<t_0\leq t<1$,
		\begin{equation*}
			f(t)
			= \frac{g(t)}{R(t^+)} + t\int_{(0,t_0]}g \, d\nu_{\frac{1}{R}} + t\int_{(t_0,t]} g \, d\nu_{\frac{1}{R}}.
		\end{equation*}
		By the mean value theorem for integrals, whenever $0<t_0<t<1$, there exists some $t_1\in (t_0,t)$ such that
		\begin{align*}
			\int_{(t_0,t]} g\, d\nu_{\frac{1}{R}}
			= \int_{(t_0,t]} \frac{g(s)}{s} s \, d\nu_{\frac{1}{R}}(ds)
			= \frac{g(t_1)}{t_1} \int_{(t_0,t]} s \, d\nu_{\frac{1}{R}}(ds)
			= \frac{g(t_1)}{t_1} \left(\frac{1}{R(t_0^+)}-\frac{1}{R(t^+)}\right).
		\end{align*}
		
		In order to show that $f\in\D_R$, we distinguish two cases. In the case when $\lim_{t\to 1}R(t)\neq 0$, we let $\varepsilon>0$ be arbitrary and choose $t_0\in (0,1)$ such that
		\begin{equation*}
			\left| \frac{g(t)}{tR(t^+)} - \frac{g(t_0)}{t_0 R(t_0^+)}\right|<\varepsilon
			\qquad\text{and}\qquad
			\left|\frac{1}{R(t_0^+)}-\frac{1}{R(t^+)}\right|<\varepsilon
			\qquad\text{for all }t\in (t_0,1).
		\end{equation*}
		By the previous step, for every $t\in (t_0,1)$, there exists some $t_1\in (t_0,t)$ such that
		\begin{align*}
			&\frac{f(t)}{t} - \frac{f(t_0)}{t_0}
			= \frac{g(t)}{tR(t^+)} - \frac{g(t_0)}{t_0 R(t_0^+)} + \int_{(t_0,t]} g\, d\nu_{\frac{1}{R}} \\
			&\qquad = \frac{g(t)}{tR(t^+)} - \frac{g(t_0)}{t_0 R(t_0^+)} - \frac{g(t_1)}{t_1} \left(\frac{1}{R(t_0^+)}-\frac{1}{R(t^+)}\right).
		\end{align*}	
		Consequently, we obtain that
		\begin{equation*}
			\left| \frac{f(t)}{t} - \frac{f(t_0)}{t_0} \right|
			\leq \varepsilon  + \frac{\norm{g}_\infty}{t_0} \varepsilon
			\qquad\text{for all }t\in (t_0,1).
		\end{equation*}
		This shows that the limit $\lim_{t\to 1}f(t)$, and thus, the limit $\lim_{t\to 1}R(t)f(t)$, exists.
		
		We now consider the case when $\lim_{t\to 1}R(t)= 0$, we let $\varepsilon>0$ be arbitrary and choose $t_0\in (0,1)$ such that
		\begin{equation*}
			\left| \frac{g(t)}{t} - \frac{g(t_1)}{t_1} \right|	<\varepsilon
			\qquad\text{for all }t,t_1\in (t_0,1).
		\end{equation*}
		By the first step of the proof, for every $t\in (t_0,1)$, there exists some $t_1\in (t_0,1)$ such that
		\begin{align*}
			&\frac{1}{t}R(t^+)f(t)
			= \frac{g(t)}{t} + R(t^+)\int_{(0,t_0]} g \, d\nu_{\frac{1}{R}} + \frac{g(t_1)}{t_1} \left(\frac{R(t^+)}{R(t_0^+)}-1\right) \\
			&\qquad = \left( \frac{g(t)}{t} - \frac{g(t_1)}{t_1} \right) + R(t^+)\left( \frac{1}{R(t_0^+)}\frac{g(t_1)}{t_1} + \int_{(0,t_0]} g \, d\nu_{\frac{1}{R}} \right).
		\end{align*}
		By passing to the limit $t\to 1$, we obtain that $\limsup_{t\to 1} | \frac{1}{t}R(t^+)f(t)  |
		\leq \varepsilon$. Since $\varepsilon>0$ was arbitrary, this shows that $\lim_{t\to 1} R(t)f(t)=0$.
	\end{proof}

	\bigskip
	Next, we consider the case where $R$ is only positive on the open subinterval $0 \in I\subset (-1,1)$. Here, we show that $T_R$ is onto the subspace of functions which are linear outside $I$, giving an analogue statement to \cref{lem:T_R_bijective}.
	
	Note that given a function $f \in C(-1,1)$, the values of $T_Rf$ on the subinterval $I$ only depend on the values of $f$ on $I$. The argument in the proof of \cref{prop:TRContContGroupProp}\cref{it:T_R_cont_to_cont} also shows that the continuity of $T_Rf$ on $I$ is a consequence of the continuity of $f$ on $I$. Hence, the transform $T_R$ naturally descends to a map $\restr{T_R}{I}:C(I)\to C(I)$. This allows us to formally extend the transform $T_R$ to $\D_R$ as follows, writing, by abuse of notation, again $T_R$ for the extension:
	\begin{equation*}
		T_R: \D_R\to C[-1,1], \quad  T_Rf:=\begin{cases}
			\frac{(\restr{T_R}{I}f)(a_-)}{a_-}t,	&t\leq a_-,\\
			(\restr{T_R}{I}f)(t),					&a_-<t<a_+, \\
			\frac{(\restr{T_R}{I}f)(a_+)}{a_+}t,	&t\geq a_+.
		\end{cases}
	\end{equation*}
	Note that, by \cref{ex:Tind} and \cref{prop:TRContContGroupProp}\cref{it:T_R_group_property}, this extension coincides with the original definition of $T_R$ on $\D_R \cap C(-1,1) \supset C[a_{-},a_+]$. Moreover, observe that the image of $T_R$ is always contained in the image of $T_{\indf_I}$. By restricting the codomain of $T_R$ to the image of $T_{\indf_I}$, we obtain bijectivity. See \cref{app:TechnLemmas} for the detailed proof.
	
	\begin{cor}\label{lem:T_R_bijective:I}
		Let $R\in\BV_+(I)$. Then the map $T_R: \D_R \to T_{\indf[I]}(C[-1,1])$ is a bijection.
	\end{cor}
	
	The bijectivity of the perturbed map $\widehat{T}_R$ below is now a simple consequence.
	
	\begin{cor}\label{lem:T_RInclSegm_bijective:I}
		Let $R\in\BV_+(I)$ and $c \in \R$. Then the map $\widehat{T}_R: \D_R \to T_{\indf_I}(C[-1,1])$, defined by
		\begin{align*}
			(\widehat{T}_R f)(t) = (T_R f)(t) + c f(0) |t|, \qquad f \in \D_R, t \in (-1,1),
		\end{align*}
		is a bijection.
	\end{cor}

	Next, we collect some properties of functions in $\D_R$, the proof is given in \cref{app:TechnLemmas}.
	
	\begin{lem}\label{lem:D_R_endpoint_limits}
		Let $R\in\BV_+(I)$, $a \in \{a_{\pm}\}$. 
		
		\begin{enumerate}[label=\upshape(\roman*)]
			\item\label{it:D_R_endpoint_limits_zero} If $\lim_{t\to a} R(t)=0$, then $\lim_{t\to a}R(t)f(t)=0$ for every $f \in \D_R$.
			\item\label{it:D_R_endpoint_limits_nzero} If $\lim_{t\to a} R(t) > 0$, then $\lim_{t\to a}f(t)$ exists for every $f \in \D_R$, that is, $f$ extends by continuity to $a$.
		\end{enumerate}
	\end{lem}
	
	For later reference, we note the following direct consequence of \cref{prop:TRContContGroupProp}\cref{it:continuous_in_D_R} and \cref{lem:T_R_subinterval_dilation}.
	\begin{lem}\label{lem:continuous_in_D_R:I}
		Let $R\in \BV_0(I)$ be of bounded variation such that the limits $\lim_{t\to a_{\pm}} R(t)$ exist. Then $C[a_{-},a_+]\subseteq \D_R$.
	\end{lem}

	\subsection{The general transformation rule}
	In this section, we will combine \cref{thm:dictZonalVal_cont_disk} and the estimates from \cref{sec:fireyEst} (\cref{thm:mixedFireyEst}) to deduce a general transformation rule for mixed volumes with reference bodies of revolution.
	
	To this end, we will deepen the study of the transform $\widehat{T}_{\CC}$ in order to determine its maximal image. As it turns out, the image of $\widehat{T}_{\CC}$ is closely related to the boundary structure of the reference bodies $\CC = (C_1, \dots, C_{n-i-1})$ at the poles $\pm e_n$. Indeed, by \cref{lem:geomIntuitRc}, $R_{C_j}(t) = 0$, $t \in (-1,1)$, whenever the normal cone $N(C_j, F(C_j, \pm e_n))$ at the face in direction $\pm e_n$ contains a unit vector $u \in \S^{n-1}$ with $\pair{\pm e_n}{u} = t$ in its interior. It thus follows from the definition of $\widehat{T}_{\CC}$ that, in this case, $\widehat{T}_{\CC} f$ is linear on a neighborhood of $\pm 1$ for every $f \in C(-1,1)$.
	
	On the contrary, the mixed area measure $S_i(K, \CC; \cdot)$ is supported outside the interior of the normal cone of any $C_j$, $j = 1, \dots, n-i-1$, at the poles $\pm e_n$, so that we can restrict to functions $f$ on a subinterval of $(-1,1)$, allowing also singularities at the boundary of the subinterval. Indeed, the following more general statement holds.
	
	\begin{lem}\label{lem:mixed_area_meas_support}
		Let $K_1,\ldots,K_{n-1}\in\K(\R^n)$ and $x$ be a boundary point of $K_j$, $j \in \{1, \dots, n-1\}$. Then for $n \geq 2$
		\begin{equation*}
			\spt S(K_1,\ldots,K_{n-1};{}\cdot{})
			\subseteq \cls(\S^{n-1}\setminus N(K_j,x)).
		\end{equation*}
	\end{lem}
	\begin{proof}
		By symmetry, we may assume that $j=1$. Next, note that $h_{K_1}=h_{\{x\}}$ on $\S^{n-1}\cap N(K_1,x)$. Hence, as mixed area measures are locally determined (see \cref{lem:MixedAreaMeas_locally_det}), 
		\begin{equation*}
			S(K_1,\ldots,K_{n-1};\S^{n-1}\cap N(K_1,x))
			= S(\{x\},K_2,\ldots,K_{n-1};\S^{n-1}\cap N(K_1,x))
			= 0.
		\end{equation*}
		Consequently, $\spt S(K_1,K_2,\ldots,K_{n-1},{}\cdot{})\cap \mathrm{int}\, N(K_1,x) = \emptyset$, as claimed.
	\end{proof}
	
	Inspired by \cref{lem:mixed_area_meas_support}, we define
	\begin{defi}\label{defi:S_C_and_I_C}
		For a family $\CC=(C_1,\ldots,C_{k})$ of convex bodies of revolution, none of which is a segment, we define
		\begin{equation*}
			\S_{\CC} := \bigcap_{j=1}^{k} \big( \S^{n-1} \setminus   N(C_j,F(C_j, \pm e_n)) \big)
			\qquad\text{and}\qquad
			I_{\CC}:=(a_{\CC,-},a_{\CC,+}):=\{\pair{e_n}{u}:u\in \S_{\CC}\}.
		\end{equation*}
	\end{defi}
	\noindent Let us point out that $-1 \leq a_{\CC, -} < 0 < a_{\CC, +} \leq 1$, that is, $0 \in I_{\CC} \subseteq (-1,1)$. Clearly, as $\CC$ consists of bodies of revolution, $\S_{\CC}$ is invariant under rotations fixing $e_n$, that is, a spherical segment, $\S_{\CC} = \{u \in \S^{n-1}:\, \pair{e_n}{u} \in I_\CC\}$. Moreover, \cref{lem:geomIntuitRc} directly implies that $R_\CC \in \BV_+(I_\CC)$.
	
	In the view of the results from \cite{Brauner2024a}, it is now natural to consider the following subset of continuous function on the open interval $I_\CC$, restricting the behavior at the boundary points of the interval. This definition specializes \cref{def:fctClassSubint} and extends a definition from \cites{Knoerr2024c,Brauner2024a}.
	\begin{defi}
		We define $\D_{\CC}$ as the space of functions $\bar{f}\in C(I_{\CC})$ for which the limits 
		\begin{equation*}
			\lim_{t\to a_{\CC,\pm}}  R_{\CC}(t)\bar{f}(t)
			\qquad\text{and}\qquad
			\lim_{t\to a_{\CC, \pm}} \int_{(0,t]} \bar{f}(t)\, d\nu_{R_{\CC}}(t)
			\qquad\text{exist}.
		\end{equation*} 
	\end{defi}
	As is shown in \cref{sec:TRinvertability}, the transformation $\widehat{T}_{\CC}$ (\cref{def:TChat}) naturally extends to a transform $\widehat{T}_\CC: \D_{\CC} \to C[-1,1]$, which we denote again by $\widehat{T}_{\CC}$, abusing notation.
	
	We are now in the position to prove the general transformation rule.
	\begin{thm}\label{thm:dictionaryGeneral}
		Let $1\leq i< n-1$, $\bar{f}\in\D_{\CC}$, and $f=\bar{f}(\pair{e_n}{{}\cdot{}}) \in C(\S_{\CC})$. Then there exists a zonal valuation $\varphi\in\Val_i(\R^n)$ with $\spt\varphi\subseteq \cls \S_{\CC}$ such that
		\begin{equation}\label{eq:dictGeneralPrinValInt}
			\varphi(K)
			= \lim_{\varepsilon\to 0^+} \int_{\S^{n-1}\setminus U_{\CC,\varepsilon}} f(u) \, dS_i(K,\CC;u),
			\qquad K\in\K(\R^n),
		\end{equation}
		where $U_{\CC, \varepsilon} = U_{\CC, \varepsilon}^- \cup U_{\CC, \varepsilon}^+$ and
		\begin{align*}
			U_{\CC, \varepsilon}^\pm = \begin{cases}
				\emptyset, & \text{ if $a_{\CC, \pm} = \pm 1$ and $\lim_{t \to \pm 1} R_\CC(t) > 0$,}\\
				\mathrm{Cap}(\pm e_n, |a_{\CC, \pm}| - \varepsilon), & \text{ else.}
			\end{cases}
		\end{align*}
		Indeed, $\varphi$ is given by
		\begin{align*}
			\varphi(K) = \int_{\S^{n-1}} g(u) \, dS_i(K, \DD; u), \qquad K \in \K(\R^n),
		\end{align*}
		where $\bar{g}=\widehat{T}_{\CC}\bar{f}$ and $g=\bar{g}(\pair{e_n}{{}\cdot{}})\in C(\S^{n-1})$.
	\end{thm}
	\begin{proof}
		Let $\bar{f} \in \D_\CC$ and write $\bar{g} = \widehat{T}_\CC \bar{f}$. By \cref{lem:T_RInclSegm_bijective:I}, $\bar{g} \in C[-1,1]$, so
		\begin{align*}
			\varphi(K) = \int_{\S^{n-1}} \bar{g}(\pair{e_n}{u}) dS_i(K,\DD; u), \quad K \in \K(\R^n),
		\end{align*}
		well-defines a zonal valuation in $\Val_i(\R^n)$. It remains to show that $\varphi$ satisfies relation \eqref{eq:dictGeneralPrinValInt}.

		To this end, let $\varepsilon > 0$ and define $I_{\CC, \varepsilon} = (a_{\CC, -} + \varepsilon, a_{\CC, +} - \varepsilon)$ and $\bar{f}_\varepsilon = \widehat{T}_{\indf_{I_{\CC, \varepsilon}}} \bar f$, that is, we cut $\bar{f}$ at $a_{\CC, \pm} \mp \varepsilon$ and extend it linearly to $(-1,1)$ (see \cref{ex:Tind}). If $U_{\CC, \varepsilon}^+ \neq \emptyset$, then, for $f_\varepsilon = \bar{f}_\varepsilon(\pair{e_n}{\cdot})$, writing $\S^{n-1}_+ = \{u \in \S^{n-1}:\, \pair{e_n}{u} > 0\}$ for the upper hemisphere,
		\begin{align*}
			\int_{\S^{n-1}_+} f_{\varepsilon} \, dS_i(K,\CC;{}\cdot{})  - \int_{\S^{n-1}_+\setminus U_{\CC,\varepsilon}^+} f \, dS_i(K,\CC;{}\cdot{})
			= \int_{U_{\CC,\varepsilon}^+} f_\varepsilon \, dS_i(K,\CC;{}\cdot{}),
		\end{align*}
		which tends to zero as $\varepsilon \to 0^+$. Indeed, as $\spt S_i(K,\CC;{}\cdot{}) \subseteq \S^{n-1} \setminus \mathrm{int} \, U_{\CC, 0}^+$, and the maximum of $|f_\varepsilon(t)|$ in $[a_{\CC,+}-\varepsilon, a_{\CC,+}]$ is attained at $t = a_{\CC, +}$, we can deduce from \cref{thm:mixedFireyEst} that
		\begin{align*}
			\left|\int_{U_{\CC, \varepsilon}^+} f_\varepsilon \, dS_i(K,\CC;{}\cdot{})\right| \leq |f_\varepsilon(a_{\CC, +})|S_i(K,\CC;U_{\CC, \varepsilon}^+)
			\leq C_{n,i}\frac{|f_\varepsilon(a_{\CC, +})| R_\CC(a_{\CC, +} - \varepsilon)}{a_{\CC, +} - \varepsilon} V_i(K|e_n^\perp).
		\end{align*}
		Inserting $f_\varepsilon(a_{\CC, +}) = \frac{f(a_{\CC,+} - \varepsilon)}{a_{\CC,+} - \varepsilon}a_{\CC, +}$, we see that the right-hand side tends to zero as $\varepsilon \to 0$, since by \cref{lem:D_R_endpoint_limits}\cref{it:D_R_endpoint_limits_zero}, $f(a_{\CC,+} - \varepsilon)R_\CC(a_{\CC, +} - \varepsilon) \to 0$. Here, we used that since $U_{\CC, \varepsilon}^+ \neq \emptyset$, we have $\lim_{t \to a_{\CC, +}}R_\CC(t) = 0$.
		
		If $U_{\CC, \varepsilon}^+ = \emptyset$, then both $f$ and $f_\varepsilon$ are continuous by \cref{lem:D_R_endpoint_limits}\cref{it:D_R_endpoint_limits_nzero} and $f_\varepsilon \to f$ uniformly on $[0,1]$. Consequently, repeating the argument at $- e_n$, we obtain
		\begin{align}\label{eq:prfGenDict1}
			\lim_{\varepsilon \to 0^+} \int_{\S^{n-1}} f_{\varepsilon} \, dS_i(K,\CC;{}\cdot{}) = \lim_{\varepsilon \to 0^+} \int_{\S^{n-1}\setminus U_{\CC,\varepsilon}} f \, dS_i(K,\CC;{}\cdot{}).
		\end{align}
		As $\bar{f}_\varepsilon$ and $\bar{g}_\varepsilon := \widehat{T}_\CC \bar{f}_\varepsilon$ are continuous on $[-1,1]$ (by \cref{lem:T_RInclSegm_bijective:I}), \cref{thm:dictZonalVal_cont_disk} implies for $g_\varepsilon = \bar{g}_\varepsilon(\pair{e_n}{\cdot})$ that
		\begin{align}\label{eq:prfGenDict2}
			\int_{\S^{n-1}} f_{\varepsilon} \, dS_i(K,\CC;{}\cdot{}) = \int_{\S^{n-1}} g_{\varepsilon} \, dS_i(K,\DD;{}\cdot{}), \quad K \in \K(\R^n).
		\end{align}
		Noting finally that, by \cref{prop:TRContContGroupProp}\cref{it:T_R_group_property}, $\bar{g}_\varepsilon = \widehat{T}_\CC (\widehat{T}_{\indf_{I_{\CC, \varepsilon}}} \bar{f}) = \widehat{T}_{\indf_{I_{\CC, \varepsilon}}} \bar g$ and, therefore, $\bar{g}_\varepsilon \to \bar{g}$ uniformly on $[-1,1]$, the claim follows by combining \eqref{eq:prfGenDict1} and \eqref{eq:prfGenDict2}.		
	\end{proof}
	
	\subsection{Hadwiger-type theorems}\label{sec:HadwigerTheorems}
	In this section, we prove the following Hadwiger-type characterization of zonal valuations under restrictions of the support using \cref{thm:dictionaryGeneral}, \cref{thm:zonalDiskHadwiger} and the following characterization of the support of valuations defined by
	\begin{align*}
		\psi_{i,g}(K) = \int_{\S^{n-1}}g \, dS_i(K, \DD; \cdot), \quad K \in \K(\R^n),
	\end{align*}
	whenever $g = \bar{g}(\pair{e_n}{\cdot}) \in C(\S^{n-1})$.
	\begin{lem}\label{lem:support_psi}
		Let $1\leq i\leq n-1$, let $g = \bar{g}(\pair{e_n}{\cdot})\in C(\S^{n-1})$. If $I\subseteq (-1,1)$ is an open interval with $0\in I$, then $\spt \psi_{i,g} \subseteq \cls\{u\in\S^{n-1}: \pair{e_n}{u}\in  I\}$ if and only if $\bar{g}=\widehat{T}_{\indf_I}\bar{g}$.
	\end{lem}
	\begin{proof}
		We may write $I=(a_-,a_+)$, where $-1\leq a_-<0<a_+\leq 1$.
		
		First, assume that $\spt \psi_{i,g} \subseteq \cls\{u\in\S^{n-1}: \pair{e_n}{u}\in  I\}$ and consider
		\begin{align*}
			C_s = \mathrm{conv}\left(\DD \cup \left\{\tfrac{\sqrt{1-s^2}}{s}e_n\right\}\right), \quad s \in [-1,1]\setminus\{0\},
		\end{align*}
		that is, the cone with basis $\DD$ and apex $\frac{\sqrt{1-s^2}}{s}e_n$. In \cite{Brauner2024a}*{Lemma~2.2}, it is shown that
		\begin{align}\label{eq:prfSuppPsiConeEval}
			\psi_{i,g}(C_s)=\kappa_{n-1}\left(\bar{g}(\sign s) + \frac{\bar{g}(s)}{\abs{s}}\right).
		\end{align}
		If $s\in (a_+,1]$, then $\bar{h}_{C_{s}}(t)=\bar{h}_{\DD}(t)$ for all $t\in [-1,s] \supseteq I$, that is, $h_{C_{s}}$ and $h_{\DD}$ coincide on some open neighborhood of $\spt\psi_{i,g}$. Consequently, $\psi_{i,g}(C_{s})=\psi_{i,g}(\DD)$, and by \eqref{eq:prfSuppPsiConeEval}
		\begin{align*}
			\bar{g}(1) + \frac{\bar{g}(s)}{s} = 2 \bar{g}(1), \quad s\in (a_+,1],
		\end{align*}
		which yields $\bar{g}(s) = sg(1)$ on $(a_+, 1]$, and by continuity on $[a_+, 1]$. Repeating the argument for $[-1,a_-)$ implies that $\bar{g}$ is linear outside $I$, which is equivalent to $\bar{g}=\widehat{T}_{\indf_I}\bar{g}$, by \cref{ex:Tind}.
		
		Next, assume that $\bar{g}=\widehat{T}_{\indf_I}\bar{g}$. Let $S_+:=\{u\in\S^{n-1}: \pair{e_n}{u}\in  [-1,a_+]\}$ and take $K,K'\in \K(\R^n)$ such that $h_K$ and $h_{K'}$ coincide on some open neighborhood of $S_+$. Since mixed area measures are locally determined by the respective convex bodies (see \cref{lem:MixedAreaMeas_locally_det}), the signed measure $\nu:=S_i(K,\DD;{}\cdot{})-S_i(K',\DD;{}\cdot{})$ is supported in $\S^{n-1}\setminus S_+$. Hence, as by assumption $\bar{g}(t) = t\bar{g}(1)$ for $t \in [a_+,1]$,
		\begin{equation*}
			\psi_{i,g}(K)-\psi_{i,g}(K')
			= \int_{\S^{n-1}} g(u)\, d\nu(u) \\
			= \int_{\S^{n-1}} \pair{e_n}{u} \bar{g}(1) \, d\nu(u)
			= 0,
		\end{equation*}
		where the last equality is due to the fact that $\nu$ is centered as difference of centered measures. We conclude that $\spt\psi_{i,g}\subseteq S_+$. Repeating the argument for $S_-:=\{u\in\S^{n-1}: \pair{e_n}{u}\in  [a_-,1]\}$ finishes the proof.
	\end{proof}
	
	We are now in position to prove \cref{mthm:mixedZonalHadwiger}, which we repeat for the reader's convenience.
	
	\begin{thm}\label{thm:MixedHadwigerGeneral}
		Let $1\leq i< n-1$ and $\CC=(C_1,\ldots,C_{n-i-1})$ be a family of convex bodies of revolution, none of which is a vertical segment. Then a valuation $\varphi\in\Val_i(\R^n)$ with $\spt \varphi\subseteq \cls \S_{\CC}$ is zonal if and only if there exists a function $f=\bar{f}(\pair{e_n}{{}\cdot{}})\in C(\S_{\CC})$ with $\bar{f}\in \D_{\CC}$ such that
		\begin{equation}\label{eq:thmMixedHadGen}
			\varphi(K)
			= \lim_{\varepsilon\to 0^+} \int_{\S^{n-1}\setminus U_{\CC,\varepsilon}} f(u) \, dS_i(K,\CC;u),
			\qquad K\in\K(\R^n),
		\end{equation}
		where $U_{\CC,\varepsilon}$ is defined as in \cref{thm:dictionaryGeneral}. Moreover, $f$ is unique up to the addition of a zonal linear function restricted to $\S_{\CC}$.
	\end{thm}
	\begin{proof}
		First, suppose that $\bar{f} \in \D_\CC$ is given. Then \cref{thm:dictionaryGeneral} shows that \eqref{eq:thmMixedHadGen} defines a zonal valuation $\varphi \in \Val_i(\R^n)$, which is equal to $\psi_{i,g}$, $g=\bar{g}(\pair{e_n}{\cdot})$ and $\bar{g} = \widehat{T}_\CC \bar{f}$. By \cref{lem:T_RInclSegm_bijective:I}, $\bar{g} = \widehat{T}_{\indf_{I_\CC}} \bar g$, so \cref{lem:support_psi} implies the claim on $\spt \varphi$.
		
		Next, suppose that $\varphi\in\Val_i(\R^n)$ with $\spt \varphi\subseteq \cls \S_{\CC}$ is zonal. By \cref{thm:zonalDiskHadwiger}, there exists $g = \bar{g}(\pair{e_n}{\cdot}) \in C(\S^{n-1})$ such that $\varphi = \psi_{i,g}$. By \cref{lem:support_psi}, $\bar g = T_{I_\CC} \bar g$, that is $\bar g \in T_{I_\CC}(C[-1,1])$. Hence, by \cref{lem:T_RInclSegm_bijective:I}, there exists $\bar f \in \D_\CC$ such that $\bar g = \widehat{T}_\CC\bar f$. \cref{thm:dictionaryGeneral} then implies that the right-hand side of \eqref{eq:thmMixedHadGen} defines a continuous valuation that coincides with $\varphi$.
		
		To show uniqueness of $f$, assume that $\varphi = 0$. Then \cref{thm:zonalDiskHadwiger} implies that $\varphi = \psi_{i,g}$ and $g$ is a zonal linear function restricted to $\S^{n-1}$. As, by \cref{lem:T_RInclSegm_bijective:I}, $\widehat{T}_\CC$ is injective and direct computation shows that it maps linear functions to linear functions, we deduce that $f$ is linear as well.
	\end{proof}

	We conclude the section by giving a geometric intuition for when the principal value integral in \cref{mthm:mixedZonalHadwiger} is actually a proper integral.
	\begin{lem}
		Let $1\leq i< n-1$ and $\CC=(C_1,\ldots,C_{n-i-1})$ be a family of convex bodies of revolution, none of which is a vertical segment. Then $U_{\CC, \varepsilon}^+ = \emptyset$, $\varepsilon > 0$, if and only if each $C_1, \dots, C_{n-i-1}$ has a flat part at $e_n$, that is, $F(C_j, e_n)$ is not a point, $1 \leq j \leq n-i-1$. The analogous statement holds for $U_{\CC, \varepsilon}^-$.
	\end{lem}
	\begin{proof}
		By definition, $U_{\CC, \varepsilon}^+ = \emptyset$ if and only if $a_+ = 1$ and $\lim_{t \to 1}R_\CC(t) > 0$. However, \cref{lem:geomIntuitRc} shows that $\lim_{t \to 1}R_\CC(t) > 0$ if and only if all $F(C_j, e_n)$ are non-degenerate disks for all $j=1, \dots, n-i-1$.
	\end{proof}
	
	\section{Christoffel--Minkowski problems}\label{sec:ChristMinkProb}
	In this section, we prove \cref{mthm:Christoffel_Minkowski_zonal}. To this end, we first prove it for the disk area measures (\cref{mthm:ZonalMinkChristDisk}) in \cref{sec:MCprobDisk} and then apply \cref{mthm:dictionary} to deduce \cref{mthm:Christoffel_Minkowski_zonal} in \cref{sec:MCprobTransfer}. For the last step, we need to consider the adjoint transform of $\widehat{T}_\CC$, see \cref{sec:adj}.	
	
	\subsection{The zonal Christoffel--Minkowski problem for the disk}\label{sec:MCprobDisk}
	Before considering \cref{mthm:ZonalMinkChristDisk}, we introduce some definitions and notation related to zonal measures. Let $1\leq i\leq n-1$, and $E \in \Gr_{i+1}(\R^n)$ with $e_n \in E$, and let $f \in C(\S^{i}(E))$. Define, abusing notation, the function $\bar{f} : [-1,1] \to \mathbb{R}$ by
	\begin{align*}
	\bar{f}(t) =  \int_{\S^{i}(E) \cap H_t} f(w) \, \sigma_t(dw), \qquad t \in (-1,1), \qquad \text{ and } \qquad \bar{f}(\pm 1) = f(\pm e_n),
	\end{align*}
	where $H_t = \{x \in \S^n : \pair{e_n}{x} = t\}$, and $\sigma_t$ denotes the Lebesgue probability measure on $\S^{i}(E) \cap H_t$; that is, the unique probability measure on the $(i-1)$-dimensional sphere $\S^{i}(E) \cap H_t$ that is invariant under rotations fixing the axis $e_n$.
	
	\begin{prop}\label{zonaldisintegration}
		Let $1\leq i\leq n-1$, $E \in \Gr_{i+1}(\R^n)$ and let $\mu$ be a zonal measure on $\S^{i}(E)$. Denote by $\bar{\mu}$ the pushforward of $\mu$ under the map $u \mapsto \langle e_n, u \rangle$. Then
		\begin{equation}\label{eq:zonaldisintegration}
			\int_{\S^{i}(E)} f(u) \, \mu(du) = \int_{[-1,1]} \bar{f}(t) \, \bar{\mu}(dt).
		\end{equation}		
	\end{prop}
	\begin{proof}
		Without loss of generality, we may assume that $E = \R^n$. Since $\mu$ is a zonal measure, we have
		\begin{equation}
			\int_{\S^{n-1}} f(u) \, \mu(du) = \int_{\S^{n-1}} f(\vartheta u) \, \mu(du), \qquad \vartheta \in \SO(n-1).
		\end{equation}
		Hence, averaging over all elements of $\SO(n-1)$ yields
		\begin{equation}
			\int_{\S^{n-1}} f(u) \, \mu(du) = \int_{\S^{n-1}} \breve{f}(u) \, \mu(du),
		\end{equation}
		where
		\[
		\breve{f}(u) := \int_{\SO(n-1)} f(\vartheta u)\, d\vartheta, \qquad u \in \S^{n-1}.
		\]
		Note that $\breve{f}$ is a zonal function. If we show that $\bar{\breve{f}} = \bar{f}$, then we reduce the problem to verifying (\ref{eq:zonaldisintegration}) for zonal functions. To this end, for fixed $t \in (-1,1)$ and any $\vartheta \in \SO(n-1)$, we have
		\begin{equation}\label{eq1:zonal_step}
			\int_{\S^{n-1} \cap H_t} f(\vartheta u)\, \sigma_t(du) = \int_{\S^{n-1} \cap H_t} f(u)\, \sigma_t(du),
		\end{equation}
		since $\sigma_t$ is $\SO(n-1)$-invariant. Clearly, (\ref{eq1:zonal_step}) and Fubini's theorem yield $\bar{\breve{f}} = \bar{f}$.
		
		Now, to verify the statement for zonal functions, assume that $f$ is zonal, and observe that $f(u) = f(v)$ for all $v \in \S^{n-1} \cap H_{\pair{e_n}{u}}$, and so $f(u) = \bar{f}(\pair{e_n}{u})$, since $\bar{f}(\pair{e_n}{u})$ is precisely the average of all these values. Therefore,
		\begin{equation}
			\int_{\S^{n-1}} f(u) \, \mu(du) = \int_{\S^{n-1}} \bar{f}(\pair{e_n}{u}) \, \mu(du) = \int_{[-1,1]} \bar{f}(t) \, \bar{\mu}(dt),
		\end{equation}
		where the last equality follows from the fact that $\bar{\mu}$ is the pushforward of $\mu$ under the map $u \mapsto \pair{e_n}{u}$, and by the change of variables formula.
	\end{proof}

	We now define a measure on the sphere, given a measure on the interval $[-1,1]$, in such a way that it serves as the inverse operation to taking the pushforward.
	
	\begin{defi}
		Let $1\leq i\leq n-1$, and let $\nu$ be a measure on $[-1,1]$, and let $E$ be a subspace of $\R^n$ with $e_n \in E$. Define the \emph{zonal pullback} measure $\nu_E$ on $\S^{i}(E)$ by
		\[
		\int_{\S^{i}(E)} f(u) \, \nu_E(du) = \int_{[-1,1]} \bar{f}(t) \, \nu(dt),
		\]
		for every function $f \in C(\S^{i}(E))$.
	\end{defi}
	
	Note that the definition of pullback depends on the choice of measure we assign to each fiber of the map $u \mapsto \langle e_n, u \rangle$. Here, we chose the Lebesgue probability measure $\sigma_t$ on each fiber when $t \in (-1,1)$, and we assign a Dirac measure of mass $1$ to each of the poles $e_n$ and $-e_n$. This guarantees that $\nu_E$ is a zonal measure on $\S^{i}(E)$ and $\overline{\nu_E} = \nu$.
	
	Now let $\mu$ be a zonal measure on $\S^{n-1}$, and define $\mu_E := (\bar{\mu})_E$ (i.e., the pullback of $\bar{\mu}$ to $\S^{i}(E)$). Of course, $\mu_{\R^n} = \mu$ by Proposition~\ref{zonaldisintegration}. Furthermore, $\overline{\mu_E} = \bar{\mu}$. 
	
	The following  corollary is an immediate consequence of Proposition~\ref{zonaldisintegration}.
	\begin{cor}
		Let $1\leq i\leq n-1$. For every zonal function $f \in C(\S^{n-1})$, and any subspace $E \in \Gr_{i+1}(\R^n)$ with $e_n \in E$, we have
		\begin{equation}\label{eq:muvsmuE}
			\int_{\S^{n-1}} f(u) \, \mu(du) = \int_{\S^i(E)} f(u) \, \mu_E(du).
		\end{equation}
	\end{cor}
	
	Next, we turn to Christoffel--Minkowski problems for bodies of revolution and the disk as reference body. Here, we will also make use of the following Kubota-type formula from \cite{Hug2024}*{Thm.~3.2} (see also \cite{Brauner2024a}): For $1\leq i\leq n-2$, $K\in\K(\R^n)$, and $f\in C(\S^{n-1})$,
	\begin{align}\label{eq:zonal_Kubota_polarized}
		\int_{\Gr_{i+1}(\R^n,e_n)} \!\int_{\S^{i}(E)} \! f(u)\, dS^E_{i}(K|E,u) \, dE
		= \frac{\kappa_{i}}{\kappa_{n-1}}\! \int_{\S^{n-1}} \! f(u)\, dS_{i}(K,\DD;u).
	\end{align}
	Here, we denote by $\Gr_k(\R^n, e_n) \subset \Gr_k(\R^n)$ the Grassmannian of all subspaces $E \in \Gr_k(\R^n)$ that contain the axis $e_n$.
	
	We are now in position to prove \cref{mthm:ZonalMinkChristDisk}, stated again for the reader's convenience.
	\begin{thm}\label{thm:ZonalMinkChristDisk}
		Let $1 \leq i < n-1$ and $\mu$ be a non-negative, zonal Borel measure on $\S^{n-1}$. Then the following are equivalent:
		\begin{enumerate}[label=\upshape(\alph*)]
			\item\label{it:MinkChristDiskEx} There exists a body of revolution $K \in \K(\R^n)$, which is not a segment and such that $\mu = S_i(K, \DD; {{}\cdot{}})$.
			\item\label{it:MinkChristDiskCond} $\mu$ is centered and not concentrated on $\S^{n-2}(e_n^\perp)$.
		\end{enumerate}
		In this case, the body $K$ is unique up to a translations by a multiple of $e_n$.
	\end{thm}
	
	\begin{proof}
		First, note that \cref{it:MinkChristDiskEx} clearly implies that $\mu$ is centered. Moreover, from \eqref{eq:zonal_Kubota_polarized}, we see that the measure $S_i(K, \DD; \cdot)$ may be supported on $\S^{n-2}(e_n^\perp)$ only if $S_i^E(K|E, \cdot)$ is supported on $\S^{i-1}(E \cap e_n^\perp)$ for all $E \in \Gr_{i+1}(\R^n, e_n)$. However, as $K|E$ is a body of revolution with axis $e_n$, by Minkowski's theorem, this is impossible for $i > 1$. For $i=1$, this implies that $K|(e_n \vee u) \subset u^\perp$ for all $u \in \S^{n-2}(e_n^\perp)$, which implies that $K$ must be a segment -- a contradiction. Hence, \cref{it:MinkChristDiskEx} implies \cref{it:MinkChristDiskCond}.
		
		Now assume that \cref{it:MinkChristDiskCond} holds, and let $E \in \Gr_{i+1}(\R^n, e_n)$ be arbitrary. If $\mu$ is supported on $\mathrm{span}\{e_n\}$, then choosing 
		$$K = \left( \frac{\mu(\S^{n-1})}{2\kappa_{n-1}} \right)^{\frac{1}{i}} \DD$$		
		satisfies \cref{it:MinkChristDiskEx}, and we are done. Otherwise, by zonality, neither $\mu$ nor $\mu_E=(\bar{\mu})_E$ is concentrated on any great subsphere. Furthermore, by \cref{it:MinkChristDiskCond} and \eqref{eq:muvsmuE}, the measure $\mu_E$ is centered. Therefore, by the solution to the Minkowski problem in the subspace $E$, there exists a full-dimensional convex body of revolution $K_E \in \K(E)$ such that $\mu_E = S_i^E(K_E, \cdot)$.
		
		Note that $\mu_E = \tau_\ast \mu_F$ whenever $E, F \in \Gr_{i+1}(\R^n, e_n)$ and $\tau \in \SO(n-1)$ satisfies $\tau F = E$. Consequently, we have $K_E = \tau K_F$, and so we define a convex body of revolution $K \in \K(\R^n)$ by setting $K|E = K_E$ for all $E \in \Gr_{i+1}(\R^n, e_n)$. Using \eqref{eq:zonal_Kubota_polarized} and \eqref{eq:muvsmuE}, we obtain
		\begin{align*}
			\int_{\S^{n-1}} f(u) \, d\mu(u) &= \int_{\S^{i}(E)} f(v) \, d\mu_E(v) = \int_{\S^{i}(E)} f(v) \, dS_i^E(K|E, v) \\
			&= \int_{\Gr_{i+1}(\R^n,e_n)} \int_{\S^{i}(E)} f(u)\, dS^E_{i}(K|E,u) \, dE = \frac{\kappa_{i}}{\kappa_{n-1}} \int_{\S^{n-1}} f(u)\, dS_{i}(K,\DD;u),
		\end{align*}
		for every zonal function $f \in C(\S^{n-1})$. Scaling $K$ appropriately, we have shown \cref{it:MinkChristDiskEx}.
		
		For the uniqueness of $K$, suppose that bodies of revolution $K, L \in \K(\R^n)$ are given such that $\mu = S_i(K, \DD; {{}\cdot{}}) =S_i(L, \DD;{{}\cdot{}})$. Then, arguing as in the previous step of the proof,
		\begin{align*}
			\int_{\S^{i}(E)} f(v) \, dS_i^E(K|E, v) = \frac{\kappa_{i}}{\kappa_{n-1}} \int_{\S^{n-1}} f(u)\, d\mu(u) = \int_{\S^{i}(E)} f(v) \, dS_i^E(L|E, v),
		\end{align*}
		for all zonal $f \in C(\S^{n-1})$ and $E \in \Gr_{i+1}(\R^n, e_n)$. As $f$ was arbitrary, we conclude that $S_i^E(K|E, \cdot) = S_i^E(L|E, \cdot)$, and the uniqueness in the classical Minkowski problem (in $E$) implies that $K|E$ and $L|E$ are translates of each other. Since both bodies are bodies of revolution, we conclude $K|E = L|E + c e_n$ for some $c \in \R$, and, thus, again by symmetry, $K = L + c e_n$, concluding the proof.
	\end{proof}
	
	\subsection[The adjoint transform of T R]{The adjoint transform of \texorpdfstring{$T_R$}{T R}}\label{sec:adj}
	Next, we study the adjoint transform of $T_R$, which is needed to transfer the result for the disk area measure to general reference bodies of revolution. As a first step, we determine the adjoint of $T_R$ as a map from $C[a_-,a_+] \to C[-1,1]$. Here and throughout, we work on the interval $(a_-,a_+) \subset (-1,1)$, assuming that $-1 \leq a_- < 0 < a_+ \leq 1$. We denote by $\mathcal{M}[a,b]$ the space of signed finite Borel measures on the interval $[a,b]$. Moreover, in the following we denote
	\begin{equation*}
		R(t^\times)=\lim_{\eta\searrow 1}R(\eta t)=\begin{cases}
			R(t^+)=\lim_{s\to t^+} R(s),	&t\geq 0,\\
			R(t^-)=\lim_{s\to t^-} R(s),	&t\leq 0.
		\end{cases}
	\end{equation*}
	
	\begin{lem}\label{lem:T_R_adjoint:I}
		Let $R\in\BV_+(a_-,a_+)$ such that $R$ vanishes outside $[a_-,a_+]$. Then the map $T_R: C[a_-,a_+]\to C[-1,1]$ is a bounded operator and its adjoint $T_R^{\ast} : \mathcal{M}[-1,1]\to \mathcal{M}[a_-,a_+]$ is given by
		\begin{equation*}
			T_{\! R}^\ast \sigma(dt)
			= R(t^\times) \sigma(dt) + \bigg(\int_{[t,\sign t]} \! s\, \sigma(ds)\bigg) \nu_R(dt).
		\end{equation*}
		Moreover, if $\sigma$ is concentrated on $\{0\}$, then so is $T_R^\ast \sigma$.
	\end{lem}
	\begin{proof}
		First, let $f \in C[a_-,a_+]$. Then, as $R \in \BV_+(a_+,a_+)$, $|\nu_R|((a_-,a_+)) < \infty$ and by \cref{eq:equivFormTR}
		\begin{align*}
			|T_R f(t)| \leq |R(0)f(t)| + \int_{(0,t]} |sf(t) - tf(s)| d|\nu_R|(s) \leq \left(|R(0)| + 2|\nu_R|((a_-,a_+))\right) \sup_{[a_-,a_+]}|f|,
		\end{align*}
		whenever $t \in (a_-,a_+)$. For $t \geq a_+$ and $t\leq a_-$, \cref{ex:Tind} shows that $T_R f(t)$ is linear and, thus, can also be bounded in terms of $\sup_{s \in [a_-,a_+]}|f(s)|$. Consequently, $T_R:C[a_-,a_+] \to C[-1,1]$ is a bounded operator and has a well-defined adjoint $T_R^\ast$.
		
		Next, take some $f\in C[a_-,a_+]$ and $\sigma\in \mathcal{M}[-1,1]$. Then
		\begin{align*}
			\int_{[a_-,a_+]} f(t)\, T_{\! R}^\ast\sigma(dt)
			= \int_{[-1,1]} T_{\! R} f(t)\, \sigma(dt).
		\end{align*}
		We split the interval $[-1,1]=[-1,0)\cup [0,1]$. By Fubini's theorem, we have that
		\begin{align*}
			&\int_{[0,1]} T_{\! R} f(t)\, \sigma(dt)
			= \int_{[0,1]} R(t^+)f(t)\, \sigma(dt) + \int_{[0,1]} t \int_{[0,t]} f(s)\, \nu_R(ds)\, \sigma(dt) \\
			&\qquad = \int_{[0,1]} R(t^+)f(t)\, \sigma(dt) + \int_{[0,1]} \int_{[s,1]} t\, \sigma(dt)\, f(s)\, \nu_R(ds),
		\end{align*}
		and, since $R$ and $\nu_R$ vanish outside $[a_-,a_+]$, the desired formula holds on $[0,1]$. The argument for $[-1,0)$ is analogous.
		
		Finally, recall that $\nu_R$ is a signed Radon measure on $(-1,1)$ and $|\nu_R|(\{0\})=0$, and thus, $T_R^\ast\delta_0=R(0)\delta_0$. This immediately yields the ``moreover'' part of the statement.
	\end{proof}
	
	The adjoint $T_R^\ast$ now allows us to relate in particular (partial) integrals over linear functions. This was needed when determining the local behavior of the mixed area measures in \cref{sec:fireyEst} and will be needed for describing the behavior at the boundary points of the interval.
	
	\begin{lem}\label{lem:T_R_adjoint:I:caps}
		Let $R\in\BV_+(a_-,a_+)$ such that $R$ vanishes outside $[a_-,a_+]$, let $\sigma\in \mathcal{M}[-1,1]$ and denote $\mu=T_R^\ast \sigma\in\mathcal{M}[a_-,a_+]$. Then for all $t\in [a_-,a_+]$,
		\begin{equation*}
			\int_{(t,a_{\sign t}]} s\, \mu(ds)
			= R(t^\times) \int_{(t,\sign t]}s\, \sigma(ds).
		\end{equation*}
		Moreover,
		\begin{equation*}
			\lim_{t\nearrow a_+} \frac{\mu((t,a_+])}{R(t)}
			= \int_{[a_+,1]} \frac{s}{a_+}\, \sigma(ds)
			\qquad\text{and}\qquad
			\lim_{t\searrow a_-} \frac{\mu([a_-,t))}{R(t)}
			= \int_{[-1,a_-]} \frac{s}{a_-}\, \sigma(ds).
		\end{equation*}
	\end{lem}
	\begin{proof}
		By \cref{lem:T_R_adjoint:I} and Fubini's theorem, for all $t\in [0,a_+]$,
		\begin{align*}
			&\int_{(t,a_{+}]} s\, \mu(ds)
			= \int_{(t,a_+]} s R(s^+)\, \sigma(ds) + \int_{(t,a_+]} \int_{[s,1]} x\, \sigma(dx)\, s\, \nu_R(ds) \\
			&\qquad = \int_{(t,1]} s R(s^+)\, \sigma(ds) + \int_{(t,1]} \int_{(t,x]} s\, \nu_R(ds)\, x\, \sigma(dx).
		\end{align*}
		Here, we used that $R$ and $\nu_R$ vanish on $(a_+, 1]$. Next, recall that, by \cref{eq:fundthmCalcBV1},
		\begin{equation*}
			\int_{(t,x]} s\, \nu_R(ds)
			= R(t^+)-R(x^+).
		\end{equation*}
		Plugging this into the expression above yields the first part of the statement for $t\in [0,a_+]$. The argument for $t\in [a_-,0]$ is analogous.
		
		For the ``moreover'' part of the lemma, observe that
		\begin{equation*}
			\lim_{t\nearrow a_+} \frac{\mu((t,a_+])}{R(t)}
			= \lim_{t\nearrow a_+} \frac{1}{R(t)} \int_{(t,a_+]} \frac{s}{a_+}\, \mu(ds)
			= \lim_{t\nearrow a_+} \int_{(t,1]} \frac{s}{a_+}\, \sigma(ds)
			= \int_{[a_+,1]} \frac{s}{a_+}\, \sigma(ds).
		\end{equation*}
		The argument for the limit where $t\searrow a_-$ is analogous.	
	\end{proof}
	
	In the proof of \cref{mthm:Christoffel_Minkowski_zonal}, we want to transfer the solution for the disk as reference body to arbitrary reference bodies of revolution. As was indicated in the introduction, this is equivalent to determine which (non-negative) measures $\mu$ lie in the image of the adjoint map $T_R^\ast$. The next proposition provides a complete description of this image.
	
	\begin{prop}\label{lem:T_R_adjoint_inverse:I}
		Let $R\in \BV_+(a_-,a_+)$ such that $R$ vanishes outside $[a_-,a_+]$, and $\mu\in\mathcal{M}[a_-,a_+]$. Then $\mu=T_R^\ast \sigma$ for some $\sigma\in\mathcal{M}[-1,1]$ if and only if
		\begin{equation*}
			\frac{1}{R(t^\times)} \mu(dt) + \bigg(\int_{[t,a_{\sign t}]} \!s\, \mu(ds)\bigg) \nu_{\frac 1R}(dt)
		\end{equation*}
		defines a finite signed Radon measure on $(a_-,a_+)$, and the limits
		\begin{equation*}
			\lim_{t\nearrow  a_+} \frac{\mu((t,a_+])}{R(t)}
			\quad\text{and}\quad
			\lim_{t\searrow a_-} \frac{\mu([a_-,t))}{R(t)}
			\qquad\text{exist and are finite}.
		\end{equation*}
		In this case, a suitable $\sigma_0\in\mathcal{M}[-1,1]$ in the preimage of $\mu$ is given by
		\begin{align}\label{eq:lem:T_R_adjoint_inverse:I}
			\begin{split}
				\sigma_0(dt)
				&= \indf_{(a_-,a_+)}(t) \left(\frac{1}{R(t^\times)} \mu(dt) + \bigg(\int_{[t,a_{\sign t}]} \!s\, \mu(ds)\bigg) \nu_{\frac 1R}(dt)\right) \\
				&\qquad  + \bigg(\lim_{s\nearrow  a_+} \frac{\mu((s,a_+])}{R(s)}\bigg) \delta_{a_+}(dt) + \bigg(\lim_{s\searrow a_-} \frac{\mu([a_-,s))}{R(s)}\bigg) \delta_{a_-}(dt),
			\end{split}
		\end{align}
		and every $\sigma$ satisfying $\mu = T_R^\ast \sigma$ is equal to $\sigma_0$ on $(a_-,a_+)$.
	\end{prop}
	\begin{proof}
		First, we take some function $g\in C[a_-,a_+]$ and some $a_-<a'<0<a<a_+$. Then $\frac 1R \in \BV_+(a',a)$ and, by \cref{lem:T_R_bijective:I} and \cref{lem:continuous_in_D_R:I}, $T_{\frac 1R}g\in C[a',a]$, and thus,
		\begin{align*}
			\int_{[0,a]} T_{\frac 1R}g \,d\mu
			= \int_{[0,a]} \left[\frac{g(t)}{R(t^+)} + t\int_{(0,t]}g(s) \, \nu_{\frac 1R}(ds)\right]  \mu(dt).
		\end{align*}
		By Fubini's theorem and since $\nu_{ \frac 1R}(\{0\}) = 0$, we have that
		\begin{align*}
			&\int_{[0,a]} t\int_{(0,t]}g(s) \, \nu_{\frac 1R}(ds) \, \mu(dt)
			= \int_{[0,a]} g(t) \int_{[t,a]} s\, \mu(ds) \, \nu_{\frac 1R}(dt) \\
			&\qquad = \int_{[0,a]} g(t) \int_{[t,a_+]} s\, \mu(ds) \, \nu_{\frac 1R}(dt) - \int_{(a,a_+]} s\, \mu(ds) \int_{[0,a]} g(t)  \, \nu_{\frac 1R}(dt).
		\end{align*}
		Combining these equations, we obtain
		\begin{align}\label{eq:proof:lem:T_R_adjoint_inverse:I}
			\begin{split}
				\int_{[0,a]} T_{\frac 1R}g \, d\mu 
				&= \int_{[0,a]} g(t) \bigg[ \frac{1}{R(t^+)}\mu(dt) + \int_{[t,a_+]}s\, \mu(ds)\, \nu_{\frac 1R}(dt) \bigg] \\
				&\qquad - \int_{(a,a_+]} \! s\, \mu(ds) \int_{[0,a]} g(t)  \, \nu_{ \frac 1R}(dt).
			\end{split}
		\end{align}
		
		Suppose now that $\mu=T_R^\ast \sigma$ for some $\sigma\in\mathcal{M}[-1,1]$. Take some function $g\in C_c(a',a)$ with $a_-<a'<0<a<a_+$. Then, $T_{\frac{1}{R}}g \in C[a_-,a_+]$ and for $t > a$
		\begin{align*}
			T_{\frac{1}{R}}g(t) = t \int_{(0,a]}g(s) d\nu_{ \frac 1R}(s),
		\end{align*}
		and similar for $t<a'$, that is, $T_{\frac{1}{R}}g$ is linear outside $[a',a]$. Consequently,
		\begin{align*}
			\int_{[0,a]} T_{\frac 1R}g \, d\mu + \int_{(a,a_+]} \! s\, \mu(ds) \int_{[0,a]} g(t)  \, \nu_{ \frac 1R}(dt) = \int_{[0,a]} T_{\frac 1R}g \, d\mu + \int_{(a,a_+]}T_{\frac{1}{R}}g\, d\mu,
		\end{align*}
		so repeating the argument for $[a',0]$, we obtain by \cref{eq:proof:lem:T_R_adjoint_inverse:I},
		\begin{align*}
			\int_{[a_-,a_+]}T_{\frac{1}{R}}g\, d\mu = \int_{[a',a]} g(t) \bigg[ \frac{1}{R(t^\times)}\mu(dt) + \int_{[t,a_{\sign t }]}s\, \mu(ds)\, \nu_{\frac 1R}(dt) \bigg],
		\end{align*}
		and
		\begin{align*}
			\int_{[a_-,a_+]}T_{\frac{1}{R}}g\, d\mu = \int_{[-1,1]} T_R T_{\frac{1}{R}}g d\sigma = \int_{[-1,1]} g d\sigma.
		\end{align*}
		Hence, letting $a \nearrow a_+$ and $a' \searrow a_-$, we obtain that
		\begin{equation}\label{eq:proofTradjInvIntNec}
			\indf_{(a_-,a_+)}(t)\sigma(dt)
			= \indf_{(a_-,a_+)}(t) \left(\frac{1}{R(t^\times)} \mu(dt) + \bigg(\int_{[t,a_{\sign t}]} \!s\, \mu(ds)\bigg) \nu_{\frac 1R}(dt)\right).
		\end{equation}
		In particular, the expression on the right hand side is a finite signed Radon measure on $(a_-,a_+)$. In combination with \cref{lem:T_R_adjoint:I:caps}, this shows that the conditions on $\mu$ in the statement of the lemma are necessary. Moreover, \cref{lem:T_R_adjoint:I:caps} and \cref{eq:proofTradjInvIntNec} show that \cref{eq:lem:T_R_adjoint_inverse:I} defines a suitable preimage $\sigma$.
		
		Conversely, suppose now that $\mu$ satisfies the conditions stated in the lemma and define $\sigma$ to be the right hand side of \cref{eq:lem:T_R_adjoint_inverse:I}.		
		In order to show that $\mu=T_R^\ast \sigma$, we need to show that
		\begin{align*}
			\int_{[a_-,a_+]} f d\mu = \int_{[-1,1]} T_R f d\sigma = \int_{[a_-,a_+]} T_R f d\sigma
		\end{align*}
		holds for all $f \in C[a_-,a_+]$. To this end, fix $f \in C[a_-,a_+]$ and write $g = T_R f$. Then, by \cref{lem:T_R_bijective:I} and \cref{lem:continuous_in_D_R:I}, $g \in C[a_-,a_+]$ and $f = T_{\frac{1}{R}}g$. Hence, we need to show that
		\begin{align*}
			\int_{[a_-,a_+]} T_{\frac{1}{R}}g d\mu = \int_{[a_-,a_+]} g d\sigma.
		\end{align*}
		As by \cref{ex:Tind}, $T_{\indf_{(-1,a)}}T_{\frac{1}{R}}g$ converges uniformly to $T_{\frac{1}{R}}g$ on $[0,a_+]$ as $a \nearrow a_+$, \eqref{eq:proof:lem:T_R_adjoint_inverse:I} implies that
		\begin{align*}
			\int_{[0,a_+]} T_{\frac{1}{R}}g d\mu &= \lim_{a \nearrow a_+}\int_{[0,a_+]}T_{\indf_{(-1,a)}}T_{\frac{1}{R}}g d\mu\\
			&=\lim_{a \nearrow a_+} \left(\int_{[0,a]}T_{\frac{1}{R}}g d\mu + \int_{(a,a_+]}s d\mu(s) \frac{T_{\frac{1}{R}}g(a)}{a} \right)\\
			&=\int_{[0,a_+)} g(t) d\sigma + \lim_{a \nearrow a_+}\int_{(a,a_+]}s d\mu(s) \left( -\int_{[0,a]} g(t)  \, \nu_{ \frac 1R}(dt) + \frac{T_{\frac{1}{R}}g(a)}{a}\right).
		\end{align*}
		Plugging in the definition of $T_{\frac{1}{R}}g$ inside the brackets,
		\begin{align*}
			-\int_{[0,a]} g(t)  \, \nu_{ \frac 1R}(dt) + \frac{T_{\frac{1}{R}}g(a)}{a} = \frac{g(a)}{a R(a^+)},
		\end{align*}
		and using the definition of $\sigma$ in \eqref{eq:lem:T_R_adjoint_inverse:I} then yields
		\begin{align*}
			\int_{[0,a_+]} T_{\frac{1}{R}}g d\mu &= \int_{[0,a_+)} g(t) d\sigma + \lim_{a \nearrow a_+}\frac{\int_{(a,a_+]}s d\mu(s)}{a R(a^+)} g(a)\\
			&=\int_{[0,a_+)} g(t) d\sigma + \left(\lim_{a \nearrow a_+}\frac{\mu((a,a_+])}{R(a^+)}\right) g(a_+) = \int_{[0,a_+]} g(t) d\sigma.
		\end{align*}
		Repeating the argument for $[a_-,0]$ then yields the claim $\mu = T_R^\ast \sigma$.
	\end{proof}
	
	For the operator $\widehat{T}_R f = T_R f + c f(0) |\cdot|$ a similar statement follows easily, only altering the condition at $t=0$. Indeed, we can write the adjoint $\widehat{T}_{R}^\ast$ in terms of the adjoint $T_R^\ast$.
	\begin{lem}\label{lem:T_R_vertSeg_adjoint:I}
		Let $R\in\BV_+(a_-,a_+)$ such that $R$ vanishes outside $[a_-,a_+]$. Then the map $\widehat{T}_R: C[a_-,a_+]\to C[-1,1]$ is a bounded operator and its adjoint $\widehat{T}_R^{\ast} : \mathcal{M}[-1,1]\to \mathcal{M}[a_-,a_+]$ is given by
		\begin{equation*}
			\widehat{T}_{\! R}^\ast \sigma 
			= T_R^\ast \sigma+ c \int_{[-1,1]} |s| d\sigma(s) \delta_0.
		\end{equation*}
		Moreover, if $\sigma$ is concentrated on $\{0\}$, then so is $\widehat{T}_R^\ast \sigma$.
	\end{lem}
	\begin{proof}
		It follows directly from \cref{lem:T_R_adjoint:I} that $\widehat{T}_R$ is a bounded operator. To determine its adjoint, note that
		\begin{align*}
			\int_{[a_-,a_+]} f d(\widehat{T}_R^\ast \sigma) = \int_{[-1,1]} (T_R f)(s) + c f(0) |s| d\sigma(s),
		\end{align*}
		which immediately yields the remaining claims.
	\end{proof}

	Similarly, an analogous statement to \cref{lem:T_R_adjoint:I:caps} is derived. For the full interval, it reads
	\begin{lem}\label{lem:T_Rhat_adjoint:I:AbsVal}
		Let $R\in\BV_+(a_-,a_+)$ such that $R$ vanishes outside $[a_-,a_+]$, let $\sigma\in \mathcal{M}[-1,1]$ and denote $\mu=\widehat{T}_R^\ast \sigma\in\mathcal{M}[a_-,a_+]$. Then 
		\begin{align}
			\int_{[a_-,a_+]} |\cdot | d \mu = R(0) \int_{[-1,1]} |\cdot| d\sigma.
		\end{align}
	\end{lem}
	\begin{proof}
		By \cref{lem:T_R_vertSeg_adjoint:I}, using that the absolute value vanishes at $t=0$,
		\begin{align*}
			\int_{[a_-,a_+]} |\cdot | d \mu = \int_{[a_-,a_+]} |\cdot| d\bigg(T_R^\ast \sigma+ c \int_{[-1,1]} |s| d\sigma(s) \delta_0\bigg) = \int_{[-1,1]} T_R( |\cdot|) d\sigma.
		\end{align*}
		Noting that by \cref{eq:equivFormTR}, $T_R(|\cdot|) = R(0) |\cdot|$ then yields the claim.
	\end{proof}

	The image of $\widehat{T}_R^\ast$ inside $\mathcal{M}[a_-,a_+]$ can now be deduced from \cref{lem:T_R_adjoint_inverse:I} and \cref{lem:T_Rhat_adjoint:I:AbsVal}, yielding merely a shift by a multiple of $\delta_0$.
	\begin{lem}\label{lem:T_R_adjoint_inverse:I:vertSegm}
		Let $R\in \BV_+(a_-,a_+)$ such that $R$ vanishes outside $[a_-,a_+]$, $c \in \R$, and $\mu\in\mathcal{M}[a_-,a_+]$. Then $\mu=\widehat{T}_R^\ast \sigma$ for some $\sigma\in\mathcal{M}[-1,1]$ if and only if
		\begin{align*}
			\mu - \left(\frac{c}{R(0)} \int_{[a_-,a_+]} |\cdot| d\mu\right) \delta_0 = T_R^\ast \sigma,
		\end{align*}
		that is, if and only if
		\begin{equation*}
			\frac{1}{R(t^\times)} \mu(dt) + \bigg(\int_{[t,a_{\sign t}]} \!s\, \mu(ds)\bigg) \nu_{\frac 1R}(dt)
		\end{equation*}
		defines a finite signed Radon measure on $(a_-,a_+)$, and the limits
		\begin{equation*}
			\lim_{t\nearrow  a_+} \frac{\mu((t,a_+])}{R(t)}
			\quad\text{and}\quad
			\lim_{t\searrow a_-} \frac{\mu([a_-,t))}{R(t)}
			\qquad\text{exist}.
		\end{equation*}
	\end{lem}
	\begin{proof}
		By \cref{lem:T_R_vertSeg_adjoint:I} and \cref{lem:T_Rhat_adjoint:I:AbsVal}, $\mu = \widehat{T}_R^\ast \sigma$, if and only if
		\begin{align*}
			\mu = T_R^\ast \sigma + \left(c \int_{[-1,1]}|\cdot| d\sigma\right) \delta_0 = T_R^\ast \sigma + \left(\frac{c}{R(0)} \int_{[a_-,a_+]} |\cdot | d \mu\right) \delta_0.
		\end{align*}
		\cref{lem:T_R_adjoint_inverse:I} thus yields the claim, noting that the conditions remain essentially unchanged when adding mass at $t=0$ to $\mu$.
	\end{proof}

	\subsection{Transfer to general reference bodies}\label{sec:MCprobTransfer}

	In this section, we prove \cref{mthm:Christoffel_Minkowski_zonal} in the following more general version.
	
	\begin{thm}\label{thm:Christoffel_Minkowski_zonal}
		Let $1\leq i < n-1$ and $\CC=(C_1,\ldots,C_{n-i-1})$ be a family of convex bodies of revolution, none of which is a segment. Suppose that $\mu$ is a non-negative, centered, zonal Borel measure on $\S^{n-1}$. Then there exists a body of revolution $K\in\K(\R^n)$ that is not a segment with $\mu=S_i(K,\CC;{}\cdot{})$ if and only if
		\begin{enumerate}[label=\upshape(\roman*)]
			\item\label{thm:Christoffel_Minkowski_zonal:support}
			$\mu$ is supported in $\{u\in\S^{n-1}: a_{\CC,-} \leq \pair{e_n}{u} \leq a_{\CC,+} \}$,
			\item\label{thm:Christoffel_Minkowski_zonal:concentr}
			$\mu$ is not concentrated on $\S^{n-2}(e_n^\perp)$,
			\item\label{thm:Christoffel_Minkowski_zonal:measPos}
			the signed Radon measure on	$(a_{\CC,-},a_{\CC,+})$ given by
			\begin{equation}\label{eq:thmChrMinkMeasPos}
				\frac{1}{R_\CC(t)} \bar{\mu}(dt) + \left(\int_{[t,\sign t]} \!s\, \bar{\mu}(ds)\right) \nu_{\frac{1}{R_\CC}}(dt),
			\end{equation}
			is non-negative and finite,
			\item\label{thm:Christoffel_Minkowski_zonal:LimEx}
			the limits $\displaystyle \lim_{t \nearrow a_{\CC,+}} \frac{\mu(\mathrm{Cap}( e_n,t))}{R_\CC(t)}$ and $\displaystyle \lim_{t \searrow a_{\CC,-}} \frac{\mu(\mathrm{Cap}(- e_n,-t))}{R_\CC(t)}$ exist and are finite, and
			\item \label{thm:Christoffel_Minkowski_zonal:AtEquat}
			$\displaystyle R_\CC(0)\mu(\S^{n-2}(e_n^\perp)) - \frac{n-i-1}{\omega_{n-i-1}}W_{\CC} \int_{\S^{n-1}} \max\{\pair{e_n}{u},0\} \, \mu(du) \geq 0$.
		\end{enumerate}
	\end{thm}
	We prove the necessity and the sufficiency of the conditions in \cref{thm:Christoffel_Minkowski_zonal} separately. However, most of the work has been already done in previous sections.
	
	\begin{thm}\label{thm:Christoffel_Minkowski_zonal_Sufficient}
		In the situation of \cref{thm:Christoffel_Minkowski_zonal}, the conditions \cref{thm:Christoffel_Minkowski_zonal:support} to \cref{thm:Christoffel_Minkowski_zonal:AtEquat} are necessary in order to have $\mu=S_i(K,\CC;{}\cdot{})$.
	\end{thm}
	\begin{proof}
		Suppose that $\mu=S_i(K,\CC;{}\cdot{})$, where $K \in \K(\R^n)$ is either a full-dimensional body of revolution or $K=\DD$. Then \cref{lem:mixed_area_meas_support} directly implies condition \cref{thm:Christoffel_Minkowski_zonal:support}. To see \cref{thm:Christoffel_Minkowski_zonal:concentr}, note that by \cite{Schneider2014}*{Thm.~5.1.8}
		\begin{align*}
			\int_{\S^{n-1}} \abs{\pair{e_n}{u}} \, d\mu(u)
			= n V(K^{[i]},\CC,[-e_n,e_n]) > 0,
		\end{align*}
		as we can find linearly independent segments in $K, \dots, K, C_1, \dots, C_{n-i-1}$ spanning $e_n^\perp$. Consequently, as $\abs{\pair{e_n}{\cdot}}$ vanishes only on $\S^{n-2}(e_n^\perp)$, $\mu$ cannot be concentrated on $\S^{n-2}(e_n^\perp)$.
		
		Next, note that by \cref{thm:dictZonalVal_cont_disk},
		\begin{align*}
			\int_{[a_{\CC,-},a_{\CC,+}]}\bar f d\bar{\mu} = \int_{\S^{n-1}} f \, dS_i(K,\CC;{}\cdot{})	= \int_{\S^{n-1}} (\widehat{T}_\CC \bar{f})(\pair{e_n}{\cdot}) \, dS_i(K,\DD;{}\cdot{}),
		\end{align*}
		for all zonal $f \in C(\S^{n-1})$, that is, $\bar{\mu} = \widehat{T}_\CC^\ast \sigma$, where $\sigma = (\pair{e_n}{\cdot})_\ast S_i(K,\DD,{}\cdot{}) \in \mathcal{M}[-1,1]$. Hence, \cref{lem:T_R_adjoint_inverse:I:vertSegm} and \cref{lem:T_R_adjoint_inverse:I} imply \cref{thm:Christoffel_Minkowski_zonal:LimEx}. Moreover, by \cref{lem:T_R_adjoint_inverse:I},
		\begin{align}\label{eq:prfNecCondChri}
				\indf_{(a_{\CC,-},a_{\CC,+})}(t)\sigma(dt) = \frac{1}{R_\CC(t)} \bar{\mu}(dt) + \bigg(\int_{[t,a_{\sign(t)}]} \!s\, \bar{\mu}(ds)\bigg) \nu_{\frac{ 1}{R_\CC}}(dt).
		\end{align}
		Since $\sigma$ is non-negative and finite, this yields condition \cref{mthm:Christoffel_Minkowski_zonal:measPos}.
		
		For condition \cref{thm:Christoffel_Minkowski_zonal:AtEquat}, note that by \cref{eq:MixedVol_Integral} and \cref{eq:MixedVol_Proj},
		\begin{equation*}
			\int_{\S^{n-1}} \!\! \max\{0,\pair{e_n}{u}\}\, \mu(du)
			= nV(K^{[i]},\CC,[0,e_n])
			= V^{e_n^\perp}\!((K|e_n^\perp)^{[i]},\CC|e_n^\perp)
			= \kappa_{n-1} R_K(0)^{i}R_{\CC}(0).
		\end{equation*}
		Moreover, by \cref{eq:defWC} and \cref{eq:MixedAreaMeasRev},
		\begin{align*}
			&\mu(\S^{n-1}\cap e_n^\perp)
			= S_i(K,\CC;\S^{n-1}\cap e_n^\perp)
			= W_{(K^{[i]},\CC)} \\
			&\qquad = \frac{\omega_{n-1}}{n-1} \bigg( \frac{n-i-1}{\omega_{n-i-1}}R_K(0)^{i}W_{\CC} + i \ell_KR_K(0)^{i-1}R_{\CC}(0) \bigg) \\
			&\qquad \geq \frac{\omega_{n-1}}{n-1} \frac{n-i-1}{\omega_{n-i-1}} R_K(0)^{i}W_{\CC}
			= \frac{n-i-1}{\omega_{n-i-1}} \frac{W_{\CC}}{R_{\CC}(0)}\int_{\S^{n-1}} \max\{0,\pair{e_n}{u}\}\, \mu(du).
		\end{align*}
		Multiplying both sides with $R_{\CC}(0)$ yields condition \cref{thm:Christoffel_Minkowski_zonal:AtEquat}.
	\end{proof}
	
	\begin{thm}\label{thm:Christoffel_Minkowski_zonal_Necessary}
		In the situation of \cref{thm:Christoffel_Minkowski_zonal}, the conditions \cref{thm:Christoffel_Minkowski_zonal:support} to \cref{thm:Christoffel_Minkowski_zonal:AtEquat} are sufficient for the existence of $K \in \K(\R^n)$ with $\mu=S_i(K,\CC;{}\cdot{})$.
	\end{thm}
	\begin{proof}
		Let $\mu \in \mathcal{M}(\S^{n-1})$ be a non-negative, centered, zonal Borel measure on $\S^{n-1}$ satisfying conditions \cref{thm:Christoffel_Minkowski_zonal:support} to \cref{thm:Christoffel_Minkowski_zonal:AtEquat} in \cref{thm:Christoffel_Minkowski_zonal}. Then $\bar{\mu} = (\pair{e_n}{\cdot})_\ast \mu$ is a non-negative Borel measure concentrated on $[a_{\CC,-},a_{\CC,+}]$, that is, $\bar{\mu} \in \mathcal{M}[-a_{\CC,-},a_{\CC,+}]$. The conditions \cref{thm:Christoffel_Minkowski_zonal:measPos}, \cref{thm:Christoffel_Minkowski_zonal:LimEx} and \cref{thm:Christoffel_Minkowski_zonal:AtEquat} assert that the conditions of \cref{lem:T_R_adjoint_inverse:I:vertSegm} are satisfied so that there exists $\sigma \in \mathcal{M}[-1,1]$ such that $\bar{\mu} = \widehat{T}_\CC^\ast \sigma$. Moreover, \cref{thm:Christoffel_Minkowski_zonal:measPos} and \cref{thm:Christoffel_Minkowski_zonal:AtEquat} imply that we can choose $\sigma \geq 0$.
		
		By \cref{lem:T_R_adjoint:I:caps} and since $\mu$ is assumed to be centered,
		\begin{align*}
			0 = \int_{[-1,1]} s d\bar{\mu}(s) = R(0)\int_{[-1,1]} s d\sigma(s),
		\end{align*}
		that is, $\sigma$ is centered as well. Next, assume that $\sigma$ is concentrated on $\{0\}$. Then, by \cref{lem:T_R_adjoint:I}, $\bar{\mu}$ is concentrated on $\{0\}$, and thus $\mu$ is concentrated on $\S^{n-2}(e_n^\perp)$. This contradicts \cref{thm:Christoffel_Minkowski_zonal:concentr}, hence, $\sigma$ is not concentrated on $\{0\}$.
		
		We conclude that the zonal measure $\widetilde{\sigma}$ on $\S^{n-1}$ uniquely determined by $(\pair{e_n}{\cdot})_\ast \widetilde{\sigma} = \sigma$ satisfies the conditions of \cref{thm:ZonalMinkChristDisk}, that is, there exists a body of revolution $K \in \K(\R^n)$ such that $\widetilde{\sigma} = S_i(K, \DD, \cdot)$, and $K$ is not a segment.
		
		Finally, let $f=\bar{f}(\pair{e_n}{\cdot}) \in C(\S^{n-1})$ and write $g = \bar{g}(\pair{e_n}{\cdot}) \in C(\S^{n-1})$ with $\bar g = \widehat{T}_\CC \bar{f}$. Then, as $\bar \mu = \widehat{T}_\CC^\ast \sigma$ and by \cref{thm:dictZonalVal_cont_disk},
		\begin{align*}
			\int_{\S^{n-1}} f d\mu = \int_{[a_{\CC,-},a_{\CC,+}]} \bar f d\bar{\mu} = \int_{[-1,1]} \bar{g} d\sigma = \int_{\S^{n-1}} g\,dS_i(K, \DD; \cdot) = \int_{\S^{n-1}} f\,dS_i(K, \CC; \cdot).
		\end{align*}
		As $f$ was arbitrary, we deduce that $\mu = S_i(K, \CC, \cdot)$, concluding the proof.
	\end{proof}

	\begin{rem}[Uniqueness]\label{rem:ChristoffelMinkowski_unique}
		We briefly discuss the uniqueness of the body of revolution $K\in\K(\R^n)$ in the mixed Christoffel--Minkowski problem. Depending on the family of reference bodies $\CC$, there is a dichotomy:
		If $\S_\CC=\S^{n-1}$, then the solution $K$ is again unique up to a vertical translation; if $\S_\CC \subsetneq \S^{n-1}$, then the solution $K$ is highly non-unique.
		
		In the case where $\S_\CC=\S^{n-1}$, it was shown in the proof of \cref{thm:Christoffel_Minkowski_zonal} that $\mu=\hat{T}_\CC^\ast \sigma$, where we denote $ \bar{\mu} = (\pair{e_n}{{}\cdot{}})_\ast S(K^{[i]},\CC;{}\cdot{})$ and $\sigma= (\pair{e_n}{{}\cdot{}})_\ast S_i(K,\DD;{}\cdot{})$. Due to \cref{lem:T_R_adjoint:I:caps,lem:T_R_adjoint_inverse:I}, the measure $\sigma$ is completely determined by $\bar{\mu}$. From the uniqueness statement in \cref{thm:ZonalMinkChristDisk}, it then follows that $K$ is unique up to a vertical translation.
		
		In the case where $\S_{\CC} \subsetneq \S^{n-1}$, the measure $\mu =S(K^{[i]},\CC;{}\cdot{})$ vanishes outside of $\S_{\CC}$ (see \cref{lem:mixed_area_meas_support}). Due to the locality of mixed area measures (see \cref{lem:MixedAreaMeas_locally_det}), any body of revolution $\tilde{K}$ with $\tau(K,\S_\CC)=\tau(\tilde{K},\S_{\CC})$ will yield $\mu=S(\tilde{K}^{[i]},\CC;{}\cdot{})$. By adding and removing caps to and from $K$, for instance, one can easily construct a large family of such bodies $\tilde{K}$.
	\end{rem}
	
	\begin{rem}[Full-dimensionality]\label{rem:ChristoffelMinkowski_fulldim}
		In the statement of \cref{thm:Christoffel_Minkowski_zonal}, we consider a body of revolution $K\in(\R^n)$ which is \emph{not} a (possibly degenerate) vertical segment; that is, $K$ is either full-dimensional or a non-degenerate disk in $e_n^\perp$. In here, we briefly discuss when $K$ can be assumed to be full-dimensional.
		
		By the discussion in \cref{rem:ChristoffelMinkowski_unique}, if the family of reference bodies $\CC$ is such that $\S_\CC\subsetneq \S^{n-1}$, then we may always choose $K$ to be full-dimensional. 
		
		If $\S_\CC=\S^{n-1}$, then the body $K$ is unique up to a vertical translation, by \cref{rem:ChristoffelMinkowski_unique}.
		In that case, the restriction of $(\pair{e_n}{{}\cdot{}})_\ast S_i(K,\DD;{}\cdot{})$ to the subinterval $(-1,1)$ is given by  \cref{eq:thmChrMinkMeasPos}. The mixed area measure $S_i(K,\DD;{}\cdot{})$ is concentrated on the poles $\pm e_n$ precisely when $K$ is a disk. Thus, $K$ is full-dimensional if and only if the measure in \cref{eq:thmChrMinkMeasPos} is non-zero.
	\end{rem}

	\appendix
	
	\section{Technical proofs}\label{app:TechnLemmas} 
	In this appendix, we give the proofs of some technical results from \cref{sec:BV0,sec:defTRtrans,sec:TRinvertability}. We start with basic properties of the functions in $\BV_0(-1,1)$.
	\begin{proof}[Proof of \cref{lem:BV0diffCont}]
		Observe that by \cref{eq:fundthmCalcBV1}, it holds for $t > 0$ that
		\begin{align*}
			|R(t^+) - R((-t)^+)| \leq \int_{(-t,t]}|s|d\abs{\nu_R}(s) \leq t \abs{\nu_R}([-t,t]).
		\end{align*}
		As $\abs{\nu_R}$ is a Radon measure, $\abs{\nu_R}([-t,t]) \to \abs{\nu_R}(\{0\}) = 0$, as $t\to 0^+$, that is, $R(0^+) = R(0^-)$. Similarly,
		\begin{align*}
			\left|\frac{R(t^+) - R(0^+)}{t}\right| + \left|\frac{R(0^+) - R((-t)^+)}{t}\right| \leq \frac{1}{t} \int_{(0,t] \cup (-t,0]} \abs{s}\,d\abs{\nu_R}(s) \leq \abs{\nu_R}([-t,t])
		\end{align*}
		yields that both difference quotients on the left-hand side converge to zero, that is, $R$ is differentiable at $t=0$ with $R'(0) =0$.
	\end{proof}
	\begin{proof}[Proof of \cref{BV_0_algebraic}]
		\ref{BV_0_algebraic:linear_comb} and \ref{BV_0_algebraic:unit} are clear from the definition. For \ref{BV_0_algebraic:product}, first note that as
		\begin{align*}
			(R(t)-R(0))(Q(t) - Q(0)) = R(t)Q(t) - R(0)Q(t) - Q(0)R(t) + R(0)Q(0),
		\end{align*}
		by \ref{BV_0_algebraic:linear_comb} and \ref{BV_0_algebraic:unit}, we can assume that $R(0) = 0 = Q(0)$. Now, restricting first to integrals over the part $(0,1)$, \eqref{eq:fundthmCalcBV1} and Fubini's theorem imply for every $\phi \in C^1_c(-1,1)$
		\begin{align*}
			\int_{(0,1)} \phi'(t)R(t)Q(t)\, dt
			= -\int_{(0,1)}  \phi'(t) Q(t) \int_{(0,t]}s \, d\nu_R(s)\, \, dt
			= -\int_{(0,1)} \int_{[s,1)} \phi'(t)Q(t)\, dt \, s \, d\nu_R(s).
		\end{align*}
		Note here, that we can replace $R(t)$ by $R(t^+)$ in the first integral without changing its value.	Applying \eqref{eq:BV_0_int_by_parts} to the inner integral, we obtain, as $\spt \phi \subset (-1,1)$ is compact,
		\begin{align*}
			\int_{[s,1)} \phi'(t) Q(t) \, dt
			= -\phi(s)Q(s^-) + \int_{[s,1)} \phi(t) t\, d\nu_Q(t),
		\end{align*}
		which yields in total
		\begin{align*}
			\int_{(0,1)} \phi'(t)R(t)Q(t)\, dt
			= \int_{(0,1)} \phi(s) Q(s^-) s \, d\nu_R(s) - \int_{(0,1)} \int_{[s,1)} \phi(t) t\, d\nu_Q(t)\, s \, d\nu_R(s).
		\end{align*}
		Exchanging again the order of integration in the last integral, we obtain, by \eqref{eq:fundthmCalcBV1} again,
		\begin{align*}
			\int_{(0,1)} \int_{[s,1)} \phi(t) t\, d\nu_Q(t)\, s \, d\nu_R(s)
			= \int_{(0,1)} \int_{(0,t]}\, s \, d\nu_R(s)\, \phi(t) t\, d\nu_Q(t)
			= -\int_{(0,1)} \phi(t) t R(t^+)\, d\nu_Q(t),
		\end{align*}
		yielding the claim for $(0,1)$. A similar argument for $(-1,0]$ then yields the first equality of \ref{BV_0_algebraic:product}. The second equality follows directly by exchanging $R$ and $Q$.
		
		For \ref{BV_0_algebraic:inverse}, note that if $R$ is of locally bounded variation and $\inf_I \abs{R} > 0$ for every compact $I \subset (-1,1)$ then $1/R$ is also of locally bounded variation (and semi-continuous). By \ref{BV_0_algebraic:unit} and a similar argument as in the proof of \ref{BV_0_algebraic:product} for $\mu_R$ instead of $\nu_R$,
		\begin{align*}
			0 = d\mu_{R\cdot\frac{1}{R}}(t) = \frac{1}{R(t^-)} d\mu_R(t) + R(t^+)d\mu_{\frac 1R}(t),
		\end{align*}
		that is, $d\mu_{\frac 1R}(t) = \frac{-1}{R(t^-)R(t^+)}d\mu_R(t)$. Consequently, $d\nu_{\frac 1R}(t) = -\frac{1}{t}d\mu_{\frac 1R}(t) = \frac{-1}{R(t^-)R(t^+)}d\nu_R(t)$ is a signed Radon measure, that is, $1/R \in \BV_0(-1,1)$, concluding the proof.
	\end{proof}
	
	Next, we continue with statements on well-definedness and basic properties of the transform $T_R$. The proof of the following lemma is a direct computation.
	\begin{lem}\label{lem:TRreflect}
		Let $R \in \BV_0(-1,1)$ and denote by $\bar R(t) = R(-t)$, $t \in (-1,1)$, the reflected function. Then $\bar R \in \BV_0(-1,1)$, $d\nu_{\bar R}(t) = d\nu_{R}(-t)$ and $T_{\bar R}f(t) = (T_R \bar f)(-t)$, $t \in (-1,1)$.
	\end{lem}
	
	\begin{proof}[Proof of \cref{prop:TRContContGroupProp}]
		For \cref{it:T_R_cont_to_cont}, let $0\leq  t' < t < 1$. By \eqref{eq:equivFormTR}, we obtain
		\begin{align*}
			T_Rf(t)-T_Rf(t')
			&= R(0)(f(t)-f(t')) - \int_{(t',t]} (sf(t)-tf(s) )\, d\nu_R(s) \\
			&\qquad - \int_{(0,t']} (s(f(t)-f(t'))-(t-t')f(s))\, d\nu_R(s).
		\end{align*}
		As $R$ is bounded, $\nu_R$ is a Radon measure and $f$ is continuous on $(-1,1)$ (whence $f$ and $s \mapsto sf(t)-tf(s)$ are bounded locally), the right-hand side becomes arbitrarily small whenever $|t-t'|$ is small, showing continuity of $T_R f$ on $[0,1)$. A similar argument shows continuity on $(-1,0]$, so $T_Rf(0^-) = R(0)f(0) = T_Rf(0^+)$ yields the claim.
		
		Next, the second part of \cref{it:T_R_group_property} is an immediate consequence of \cref{eq:T_R_group_property}, \cref{BV_0_algebraic}\ref{BV_0_algebraic:inverse}, and the fact that $T_1 = \mathrm{id}$. To show \eqref{eq:T_R_group_property}, let $f\in C(-1,1)$ and note first that, by \eqref{eq:fundthmCalcBV1},
		\begin{align}\label{eq:prfTRgroup_firstid}
			sT_Qf(t) - tT_Qf(s) =& Q(0)(sf(t) - tf(s)) - \int_{(0,s]} sxf(t) - stf(x) -(txf(s)-tsf(x)) d\nu_Q(x)\nonumber\\&
			- s\int_{(s,t]} xf(t) - tf(x) d\nu_Q(x)\nonumber\\
			=& Q(x^+)(sf(t) - tf(s)) - s\int_{(s,t]} xf(t) - tf(x) d\nu_Q(x)
		\end{align}
		for all $0 < s \leq t < 1$. Moreover, Fubini's theorem and \eqref{eq:fundthmCalcBV1} imply for $t \in (0,1]$ that
		\begin{align}\label{eq:prfTRgroup_secid}
			\int_{(0,t]}s\int_{(s,t]} xf(t) - tf(x) d\nu_Q(x) d\nu_R(s) &= \int_{(0,t]} (xf(t) - tf(x)) \int_{(0,x)} s d\nu_R(s) d\nu_Q(x)\nonumber\\
			&=-\int_{(0,t]} (xf(t) - tf(x)) (R(x^-) - R(0))d\nu_Q(x).
		\end{align}
		Combining \eqref{eq:prfTRgroup_firstid} and \eqref{eq:prfTRgroup_secid} with the definition of $T_R$, then yields for all $t\in [0,1)$,
		\begin{align*}
			T_R (T_Q f)(t) &= R(0) T_Qf(t) - \int_{(0,t]} (sT_Qf(t) - t T_Q f(s)) d\nu_R(s)\\
			&=R(0) Q(0)f(t) - \int_{(0,t]} (xf(t) - tf(x)) \left(Q(x^+) d\nu_R(x) + R(x^-)d\nu_Q(x)\right).
		\end{align*}
		Thus, by \cref{BV_0_algebraic}\ref{BV_0_algebraic:product}, $T_R(T_Q f)(t)=T_{RQ}f(t)$ for all $t \in [0,1)$. By \cref{lem:TRreflect}, the claim also follows for $t \in (-1,0)$.
		For \cref{it:continuous_in_D_R}, finally, let $f\in C[-1,1]$. Then clearly, the limits $\lim_{t\to\pm 1} R(t)f(t)$ exist. Moreover, for any given $\varepsilon>0$, there exists $t_0\in (0,1)$ such that $\abs{f(t)-f(1) t }<\varepsilon$ for all $t\in (t_0,1)$. Due to the fact that $R$ is of bounded variation, the signed measure $\nu_R$ is finite, so for all $t\in (t_0,1)$, 
		\begin{align*}
			\bigg| \int_{(t_0,t]} f\, d\nu_R - f(1)(R(t_0^+)-R(t^+))  \bigg|
			= \bigg| \int_{(t_0,t]} (f(t)-f(1)t) \, \nu_R(dt) \bigg|
			\leq \abs{\nu_R}(-1,1) \, \varepsilon.
		\end{align*}
		This argument shows that whenever $(t_n)_n$ is a sequence in $(-1,1)$ that tends to $1$, then $\int_{(0,t_n]}f\, d\nu_R + f(1)R(t_n^+)$ yields a Cauchy sequence. By our assumption on $R$, this implies that the limit $\lim_{t\to 1} \int_{(0,t]} f\, d\nu_R$ exists. The argument for the corresponding limit at $t=-1$ is completely analogous, concluding the proof.		
	\end{proof}
	
	In order to transfer the results for $T_R$ on the full interval $(-1,1)$ to a subinterval $I = (a_-,a_+) \subset (-1,1)$ with $-1 \leq a_- < 0 < a_+ \leq 1$, we define an auxiliary map
	\begin{equation*}
		z:(-1,1)\to I: z(t):=\begin{cases}
			\abs{a_-}t,	&t\leq 0,\\
			a_+ t,		&t\geq 0.
		\end{cases}
	\end{equation*}
	Given a function $f$ on $I$, the composition $f\circ z$ is a dilated copy of $f$ that is now defined on $(1,1)$.
	This allows us to express the restricted map $\restr{T_R}{I}$ in terms of $T_R$ as follows.
	
	\begin{lem}\label{lem:T_R_subinterval_dilation}
		If $R\in \BV_+(I)$, then $R\circ z\in \BV_+(-1,1)$ and for every $f\in C(I)$, we have that
		$T_{R\circ z}(f\circ z)= (\restr{T_R}{I}f)\circ z$. Moreover, $f\in\D_R$ if and only if $f\circ z\in \D_{R\circ z}$.
	\end{lem}

	With these definition in place, it is clear that the operator $\restr{T_R}{I}$ maps the space $\D_R$ into the space $C[a_-,a_+]$. We can extend every $C[a_-,a_+]$ function to a $C[-1,1]$ function by means of the operator $J_I: C[a_-,a_+]\to C[-1,1]$, defined as
	\begin{equation*}
		J_If(t):=
		\begin{cases}
			\frac{f(a_-)}{a_-}t,	&t\leq a_-,\\
			f(t),					&a_-<t<a_+, \\
			\frac{f(a_+)}{a_+}t,	&t\geq a_+.
		\end{cases}
	\end{equation*}

	\begin{proof}[Proof of \cref{lem:T_R_subinterval_dilation}]
		Clearly, $R \circ z$ is again a function of bounded variation. Moreover, a short calculation substituting $u=z(s)$ yields 
		\begin{align*}
			\int_{(-1,1)} \phi'(s) (R \circ z)(s) ds = \int_{(a_{-},a_+)}(\phi \circ z^{-1})'(u)R(u) du = \int_{(a_-,a_+)} (\phi \circ z^{-1})(u) u d\nu_R(u)
		\end{align*}
		for every $\phi \in C^1_c(-1,1)$. Consequently, $R \circ z \in \BV_0(-1,1)$ with
		\begin{align}\label{eq:prfTransfNuRz}
			\nu_{R \circ z} = \left(|a_{-}| \indf_{(-1,0)} + a_+ \indf_{(0,1)}\right)(z^{-1})_\ast \nu_R,
		\end{align}
		which directly implies that $R \circ z \in \BV_+(-1,1)$ as well. Next, \eqref{eq:equivFormTR} implies that for $t > 0$
		\begin{align*}
			T_{R \circ z}(f\circ z)(t) &= (R \circ z)(0) (f\circ z)(t) - \int_{(0,t]} s (f\circ z)(t) - t (f\circ z)(s) d\nu_{R \circ z}(s) \\
			&=R(0)f(z(t)) - \int_{(0,z(t)]} u f(z(t)) - z(t) f(u) d\nu_{R}(u) = (T_R f)(z(t)).
		\end{align*}
		A similar argument for $t < 0$ then yields the second claim. The last claim, finally, follows directly from \eqref{eq:prfTransfNuRz} and the definition of $\D_R$.
	\end{proof}
	
	The proof of \cref{lem:T_R_bijective:I} is now a direct application of \cref{lem:T_R_bijective} and \cref{lem:T_R_subinterval_dilation}.
	\begin{proof}[Proof of \cref{lem:T_R_bijective:I}]
		First, observe that the map $J_I : C[a_-,a_+]\to T_{\indf[I]}(C[-1,1])$ is bijective. Indeed, if $\tilde{g}\in C[a_-,a_+]$, then $\restr{(J_I\tilde{g})}{\cls I}=\tilde{g}$, and conversely, if $g\in T_{\indf[I]}(C[-1,1])$, then $J_I(\restr{g}{\cls I})=T_{\indf[I]}g=g$. Hence, it suffices to show that the map $\restr{T_R}I: \D_R\to C[a_-,a_+]$ is bijective.
		
		\begin{figure*}[h]
			\begin{minipage}{0.29\textwidth}
				\centering
				\begin{tikzcd}
					\D_R \arrow[d,"z^t"',"\cong"] \arrow[r,"\restr{T_R}{I}"] & C[a_-,a_+] \arrow[d,"z^t"',"\cong"] \\
					\D_{R\circ z} \arrow[r,"T_{R\circ z}","\cong"'] & C[-1,1]
				\end{tikzcd}
			\end{minipage}
			\begin{minipage}{0.7\textwidth}
				To this end, note that by \cref{lem:T_R_subinterval_dilation}, the map $z^t:f\mapsto f\circ z$ provides a linear isomorphism from $\D_R$ to $\D_{R\circ z}$ and from $C[a_-,a_+]$ to $C[-1,1]$, respectively, and the diagram on the left hand side commutes. By \cref{lem:T_R_bijective}, the transform $T_{R\circ z}$ is a bijection between $\D_{R\circ z}$ and $C[-1,1]$.
			\end{minipage}
		\end{figure*}
		
		\noindent Thus, the map $\restr{T_R}I: \D_R\to C[a_-,a_+]$ is bijective, which concludes the argument.
	\end{proof}
	
	We further need the following description of the behavior at the endpoints of the interval.
	\begin{proof}[Proof of \cref{lem:D_R_endpoint_limits}]
		We show claim~\cref{it:D_R_endpoint_limits_zero} first for $I = (-1,1)$. By \cref{lem:TRreflect}, we can assume without loss of generality that $a = 1$. Moreover, if $-\nu_R \geq 0$, $R$ is positive and monotonically increasing on $(0,1)$ and so $\lim_{t \to 1}R(t)$ can never be zero. Hence, the assumptions imply that we must have $\nu_R \geq 0$ in \cref{it:D_R_endpoint_limits_zero}.
		
		For \cref{it:D_R_endpoint_limits_zero}, let $0 < t_0 < t < 1$, where we pick $t_0$ to be a continuity point of $R$, that is, $R(t_0^-) = R(t_0^+) = R(t_0)$. Since by assumption $R$ is positive and $\nu_R \geq 0$ on $[0,1)$, we can estimate
		\begin{equation*}
			\int_{[t_0,t)} f\, d\nu_{R}
			= \int_{[t_0,t)} \frac{R(s)f(s)}{s} \frac{s}{R(s)}\, \nu_{R}(ds)
			\geq \frac{1}{t_0}\, \inf_{s \in [t_0,t)} \left\{R(s)f(s)\right\} \, \int_{[t_0,t)}\frac{s}{R(s)}\, \nu_{R}(ds).
		\end{equation*}
		Again, as $\nu_R$ is non-negative and, thus, $R$ is decreasing on $[0,1)$, by \eqref{eq:fundthmCalcBV1},		
		\begin{align*}
			\int_{[t_0,t)}\frac{s}{R(s)}\, \nu_{R}(ds) \geq \frac{1}{R(t_0)} \int_{[t_0,t)}s\, \nu_{R}(ds) = \frac{R(t_0^-) - R(t^-)}{R(t_0)} \geq 1 - \frac{R(t^-)}{R(t_0)},
		\end{align*}
		where, for fixed $t_0$, the last term tends to one as $t$ tends to zero, using that $\lim_{t\to 1} R(t) = 0$. Consequently, we obtain for $t < 1$ large enough
		\begin{align*}
			\inf_{s \in [t_0,1)} \left\{R(s)f(s)\right\} \leq \inf_{s \in [t_0,t)} \left\{R(s)f(s)\right\} \leq 2t_0 \int_{[t_0,t)} f\, d\nu_{R}.
		\end{align*}
		As $f \in \D_R$ and, thus, $|\int_{[0,1)} f\, d\nu_{R}| < \infty$, this implies that
		\begin{align*}
			\liminf_{s\to 1}\left\{R(s)f(s)\right\} \leq 2 \liminf_{s\to 1} \int_{[s,1)} f\, d\nu_{R} = 0.
		\end{align*}
		Since actually the limit of the left-hand side exists, we conclude that $\lim_{s \to 1} R(s)f(s) \leq 0$. Noting that with $f \in \D_R$, we also have $-f \in \D_R$, we can repeat the argument to obtain $\lim_{s \to 1} R(s)f(s) = 0$, concluding the proof of \cref{it:D_R_endpoint_limits_zero} for $I = (-1,1)$.
		
		For general $I \subset (-1,1)$, take $f\in\D_R$. By \cref{lem:T_R_subinterval_dilation}, we have that $R\circ z\in\BV_+(-1,1)$ and $f\circ z\in\D_{R\circ z}$. 
		Since $\lim_{t\to \mathrm{sgn} (a)}(R\circ z)(t)=\lim_{t\to a}R(t)=0$, by the first part,
		\begin{equation*}
			\lim_{t\to a} R(t)f(t)
			= \lim_{t\to \mathrm{sgn}(a) } (R\circ z)(t)(f\circ z)(t)
			= 0,
		\end{equation*}
		which concludes the argument.
		
		For \cref{it:D_R_endpoint_limits_nzero}, finally, we note that as $f \in \D_R$, by definition,
		\begin{align*}
			\lim\limits_{t \to a} R(t) f(t)
		\end{align*}
		exists. As $R(t) > 0$ and $\lim_{t \to a}R(t) > 0$, however, this directly implies that
		\begin{align*}
			\lim\limits_{t \to a} f(t) = \frac{\lim\limits_{t \to a} R(t) f(t)}{\lim\limits_{t \to a} R(t)}
		\end{align*}
		exists, concluding the proof.
	\end{proof}
	
	\begin{proof}[Proof of \cref{lem:T_RInclSegm_bijective:I}]
		First note that $\widehat{T}_R$ clearly is well-defined. Moreover, since by \cref{eq:equivFormTR},
		\begin{align*}
			(T_R f)(0) = R(0) f(0),
		\end{align*}
		by \cref{lem:T_R_bijective:I}, the inverse of $\widehat{T}_R$ is given by
		\begin{align*}
			(\widehat{T}_R)^{-1}g = T_R^{-1}\left(g - c\frac{g(0)}{R(0)} |\cdot|\right) = T_R^{-1}g - c\frac{g(0)}{R(0)^2}|\cdot|.
		\end{align*}
		Here, we used that, by assumption, $R$ is strictly positive on $I$ and that $T_{\indf_I}|\cdot| = |\cdot|$, that is, $|\cdot| \in T_{\indf_I}(C[-1,1])$.
	\end{proof}

	\section*{Acknowledgments}
	
	The first-named author was supported by the Deutsche Forschungsgesellschaft (DFG), project \href{https://gepris.dfg.de/gepris/projekt/520350299}{520350299}. The second-named author was supported by the Austrian Science Fund (FWF), project \href{https://doi.org/10.55776/P34446}{doi:10.55776/P34446}.
	The third-named author was supported by the Austrian Science Fund (FWF), project \href{https://doi.org/10.55776/ESP236}{doi:10.55776/ESP236}, ESPRIT program.

	\begingroup
	\let\itshape\upshape
	
	\bibliographystyle{abbrv}
	\bibliography{references}{}
	
	\endgroup
	
\end{document}